\documentclass[english,smallextended]{svjour3}
\usepackage{mathptmx}
\usepackage[T1]{fontenc}
\usepackage[latin9]{inputenc}
\usepackage{textcomp}
\usepackage{amsmath}
\usepackage{amssymb}
\usepackage{graphicx}
\usepackage{rotfloat}

\makeatletter

\providecommand{\tabularnewline}{\\}

\newcommand{\R}{\mathbb{R}}
\newcommand{\Natu}{\mathbb{N}}

\newcommand{\diag}{\operatorname{diag}}

\newcommand{\dist}{\operatorname{dist}}

\spnewtheorem{experiment}{Experiment}{\it}{\rm}
\newcommand{\myendexp}{\rule{1cm}{0pt}}

\makeatother

\usepackage{babel}

\begin{document}
\title{Towards a reliable implementation of least-squares collocation for higher-index differential-algebraic
equations}
\titlerunning{Implementation of least-squares collocation}
\author{Michael Hanke \and Roswitha März}
\institute{Michael Hanke \at KTH Royal Institute of Technology, Department of Mathematics, S-10044 Stockholm, Sweden \\ \email{hanke@nada.kth.se}
\and Roswitha M\"arz \at Humboldt University of Berlin, Institute of Mathematics, D-10099 Berlin, Germany \\ \email{maerz@mathematik.hu-berlin.de}
}

\date{\today}

\maketitle
\begin{abstract}
In this note we discuss several questions concerning the implementation of overdetermined least-squares
collocation methods for higher-index differential algebraic equations (DAEs). Since higher-index DAEs lead to ill-posed problems in natural settings, the dicrete counterparts are expected to be very sensitive, what attaches particular importance to their implementation. 
We provide a robust selection of basis functions and collocation points 
 to design the discrete problem and substantiate a procedure for its numerical solution.
Additionally, a number of new error estimates are proven that support some of the design decisions.

\end{abstract}

\section{Introduction}

An overdetermined least-squares collocation method for the solution
of boundary-value problems for higher-index differential-algebraic
equations (DAEs) has been introduced in \cite{HMTWW} and further
investigated in \cite{HMT,HM,HM1}. A couple of sufficient convergence conditions
have been established. Numerical experiments indicate an excellent
behavior. Moreover, it is particularly noteworthy that the computational effort is not much more expensive
than for standard collocation methods applied to boundary-value problems
for ordinary differential equations. However, the particular procedures are much more sensitive, which reflects the  ill-posedness of higher-index DAEs. The question of a reliable
implemention is almost completely open. The method offers a number of parameters and options
whose selection has not been backed up by any theoretical justifications.
The present paper is devoted to a first investigation of this topic. We focus on 
the choice of collocation nodes, the representation of the ansatz function as well as the shape and structure of the resulting discrete problem.
We apply various theoretical arguments, among them also new sufficient convergence conditions in Theorems \ref{t.C_precision}, \ref{t.1}, and \ref{t.2}, and report on corresponding systematic comprehensive numerical experiments.

The paper ist organized as follows:
Section \ref{s.Fundamentals}  contains the information concerning the problem to be solved as well as the basics on the overdetermined least-squares approach, and, additionally, the new error estimates.
Section \ref{s.Nodes,weights} deals with the selection and calculation of collocation points and integration weights for the different functionals of interest and Section \ref{s.Basis} provides a robust selection of  basis of the ansatz space. The resulting  discrete least-squares problem is treated in Section \ref{sec:The-discrete-least-squares},  a number of experiments is reported. The more detailed structur of the discrete problem is described in the Appendix. We conclude with Section \ref{s.Final}, which contains a summary and 
further comments.

The algorithms have been implemented in C++11. All computations have
been performed on a laptop running OpenSuSE Linux, release Leap 15.1,
the GNU g++ compiler (version 7.5.0) \cite{GCC}, the Eigen matrix
library (version 3.3.7) \cite{EigenLib}, SuiteSparse (version 5.6.0)
\cite{DavisSS}, in particular its sparse QR factorization \cite{SPQR},
Intel® MKL (version 2019.5-281), all in double precision with a rounding unit of $\epsilon_{\textrm{mach}}\approx 2.22\times10^{-16}$.\footnote{Intel is a registered trademark of Intel Corporation.}
The code is optimized using the level -O3.

\section{Fundamentals of the problem and method}\label{s.Fundamentals}
Consider a linear boundary-value problem for a DAE with properly involved derivative,
\begin{align}
A(t)(Dx)'(t)+B(t)x(t) & =q(t),\quad t\in[a,b],\label{eq:DAE}\\
G_{a}x(a)+G_{b}x(b) & =d.\label{eq:BC}
\end{align}
with $[a,b]\subset\R$ being a compact interval, $D=[I\;0]\in\R^{k\times m}$, $k<m$, 
with the identity matrix $I\in\R^{k\times k}$.
Furthermore, $A(t)\in\R^{m\times k}$, $B(t)\in\R^{m\times m}$, and $q(t)\in\R^{m}$  are assumed to be sufficiently smooth with respect to $t\in[a,b]$. Moreover,
$G_{a},G_{b}\in\R^{l\times m}$. Thereby, $l$ is the dynamical degree of freedom of the DAE, that is, the  number of free parameters which can be fixed by initial and boundary conditions. 
Unlike regular ordinary differential equations (ODEs) where $l=k=m$, for DAEs it holds that $0\leq l\leq k<m$, in particular, $l=k$ for index-one DAEs, $l<k$ for higher-index DAEs, and $l=0$ can certainly  happen.

Supposing accurately stated initial and boundary conditions, index-one DAEs yield well-posed problems in natural settings and can be numerically treated quite well similarly as ODEs \cite{LMW}.
In contrast, in the present paper, we are mainly interested in higher-index DAEs which lead to essentially ill-posed problems even if the boundary conditions are stated accurately \cite{CRR,LMW,HMT}. The tractability index and projector-based analysis serve as the basis for our investigations. We refer to \cite{CRR} for a detailed presentation and to \cite{LMW,Mae2014,HMT} for corresponding short sketches.

We assume that the DAE is regular with arbitrarily high index $\mu\in \Natu$ and the boundary conditions are stated accurately so that solutions of the problem (\ref{eq:DAE})-(\ref{eq:BC}) are unique. We also assume that a solution $x_{\ast}:[a,b]\rightarrow\R^{m}$ actually exists and is sufficiently smooth.
\medskip

For the construction of a regularization method to treat an essentially
ill-posed problem 
 a Hilbert space setting of the problem is
most convenient. For this reason, as in \cite{HMTWW,HMT,HM},  we apply 
the spaces
\begin{align*}
H_{D}^{1} & :=H_{D}^{1}((a,b),\R^{m})=\{x\in L^{2}((a,b),\R^{m}):Dx\in H^{1}((a,b),\R^{k})\},\\
L^{2} & :=L^{2}((a,b),\R^{m}),
\end{align*}
which are suitable for describing the underlying operators. In particular,
let
$
\mathcal{T}:H_{D}^{1}\rightarrow L^{2}\times\R^{l}
$
be given by
\begin{align}
(\mathcal{T}x)(t)= & \left[\begin{array}{c}
A(t)(Dx)'(t)+B(t)x(t)\\
G_{a}x(a)+G_{b}x(b)
\end{array}\right].\label{eq:T}
\end{align}
Then the boundary-value problem can be described by $\mathcal{T}x=(q,d)^{T}$.

 For $K>0$, let $\mathfrak{P}_{K}$
denote the set of all polynomials of degree less than or equal to
$K$.
Next, we define a finite dimensional subspace $X_{\pi}\subset H_{D}^{1}$ 
 of piecewise polynomial functions which should serve as ansatz space for the least-squares approximation: Let the partition $\pi$ be given by
\begin{equation}
\pi: \quad a=t_{0}<t_{1}<\cdots<t_{n}=b,\label{eq:mesh}
\end{equation}
with the stepsizes $h_{j}=t_{j}-t_{j-1}$,  $h=\max_{1\leq j\leq n}h_{j}$, and $h_{min}=\min_{1\leq j\leq n}h_{j}$.

Let $C_{\pi}([a,b],\R^{m})$ denote the space of piecewise continuous
functions having breakpoints merely at the meshpoints of the partition
$\pi$. Let $N\geq1$ be a fixed integer. Then, we define
\begin{align}
X_{\pi} & =\{x\in C_{\pi}([a,b],\R^{m}):Dx\in C([a,b],\R^{k}),\nonumber \\
 & x_{\kappa}\lvert_{[t_{j-1},t_{j})}\in\mathfrak{P}_{N},\,\kappa=1,\ldots,k,\quad x_{\kappa}\lvert_{[t_{j-1},t_{j})}\in\mathfrak{P}_{N-1},\,\kappa=k+1,\ldots,m,\;j=1,\ldots,n\}.\label{eq:Xn}
\end{align}
The continuous version of the least-squares method reads: Find an
$x_{\pi}\in X_{\pi}$ which minimizes the functional
\begin{equation}
\Phi(x)=\|\mathcal{T}x\|^{2}=\int_{a}^{b}|A(t)(Dx)'(t)+B(t)x(t)-q(t)|^{2}{\rm d}t
+|G_ax(a)+G_bx(b)
-d|^{2}.\label{eq:Phi}
\end{equation}
It is ensured by \cite[Theorem 4.1]{HMT} that, for all sufficiently fine partitions $\pi$  with  bounded ratios $1\leq h/h_{min}\leq \rho$, $\rho$ being a global constant, there exists a unique solution $x_\pi\in X_\pi$ and the inequality
\begin{equation}
\|x_{\pi}-x_{\ast}\|_{H_{D}^{1}}\leq Ch^{N-\mu+1}\label{eq:convLS}
\end{equation}
is valid. The constant $C\in \R$ depends on the solution $x_{*}$, the degree $N$, and the index $\mu$, but it is independent of $h$. If $N>\mu-1$ then (\ref{eq:convLS}) apparently  indicates convergence $x_{\pi}\xrightarrow{h\rightarrow 0}x_{*}$ in $H_{D}^{1}$. 

At this place it is important to mention that, so far, we are aware of only sufficient conditions of convergence and the error estimates may not be sharp. Not only more practical questions of implementation are open, but also several questions about the theoretical background. We are optimistic that much better estimates are possible since the results of  the numerical experiments have performed impressively better than theoretically expected till now.

The following theorem  can be understood as a specification of \cite[Theorem 4.1]{HMT} by a more detailed description of the ingredients of the constant C, in particular, now the role of $N$ is better clarified, which could well be important for the practical realization. In particular, it suggests that smooth problems could perhaps be solved better with large N and coarser partitions.
\begin{theorem}\label{t.C_precision}
Let the DAE (\ref{eq:DAE}) be regular with index $\mu\in \Natu$ and let the boundary condition (\ref{eq:BC}) be accurately  stated. Let $x_{*}$ be a solution of the boundary value problem (\ref{eq:DAE})--(\ref{eq:BC}), and let $A,B,q$ and also $x_{*}$ be sufficiently smooth.

Let $N\geq 1$ and all partitions $\pi$  be such that $h/h_{min}\leq \rho$, with a global constant $\rho$. 
Then, for all such partitions with sufficiently small $h$, the estimate \eqref{eq:convLS} is valid with
\[
 C=\frac{N!}{(2N)!\sqrt{2N+1}}C_{N}C_{*}\rho^{\mu-1} C_{data},
\]
where
\begin{align*}
 C_{\ast}&=
\max \{\|x_{\ast}^{(N)}\|_{\infty},\|x_{\ast}^{(N+1)}\|_{\infty}\}(m+4k(b-a)^{3})^{1/2},\\
C_{data}& \quad\text{is independent of $N$ and $h$, it depends only on the data }\; A,D,B, G_{a}, G_{b},
\end{align*}
and
$C_{N}$ is a rather involved function of $N$. In particular,
there is an integer $K$ with $N\leq K\leq 2(\mu-1)+N$ such that,
for $N\rightarrow\infty$, $C_{N}$ does not grow faster than $K^{2(\mu-1)}$. If $A$ and $B$ are constant, it holds $K=N$.
\end{theorem}
At this place it should be mentioned that estimate \cite{Robbins55}
\begin{equation*}
\sqrt{2\pi N}\left(\frac{N}{e}\right)^{N} e^{1/(12N+1)}\leq N!\leq \sqrt{2\pi N}\left(\frac{N}{e}\right)^{N} e^{1/(12N)},
\end{equation*}
or its slightly less sharp version,
\begin{equation*}
 \sqrt{2\pi N}\left(\frac{N}{e}\right)^{N}\leq N!\leq \sqrt{2\pi N}\left(\frac{N}{e}\right)^{N} e^{1/12}
\end{equation*}
allow the growth estimate $\frac{N!}{(2N)!}\leq \sqrt{\pi N}e^{1/6} \frac{1}{N!} \frac{1}{4^{N}}$, thus
\begin{align}
 C&\leq \sqrt{\pi N} e^{1/6} \frac{1}{N!} \frac{1}{4^{N}\sqrt{2N+1}} C_{N}C_{*}\rho^{\mu-1} C_{data}
 \leq \sqrt{\frac{\pi}{2}}\, e^{1/6} \frac{1}{N!} \frac{1}{4^{N}} C_{N}C_{*}\rho^{\mu-1} C_{data} \nonumber\\
  &\leq \frac{1}{N!} \frac{1}{4^{N}} C_{N}C_{*} \rho^{\mu-1} \sqrt{\frac{\pi}{2}} e^{1/6} C_{data}.\label{C}
\end{align}
However, it should be considered that $C_{N}$ and $C_{*}$ also depend on $N$.
\begin{proof}
We apply the estimate \cite{HMT}
\[
\|x_{\pi}-x_{\ast}\|_{H_{D}^{1}}\leq\frac{\|\mathcal{T\|\alpha_{\pi}}}{\gamma_{\pi}}+\alpha_{\pi},
\]
in which the approximation error $\alpha_{\pi}$ and the instability threshold $\gamma_{\pi}$ are  given by
\[
\alpha_{\pi}=\inf_{x\in X_{\pi}}\|x-x_{\ast}\|_{H_{D}^{1}},\quad\gamma_{\pi}=\inf_{x\in X_{\pi},x\neq0}\frac{\|\mathcal{T}x\|}{\|x\|_{H_{D}^{1}}}.
\]
Owing to \cite[Theorem 4.1]{HMT}, there is a constant $c_{\gamma}>0$ independent of $\pi$ so that the instability threshold $\gamma_{\pi}$ satisfies the inequality
\[
c_{\gamma}h_{min}^{\mu-1}\leq \gamma_{\pi}\leq \lVert \mathcal T\rVert,
\]
for all partitions with sufficiently small $h$. This leads to
\[
\|x_{\pi}-x_{\ast}\|_{H_{D}^{1}}\leq \frac{\alpha_{\pi}}{\gamma_{\pi}}(\|\mathcal{T}\|+\gamma_{\pi})\leq 2\frac{\alpha_{\pi}}{\gamma_{\pi}} \|\mathcal{T}\|.
\]
Choosing $N$ interpolation points
$\rho_{i}$ with
\begin{align}
0&<\rho_{1}<\cdots<\rho_{N}<1,\label{eq:IntNodes}\\
\tilde{\omega}(\rho)&=(\rho-\rho_{1})\cdots(\rho-\rho_{N}),\nonumber
\end{align}
the approximation error can be estimated by straightforward but elaborate computations by constructing $p_{*}\in X_{\pi}$ such that $p_{*,s}'(t_{j-1}+\rho_{i}h_{j})=
x_{*,s}'(t_{j-1}+\rho_{i}h_{j})$, $p_{*,s}(a)=x_{*,s}(a)$, $s=1,\ldots,k$,
$p_{*,s}(t_{j-1}+\rho_{i}h_{j})=
x_{*,s}(t_{j-1}+\rho_{i}h_{j})$, $s=k+1,\ldots,m$, $i=1,\ldots,N$, $j=1,\ldots,n$, and regarding $\alpha_{\pi}\leq \|p_{*}-x_{\ast}\|_{H_{D}^{1}}$. One obtains
\begin{align}
 \alpha_{\pi}&\leq \frac{h^{N}}{N!} \lVert\tilde{\omega}\rVert_{L^{2}(0,1)} C_{\ast},\label{eq:alpha} \\
&\quad C_{\ast}=
\max \{\|x_{\ast}^{(N)}\|_{\infty},\|x_{\ast}^{(N+1)}\|_{\infty}\}(m+4k(b-a)^{3})^{1/2}.\nonumber
\end{align}
Turning to shifted Gauss-Legendre nodes that minimize $\lVert\tilde{\omega}\rVert_{L^{2}(0,1)}$ we obtain
\[
 \lVert\tilde{\omega}\rVert_{L^{2}(0,1)}=\frac{(N!)^{2}}{(2N)!\sqrt{2N+1}}.
\]
To verify this, we consider the polynomial 
\begin{align*}
 \omega(t)=2^{N}\tilde\omega\bigg(\frac{t+1}{2}\bigg)=(t-t_{1})\cdots(t-t_{N})
\end{align*}
with zeros $t_{j}=2\rho_{j}-1$, $j=1,\ldots,N$, which is nothing else but the standard Legendre polynomial with leading coefficient one.  Using the Rodrigues formula and other arguments from \cite[Section 5.4]{HaemHoff91}, one obtains
\[
 \lVert\omega\rVert_{L^{2}(-1,1)}= 2^{N+\frac{1}{2}}\frac{(N!)^{2}}{(2N)!\sqrt{2N+1}}.
\]
Finally, shifting back to the interval $(0,1)$ leads to $\lVert\tilde{\omega}\rVert_{L^{2}(0,1)}=2^{-(N+\frac{1}{2})}\lVert\omega\rVert_{L^{2}(-1,1)}$.

Thus we have 
\begin{equation}
 \alpha_{\pi}\leq \frac{h^{N}}{N!} \frac{(N!)^{2}}{(2N)!\sqrt{2N+1}}C_{\ast}= h^{N} \frac{N!}{(2N)!\sqrt{2N+1}}C_{\ast}.\label{eq:alpha1}
\end{equation}
Next, a careful review of the proof of \cite[Theorem 4.1 (a)]{HMT} results in the representation (in terms of \cite{HMT}) 
\begin{align*}
\frac{1}{c_{\gamma}}&= 12 c_{Y}\sqrt{g_{\mu-1}}=12 c_{Y}\sqrt{d_{1,\mu-1}c^{*}_{\mu-1}\lVert D\mathcal L_{\mu-1}\rVert^{2}_{\infty}}\\
&= 
12 c_{Y}\sqrt{2}\lVert D\Pi_{0}Q_{1}\cdots Q_{\mu-1}D^{+}\rVert_{\infty}\lVert D\mathcal L_{\mu-1}\rVert_{\infty}\sqrt{c^{*}_{\mu-1}}.
\end{align*}
The factors $\lVert D\Pi_{0}Q_{1}\cdots Q_{\mu-1}D^{+}\rVert_{\infty}$ and $\lVert D\mathcal L_{\mu-1}\rVert_{\infty}$ depend only on the data $A,D,B$, likewise the bound $c_{Y}$ introduced in \cite[Proposition 4.3]{HMT}.

In contrast,
the term $c^{*}_{\mu-1}$ depends additionally on $N$ besides the problem data. Let $K$ denote the degree
of the auxiliary polynomial $q_{\mu-1}=\mathfrak A_{\mu-1}(Dp)'+\mathfrak B_{\mu-1}p,\; p\in X_{\pi}$ in the proof of \cite[Theorem 4.1]{HMT}. Then we have $N\leq K\leq N+2(\mu-1)$ and, by \cite[Lemma 4.2]{HMT},
 $c_{\mu-1}^{*}=4^{\mu-1}\lambda_{K}\cdots\lambda_{K-\mu+2}$, where each $\lambda_{S}>0$ is the maximal eigenvalue of a certain symmetric, positive semidefinite matrix of size $(S+1)\times(S+1)$ \cite[Lemma 3.3]{HMTWW}.

Owing to \cite[Corollary A.3]{HMTWW} it holds that $\lambda_{S}\leq\frac{4}{\pi^{2}}S^{4}+O(S^{2})$ for large $S$, and therefore 
\begin{align*}
 c_{\mu-1}^{*}&=4^{\mu-1}\lambda_{K}\cdots\lambda_{K-\mu+2} \\
 &\leq 4^{\mu-1} (\frac{4}{\pi^{2}})^{\mu-1}K^{4}(K-1)^{4}\cdots(K-\mu+2)^{4}+O(K^{4(\mu-1)-1})\\
 &=4^{\mu-1} (\frac{4}{\pi^{2}})^{\mu-1}K^{4(\mu-1)}+O(K^{4(\mu-1)-1})\\
 &\leq 4^{\mu-1} (\frac{4}{\pi^{2}})^{\mu-1}(N+2(\mu-1))^{4(\mu-1)}
  +O((N+2(\mu-1))^{4(\mu-1)-1}).
\end{align*}
Finally, letting
\begin{align*}
 C_{data}= 2\|\mathcal{T}\|12 c_{Y}\sqrt{2} \lVert D\Pi_{0}Q_{1}\cdots Q_{\mu-1}D^{+}\rVert_{\infty}\lVert D\mathcal L_{\mu-1}\rVert_{\infty},\quad
 C_{N}=\sqrt{c^{*}_{\mu-1}},
\end{align*}
we are done.
 \qed
\end{proof}
Observe that, for smooth problems, any fixed sufficiently fine partition $\pi$, and $N\rightarrow\infty$, the growth rate of   the error $\lVert x_{\pi}-x_{*}\rVert_{H^{1}_{D}}$ is not greater than that of
\begin{align}\label{Ngrowth}
 C_{*}h^{N}\frac{(N+2(\mu-1))^{2(\mu-1)}}{4^{N} N!}
 =C_{*}\left(\frac{h}{4}\right)^{N}\frac{(N+2(\mu-1))^{2(\mu-1)}}{ N!}
\end{align}
and, for constant matrix function $A$ and $B$,
\begin{align}\label{NgrowthC}
 C_{*}h^{N}\frac{N^{2(\mu-1)}}{4^{N} N!}
 =C_{*}\left(\frac{h}{4}\right)^{N}\frac{N^{2(\mu-1)}}{  N!}.
\end{align}
Remember that $C_{*}$ is a function of $N$.
\begin{remark}\label{r.Jordan}
The specific error estimation provided in \cite{HMTWW} for the case of DAEs in Jordan chain form on equidistant grids 
may provide some further insight in the behavior of the instability threshold $\gamma_{\pi}$. It is shown that
\[
\gamma_{\pi}\geq\bar{C}_{\mu}\left(\frac{h}{\sqrt{\lambda_{N}}}\right)^{\mu-1}
\]
holds true for sufficiently small $h$ where $\bar{C}_{\mu}$ is a
moderate constant depending only on $\mu$ \cite[Theorem 3.6]{HMTWW}.
This leads to the dominant error term 
\begin{align*}
 \frac{\alpha_{\pi}}{\gamma_{\pi}}\leq \frac{C_{\ast}}{\bar C_{\mu}} \sqrt{\frac{\pi}{2}}e^{1/6}\frac{1}{2^{2N}}\frac{\lambda_{N}^{\frac{\mu-1}{2}}}{N!} h^{N-\mu+1}=\frac{1}{\bar C_{\mu}} \sqrt{\frac{\pi}{2}}e^{1/6}\frac{1}{h^{\mu-1}} C_{\ast} \left(\frac{h}{4}\right)^{N}\frac{\lambda_{N}^{\frac{\mu-1}{2}}}{N!},
\end{align*}
indicating again that, for smooth problems it seems reasonable to calculate with larger N and coarse partitions.
Moreover, for sufficiently small $\frac{h}{\sqrt{\lambda_{N}}}$, the estimation  $\lambda_{N}\leq\frac{4}{\pi^{2}}N^{4}+O(N^{2})$ becomes valid 
\cite[Remark 3.4]{HMTWW}, and hence the growth characteristic \eqref{NgrowthC} for large $N$ is confirmed once more.
\qed
\end{remark}
\bigskip

The functional values $\Phi(x)$, which are needed when minimizing for $x\in X_{\pi}$, cannot be evaluated exactly and the integral 
must be discretized accordingly. Taking into account that the boundary-value
problem is ill-posed in the higher index case $\mu>1$, perturbations
of the functional may have a serious influence on the error of the
approximate least-squares solution or even prevent convergence towards
the solution $x_{\ast}$. Therefore, careful approximations of the integral in $\Phi$
are required. We discuss the following three options:
\begin{equation}
\Phi_{\pi,M}^{C}(x)=\sum_{j=1}^{n}\frac{h_{j}}{M}\sum_{i=1}^{M}
|A(t_{ji})(Dx)'(t_{ji})+B(t_{ji})x(t_{ji})-q(t_{ji})|^{2}
+|G_ax(a)+G_bx(b)-d|^{2},\label{eq:PhiC}
\end{equation}
\begin{equation}
\Phi_{\pi,M}^{I}(x)=\sum_{j=1}^{n}h_{j}\sum_{i=1}^{M}\gamma_{i}
|A(t_{ji})(Dx)'(t_{ji})+B(t_{ji})x(t_{ji})-q(t_{ji})|^{2}
+|G_ax(a)+G_bx(b)-d|^{2},\label{eq:PhiI}
\end{equation}
and 
\begin{multline}
\Phi_{\pi,M}^{R}(x)=\sum_{j=1}^{n}\int_{t_{j-1}}^{t_{j}}
\left|\sum_{i=1}^{M}l_{ji}(t)
(A(t_{ji})(Dx)'(t_{ji})+B(t_{ji})x(t_{ji})-q(t_{ji}))\right|^{2}{\rm d}t \\
+|G_ax(a)+G_bx(b)-d|^{2},
\label{eq:PhiR}
\end{multline}
in  which from the DAE (\ref{eq:DAE}) and $x\in X_{\pi}$ only data at the points
\[
t_{ji}=t_{j-1}+\tau_{i}h_{j},\quad i=1,\ldots,M,\;j=1,\ldots,n,
\]
are included, with 
\begin{equation}
0\leq\tau_{1}<\cdots<\tau_{M}\leq 1.\label{eq:nodes}
\end{equation}
In the last functional $\Phi_{\pi,M}^{R}$ Lagrange basis polynomials appear, i.e.,
\begin{equation}
l_{ji}(t)=\frac{\prod_{\substack{\kappa=1\\
\kappa\neq i
}
}^{M}(t-t_{j\kappa})}{\prod_{\substack{\kappa=1\\
\kappa\neq i
}
}^{M}(t_{ji}-t_{j\kappa})}
=\frac{\prod_{\substack{\kappa=1\\
\kappa\neq i
}
}^{M}(\tau-\tau_{\kappa})}{\prod_{\substack{\kappa=1\\
\kappa\neq i
}
}^{M}(\tau_{i}-\tau_{\kappa})}=:l_{i}(\tau),\quad \tau=(t-t_{j-1})/h_{j}.\label{eq:elfa}
\end{equation}
\begin{remark}\label{r.comp}
The direct numerical implementation of $\Phi^{R}_{\pi,M}(x)$ with the Lagrangian
basis functions includes the use of the mass matrix belonging to such
functions. It is well known that this matrix may be very bad conditioned
thus leading to an amplification of rounding errors. In connection
with the ill-posedness of higher-index DAEs, this may render the numerical
solutions useless.
The solution of
the least-squares problem with $\Phi_{\pi,M}^{I}$ is much less expensive
than that with $\Phi^{R}_{\pi,M}$, and in turn,
solving system (\ref{eq:coll})-(\ref{eq:BCcoll}) below for $x\in X_{\pi}$  in a least-squares sense
using the (diagonally weighted) Euclidean norm in $\R^{nMm+l}$ according to $\Phi_{\pi,M}^{C}$ is
even less computationally expensive than using $\Phi^{I}_{\pi,M}(x)$. \qed
\end{remark}
Introducing, for each $x\in X_{\pi}$ and $w(t)=A(t)(Dx)'(t)+B(t)x(t)-q(t)$,
the corresponding vector $W\in\R^{mMn}$ by
\begin{equation}
W=\left[\begin{array}{c}
W_{1}\\
\vdots\\
W_{n}
\end{array}\right]\in\R^{mMn},\quad W_{j}=h_{j}^{1/2}\left[\begin{array}{c}
w(t_{j1})\\
\vdots\\
w(t_{jM})
\end{array}\right]\in\R^{mM},\label{eq:W}
\end{equation}
we obtain new representations of these functionals, namely
\[
\Phi_{\pi,M}^{C}(x)=W^{T}\mathcal{L}^{C}W+|G_ax(a)+G_bx(b)-d|^{2},
\]
\[
\Phi_{\pi,M}^{I}(x)=W^{T}\mathcal{L}^{I}W+|G_ax(a)+G_bx(b)-d|^{2},
\]
and
\[
\Phi_{\pi,M}^{R}(x)=W^{T}\mathcal{L}^{R}W+|G_ax(a)+G_bx(b)-d|^{2},
\]
whereby the first two formulae are evident, with $\mathcal{L}^{C}=\diag(L^{C}\otimes I_{m},\ldots,L^{C}\otimes I_{m})$, $\otimes$ denoting the Kronecker product,
and $L^{C}=M^{-1}I_{M}$ such that finally $\mathcal{L}^{C}=M^{-1}I_{nMm}$, and further, $\mathcal{L}^{I}=\diag(L^{I}\otimes I_{m},\ldots,L^{I}\otimes I_{m})$
and $L^{I}=\diag(\gamma_{1},\ldots,\gamma_{M})$. 
$L^{C}$ and thus $\mathcal{L}^{C}$ are positive definite. The matrices $L^{I}$
and $\mathcal{L}^{I}$ are positive definite if and only if all quadrature
weights are positive.

The formula for $\Phi_{\pi,M}^{R}(x)$ can be established by straightforward evaluations following the lines\footnote{\cite[Section 2.3]{HMTWW} is restricted to equidistant partitions $\pi$ and collocation points $0<\tau_{1}<\cdots<\tau_{M}<1$. The generalization works without further ado.} of \cite[Section 2.3]{HMTWW}, in which $\mathcal{L}^{R}=\diag(L^{R}\otimes I_{m},\ldots,L^{R}\otimes I_{m})$,
 $L^{R}$ is the mass matrix
associated with the Lagrange basis functions $l_{i}$, $i=1,\ldots,M$,
(\ref{eq:elfa}) for the node sequence (\ref{eq:nodes}), more precisely,
\begin{equation}
L^{R}=(L_{i\kappa}^{R})_{i,\kappa=1,\ldots,M},\quad L_{i\kappa}^{R}=\int_{0}^{1}l_{i}(\tau)l_{\kappa}(\tau)d\tau.\label{eq:massmatrix}
\end{equation}
$L^{R}$ is symmetric and positive definite and, consequently,
$\mathcal{L}^{R}$ is so. 

We emphasize that the matrices $L^{C}, L^{I}, L^{R}$ depend only on $M$, the node
sequence (\ref{eq:nodes}), and the quadrature weights, but do not depend on the partition $\pi$ and $h$ at all.
\medskip

We set always $M\geq N+1$. Although the nodes (\ref{eq:nodes}) serve as interpolation points in the functional $\Phi^{R}_{\pi,M}$, we still call them \textit{collocation nodes} after \cite{HMTWW}. 
It should be underlined here that minimizing  each of the above functionals on $X_{\pi}$  can be viewed as a special least-squares method to solve the \textit{overdetermined collocation system} $W=0$, $G_ax(a)+G_bx(b))=d$, with respect to $x\in X_{\pi}$, that is in detail,
 the collocation system
\begin{align}
A(t_{ji})(Dx)'(t_{ji})+B(t_{ji})x(t_{ji}) & =q(t_{ji}),\quad j=1,\ldots,M,\quad i=1,\ldots,n,\label{eq:coll}\\
G_ax(a)+G_bx(b)) & =d.\label{eq:BCcoll}
\end{align}
The system (\ref{eq:coll})-(\ref{eq:BCcoll}) for $x\in X_{\pi}$
becomes overdetermined since  $X_{\pi}$ has dimension $m n N+k$, whereas the system consists of $mnM+l\geq mnN+mn+l\geq nmN+m+l>nmN+k+l\geq mnN+k$ scalar equations.\footnote{If the DAE (\ref{eq:DAE}) has index $\mu=1$, then $l=k$, and hence also the choice $M=N$ makes sense. Then the system (\ref{eq:coll})-(\ref{eq:BCcoll}) for $x\in X_{\pi}$ is nothing else but the classical collocation approach, and  the least squares solution becomes the exact solution
of the collocation system. We refer to \cite{LMW} for a detailed survey of classical collocation methods, however, here we mainly focus on higher index cases yielding overdetermined systems.}
\begin{remark}\label{r.PHI_C}
Based on collocation methods for index-1 DAEs, the first thought in \cite{HMTWW,HMT} was to turn to
the functional $\Phi^{C}_{\pi,M}$ with nodes $0<\tau_{1}<\cdots<\tau_{M}<1$. However, the use of the special discretized norm in these papers for providing convergence results is in essence already the use of the functional  $\Phi^{R}_{\pi,M}$.
 
 For a general
set of nodes (\ref{eq:nodes}), $\Phi^{C}_{\pi,M}$ represents a simple Riemann approximation of the corresponding integral, which has first
order of accuracy, only. If, however, the nodes are chosen as those
of the Chebyshev intergration, the orders $1,\dots,7$ and $9$ can
be obtained for the corresponding number $M$ of nodes \cite[p 349]{Hildebrand56}.
The marking with the upper index $C$ indicates now that Chebyshev integration formulas are conceivable.
As developed in \cite[Section 7.5.2]{HaemHoff91}, integration formulas with uniform weights, i.e., Chebyshev formulas, are those where random errors in the function values have the least effect on the quadrature result. This makes these formulas very interesting in our context. However,
although a lot of test calculations runs well, we are not aware of convergence statements going along  with $\Phi^{C}_{\pi,M}$ so far. \qed
\end{remark}
\begin{remark}\label{r.PHI_R}
The functional $\Phi^{R}_{\pi,M}$ gets its upper index $R$ from the restriction operator  $R_{\pi,M}$ introduced in \cite{HM} with nodes $0<\tau_{1}<\cdots<\tau_{M}<1$. Note that \cite[Theorem 2.3]{HM} generalizes convergence results from \cite{HMTWW,HMT} to a large extend. Theorem \ref{t.1} below allows even any nodes with (\ref{eq:nodes}). \qed
\end{remark}
\begin{remark}\label{r.PHI_I}
The functional $\Phi^{I}_{\pi,M}$ has its upper index $I$ simply from the word integration formula.
We will see first convergence results going along with  $\Phi^{I}_{\pi,M}$ in Theorem \ref{t.1} below.

Intuitively, it seems reasonable to use a Gaussian quadrature
rule for these purposes. However, it is not known if such a rule is
most robust against rounding errors and/or other choices of the overall
process. \qed 
\end{remark}

\begin{remark}\label{r.Unstetigkeit}
Our approximations are according to the basic ansatz space $X_{\pi}$ discontinuous, with possible jumps at the grid points in certain components. In this respect it does not matter which of our functionals is selected. Since we always have overdetermined systems \eqref{eq:coll}-\eqref{eq:BCcoll}, it can no longer be expected that all components of the approximation are continuous even for the case $\tau_{1}=0, \tau_{M}=1$. This is an important difference to the classical collocation methods for Index-1 DAEs, which base on classical uniquely solvable linear systems, e.g.,\cite{LMW}. 
\end{remark}

\begin{theorem}\label{t.1}
Let the DAE (\ref{eq:DAE}) be regular with index $\mu\in \Natu$ and let the boundary condition (\ref{eq:BC}) be accurately  stated. Let $x_{*}$ be a solution of the boundary value problem (\ref{eq:DAE})--(\ref{eq:BC}), and let $A,B,q$ and also $x_{*}$ be sufficiently smooth.

Let all partitions $\pi$  be such that $h/h_{min}\leq \rho$, with a global constant $\rho$.
Then, with 
\[
 M\geq N+\mu,
\]
the following statements are true:
\begin{description}
 \item[\rm (1)] For sufficient fine partitions $\pi$ and each sequence of arbitrarily placed nodes (\ref{eq:nodes}), there exists exactly one $x_{\pi}^{R}\in X_{\pi}$ minimizing the functional $\Phi^{R}_{\pi,M}$ on $X_{\pi}$, and
 \begin{align*}
 \|x_{\pi}^{R}-x_{\ast}\|_{H_{D}^{1}}\leq C_{R}h^{N-\mu+1}. 
 \end{align*}
 \item[\rm (2)] For each integration rule related to the interval $[0,1]$, with $M$ nodes (\ref{eq:nodes}) and positive weights $\gamma_{1},\ldots,\gamma_{M}$, which is exact for polynomials with degree less than or equal to $2M-2$, and sufficient fine partitions $\pi$, there exists exactly one $x_{\pi}^{I}\in X_{\pi}$ minimizing the functional $\Phi^{I}_{\pi,M}$ on $X_{\pi}$, and $x_{\pi}^{I}=x_{\pi}^{R}$, thus
 \begin{align*}
 \|x_{\pi}^{I}-x_{\ast}\|_{H_{D}^{1}}\leq C_{R}h^{N-\mu+1}. 
 \end{align*}
\end{description}
\end{theorem}
 Since Gauss-Legendre and Gauss-Radau integration rules are exact for polynomials up to degree $2M-1$ and $2M-2$, respectively, with positive weights, they are well suitable here, but Gauss-Lobatto rules do not meet the requirement of Theorem~\ref{t.1}(2).
\begin{proof}
(1): In \cite{HM}, additionally supposing $0<\tau<\cdots<\tau_{M}<1$, conditions are derived that ensure the existence and uniqueness of
 $x_{\pi}^{R}$ minimizing $\Phi^{R}_{\pi,M}$ on $X_{\pi}$. It is shown that $x_{\pi}^{R}$ has similar convergence
properties as  $x_{\pi}$ minimizing $\Phi$ on $X_{\pi}$. Merely the constant $C_{R}$ is slightly  larger than $C$ in (\ref{eq:convLS}).
A further careful check of the proofs in \cite{HM} shows that the assertion holds
also true for $\tau_{1}=0$ and/or $\tau_{M}=1$, possibly with a larger constant $C_{R}$.

(2): For each arbitrary $x\in X_{\pi}$, the expression
\begin{align*}
 \theta_{j}(t):=\left|\sum_{i=1}^{M}l_{ji}(t)(A(t_{ji})(Dx)'(t_{ji})+B(t_{ji}x(t_{ji})-q(t_{ji}))\right|^{2},\quad t\in (t_{j-1},t_{j}),
\end{align*}
shows that $\theta_{j}$ is a polynomial with degree less than or equal to $2M-2$, thus
\begin{align*}
 \int_{t_{j-1}}^{t_{j}}\theta_{j}(t){\rm dt}=h_{j}\sum_{i=1}^{M} \gamma_{i}\theta_{j}(t_{ji})=
 h_{j}\sum_{i=1}^{M}\gamma_{i}\left|A(t_{ji})(Dx)'(t_{ji})+B(t_{ji}x(t_{ji})-q(t_{ji})\right|^{2}
\end{align*}
Therefore, it follows that $\Phi_{\pi,M}^{I}(x)=\Phi_{\pi,M}^{R}(x)$, for all $x\in X_{\pi}$, and $\Phi_{\pi,M}^{I}$ coincides with the special functional $\Phi_{\pi,M}^{I}$ having the same nodes. Eventually, (2) is a particular case of (1). \qed
\end{proof}
As already emphasized above, until now we are aware of only sufficient convergence conditions, which is, in particular, especially applicable for the size of $M$. So far, often the applications run well with $M=N+1$ and no significant difference to calculations with a larger M was visible, e.g., \cite[Section 6]{HMT} and \cite[Section 4]{HMTWW}. Also the experiments in Section \ref{s.Basis} below are carried out with $M=N+1$. The following statement for $A$ and $B$ with polynomial entries allows to choose $M$ independently of the index $\mu$ and confirms  the choice $M=N+1$ for constant $A$ and $B$.

\begin{theorem}\label{t.2}
Let the DAE (\ref{eq:DAE}) be regular with index $\mu\in \Natu$ and let the boundary condition (\ref{eq:BC}) be accurately  stated. Let $x_{*}$ be a solution of the boundary value problem (\ref{eq:DAE})--(\ref{eq:BC}), and let $q$ and also $x_{*}$ be sufficiently smooth.
Let the entries of $A$ and $B$ be polynomials with degree less than or equal to $N_{AB}$.
Let all partitions $\pi$  be such that $h/h_{min}\leq \rho$, with a global constant $\rho$.
Then, with 
\[
 M\geq N+1+N_{AB},
\]
the following statements are true:
\begin{description}
 \item[\rm (1)] For sufficient fine partitions $\pi$ and each sequence of arbitrarily placed nodes (\ref{eq:nodes}), there exists exactly one $x_{\pi}^{R}\in X_{\pi}$ minimizing the functional $\Phi^{R}_{\pi,M}$ on $X_{\pi}$, and
 \begin{align*}
 \|x_{\pi}^{R}-x_{\ast}\|_{H_{D}^{1}}\leq C_{R}h^{N-\mu+1}. 
 \end{align*}
 \item[\rm (2)] For each integration rule of interpolation type related to the interval $[0,1]$, with $M$ nodes (\ref{eq:nodes}) and positive weights $\gamma_{1},\ldots,\gamma_{M}$, which is exact for polynomials with degree less than or equal to $2M-2$, and sufficient fine partitions $\pi$, there exists exactly one $x_{\pi}^{I}\in X_{\pi}$ minimizing the functional $\Phi^{I}_{\pi,M}$ on $X_{\pi}$, and $x_{\pi}^{I}=x_{\pi}^{R}$, thus
 \begin{align*}
 \|x_{\pi}^{I}-x_{\ast}\|_{H_{D}^{1}}\leq C_{R}h^{N-\mu+1}. 
 \end{align*}
 \item[\rm (3)] If $A$ and $B$ are even constant matrices, for sufficient fine partitions $\pi$ and each sequence of arbitrarily placed nodes (\ref{eq:nodes}), there exists exactly one $x_{\pi}^{R}\in X_{\pi}$ minimizing the functional $\Phi^{R}_{\pi,M}$ on $X_{\pi}$, and
 \begin{align*}
 \|x_{\pi}^{R}-x_{\ast}\|_{H_{D}^{1}}\leq C_{R}h^{\max\{0,N-\mu+1\}}. 
 \end{align*}
\end{description}
\end{theorem}
\begin{proof}
(1): This follows from  \cite[Proposition 2.2(iv)]{HM} and \cite[Theorem 4.1]{HMT}.
(2): As in the proof of the previous theorem, this is again a comsequence of (1).
 (3): The statement is a consequence of \cite[Proposition 2.2(iv)]{HM} and \cite[Theorem 4.7]{HMT}.\qed
\end{proof}
\begin{remark}\label{r.N1M2}
 Observe a further interesting feature. Let $A$ and $B$ be constant matrices. Set $N=1$, $M=N+1$. Then, it holds that 
 \begin{align*}
 \Phi_{\pi,M}^{C}(x)=\Phi_{\pi,M}^{R}(x)=\Phi_{\pi,M}^{I}(x),\quad x\in X_{\pi},
 \end{align*}
in which $\Phi_{\pi,M}^{I}$ is associated to the corresponding  Gauss-Legendre or Gauss-Radau rules. This follows from the fact that the 2-point Chebyshev integration nodes are just the Gauss-Legendre nodes.

We underline that, by Theorem \ref{t.2}(3), the approximate solutions stay bounded also for DAEs with larger index $\mu$, for instance \cite[Table 6]{HMTWW} confirming that for an index four Jordan DAE.\qed
\end{remark}

Having in mind the implementation of such an overdetermined least-squares collocation, for given partition
$\pi$ and a given polynomial degree $N$, 
 a number of parameters and options must be selected:
\begin{itemize}
\item basis functions for $X_{\pi}$;
\item number $M$ of collocation points and their location $0\leq\tau_1<\cdots<\tau_{M}\leq1$;
\item setup and solution of the discrete least-squares problem.
\end{itemize}
Below we will discuss several issues in this context. The
main aim is on implementations being as stable as possible, not necessarily
best computational efficiency.

\section{Collocation nodes, mass matrix and integration weights}\label{s.Nodes,weights}

\subsection{Collocation nodes for $\Phi^{R}_{\pi,M}$}
The functional $\Phi^{R}_{\pi,M}$ in (\ref{eq:PhiR}) is based on polynomial interpolation using $M$ nodes (\ref{eq:nodes}). It seems reasonable to choose these nodes in such a way that, separately on each subinterval
$[t_{j-1},t_{j}]$ of the partition, the interpolation error is as small as possible in a certain sense. 
Without restriction of the generality we can trace back the matter to the interval $[0,1]$.

Consider functions $q\in C([0,1],\R^{m})$ and define the interpolation operator $R_{M}:C([0,1],\R^{m})\rightarrow C([0,1],\R^{m})$ by
\begin{align*}
 R_{M}q=\sum_{i=1}^{M}l_{i}q(\tau_{i}),
\end{align*}
with the Lagrange basis functions (\ref{eq:elfa}) such that $(R_{M}q)(\tau_{i})=q(\tau_{i})$, $i=1,\ldots,M$, and $R_{M}q\in Y_{M}$, where $Y_{M}\subset C([0,1],\R^{m})$ is the subspace of all functions whose components are polynomials up to degree $M-1$. Introducing $\omega(\tau)=(\tau-\tau_{1})(\tau-\tau_{2})\cdots(\tau-\tau_{M})$ and using componentwise the divided differences
we have the error representation, e.g., \cite[Kapitel~5]{HaemHoff91},
\begin{align*}
 q(\tau)-(R_{M}q)(\tau)=\omega(\tau)\,q[\tau_{1},\ldots,\tau_{M},\tau].
\end{align*}
For smooth functions $q\in C^M([0,1],\R^{m})$ it follows that
\begin{align*}
 \lVert q-R_{M}q\rVert^{2}_{L^{2}}=\int_{0}^{1}\omega(\tau)^{2}\,\lvert q[\tau_{1},\ldots,\tau_{M},\tau]\rvert^{2}{\rm d\tau}\leq \int_{0}^{1}\omega(\tau)^{2}{\rm d\tau}\frac{m}{(M!)^{2}}\lVert q^{(M)}\rVert^2_{\infty}.
\end{align*}
For the
evaluation of $\Phi_{\pi,M}^{R}$ (\ref{eq:PhiR}), it seems resonable
to choose the collocation nodes in such a way that this expression
is minimized for all functions $q\in C^{(M)}([0,1],\R^{m})$. 
The optimal set of nodes is determined by the condition
\[
\min_{0\leq\tau_{1}<\cdots<\tau_{M}\leq1}\|\omega\|_{L^{2}(0,1)}.
\]
It is well known that this functional is minimized if the collocation
nodes are chosen to be the Gauss-Legendre nodes \cite[Chapter 7.5.1 and 4.5.4]{HaemHoff91}.

On the other hand, the best polynomial approximation to a given function $q$ in the
$L^{2}$-norm is obtained if the Fourier approximation with
respect to the Legendre polynomials is computed. However, to the best
of our knowledge, there are no estimations of the interpolation error
in $L^{2}((0,1),\R^{m})$ known.\footnote{It holds that 
$\lVert R_{M}q\rVert_{\infty}\leq \max_{\tau\in[0,1]}\sum_{i=1}^{M} \lvert l_{i}(\tau)\rvert\lvert  q(\tau_{i})\rvert\leq \Lambda_{M} \lVert q\rVert_{\infty}$ which means that the interpolation operator $R_{M}$ is bounded in $C([0,1],\R^{m})$, and the Lebesgue constant is a bound of the operator norm. In contrast, $R_{M}$ it is unbounded in $L^{2}((0,1),\R^{m})$!}
However, in the uniform norm and with arbitrary node sequences, for each $q\in C([0,1],\R^{m})$, 
the estimate
\[
\|R_{M}q-q\|_{\infty}\leq(1+\Lambda_{M})\dist_{\infty}(q,Y_{M})
\]
holds true where $\dist_{\infty}(q,Y_{M})=\min\{\|q-y\|_{\infty}|y\in Y_{M}\}$ and 
$\Lambda_{M}$ is so-called Lebesgue constant defined by
\[
\Lambda_{M}=\max_{\tau\in[0,1]}\sum_{i=1}^{M}|l_{i}(\tau)|
\]
in which $l_{i}$ are again the Lagrange basis functions (\ref{eq:elfa}).

The Lebesgue constant $\Lambda_{M}^{L}$ for Gauss-Legendre nodes
has the property $\Lambda_{M}^{L}=O(\sqrt{M})$. If instead Chebyshev
nodes are used, the corresponding Lebesgue constant $\Lambda_{M}^{C}$
behaves like $\Lambda_{M}^{C}=O(\log M)$ (\cite[p 206]{FoSl94} and
the references therein). 
For uniform polynomial approximations, these nodes are known to be
optimal \cite[Theorem 7.6]{DeuflHohmann}. Table \ref{tab:Lebesgue}
shows some values for the Lebesgue constants. Note that the Lebesgue
constants $\Lambda_{M}^{U}$ for equidistant nodes grow exponentially,
see e.g. \cite{TrefWeid91}.\footnote{For each $M$, there is a set of interpolation nodes $\tau_{i}^{\ast}$
which minimizes the corresponding Lebesgue constant $\Lambda_{M}^{\ast}$.
This constant is only slightly smaller than $\Lambda_{M}^{C}$ \cite{Ibrahimoglu16}.}

\begin{table}
\caption{Lebesgue contants for Chebyshev nodes ($\Lambda_{M}^{C}$), Gauss-Legendre
nodes ($\Lambda_{M}^{L}$), Gauss-Lobatto nodes ($\Lambda_{M}^{Lo}$),
Gauss-Radau nodes ($\Lambda_{M}^{R}$), and uniform nodes including
the boundaries ($\Lambda_{M}^{U}$) and without boundaries ($\Lambda_{M}^{O}$)}

\begin{centering}
\begin{tabular}{|c|c|c|c|c|r|r|}
\hline 
$M$ & $\Lambda_{M}^{C}$ & $\Lambda_{M}^{L}$ & $\Lambda_{M}^{Lo}$ & $\Lambda_{M}^{R}$ & $\Lambda_{M}^{U}$ & $\Lambda_{M}^{O}$\tabularnewline
\hline 
\hline 
5 & 1.989 & 3.322 & 1.636 & 4.035 & 2.708 & 10.375\tabularnewline
\hline 
10 & 2.429 & 5.193 & 2.121 & 6.348 & 17.849 & 204.734\tabularnewline
\hline 
15 & 2.687 & 6.649 & 2.386 & 8.126 & 283.211 & 5107.931\tabularnewline
\hline 
20 & 2.870 & 7.885 & 2.576 & 9.627 & 5889.584 & 138852.138\tabularnewline
\hline 
\end{tabular}\label{tab:Lebesgue}
\par\end{centering}
\end{table}

\begin{remark}
Computation of nodes and weights for for Gauss-type integration formulae

In the following we will make heavy use of Gauss-Legendre, Gauss-Radau,
and Gauss-Lobatto integration nodes and their corresponding weights.
Since we do not have them available in tabular form for large $M$
with sufficient accuracy, they will be computed on the fly. A severe
concern is the accuracy of the nodes and weights. In the case of Gauss-Legendre
integration rules, the computed nodes and weights have been provided by the
Gnu Scientific Library routine \texttt{glfixed.c} \cite{GSL09}. It
makes use of tabulated values for $M=1(1)20$, 32, 64, 96, 100, 128,
256, 512, 1024 with an accuracy of 27 digits. Other values are computed
on the fly with an accuracy being a small multiple of the machine
rounding unit using an adapted version of the Newton method.

For computing the Gauss-Lobatto nodes and weights, the methods of
\cite{Michels63} (using the Newton method) as well as \cite{Gautschi00b}
(a variant of the method in \cite{GolubWelsch69}) have been implemented.
Table \ref{tab:Lobatto} contains some comparisons to the tabulated
values in \cite{Michels63} that have 20 digits. The method of \cite{Michels63}
provides slighly more accurate values than that of \cite{Gautschi00b}.
Therefore, the former has been used further on.

We did not find sufficiently accurate tabulated values for the Gauss-Radau
nodes and weights. Therefore, the method of \cite{Gautschi00a} has
been implemented. We assume that the results obtained have an accuracy
similar to the values for the Gauss-Lobatto nodes and weights using
the method in \cite{Gautschi00b}. \qed
\end{remark}

\begin{table}
\caption{Accuracy of the computed nodes and weights of the Gauss-Lobatto integration
rules. For each method, the absolute error of the nodes (A), the absolute
error of the weights (B), the maximum componentwise relative error
of nodes (C) and weights (D) are shown. The machine accuracy (machine
epsilon) is $2.22\times10^{-16}$}
\label{tab:Lobatto}
\begin{centering}
\begin{tabular}{|c|c|c|c|c|c|c|c|c|}
\hline 
$M$ & \multicolumn{4}{c|}{Method of\cite{Michels63}} & \multicolumn{4}{c|}{Method of \cite{Gautschi00b}}\tabularnewline
\hline 
 & (A) & (B) & (C) & (D) & (A) & (B) & (C) & (D)\tabularnewline
\hline 
\hline 
6 & 1.11e-16 & 1.11e-16 & 4.73e-16 & 5.87e-16 & 1.11e-16 & 3.33e-16 & 4.73e-16 & 9.99e-15\tabularnewline
\hline 
12 & 1.11e-16 & 8.33e-17 & 2.01e-15 & 6.14e-16 & 3.33e-16 & 4.44e-16 & 6.04e-15 & 5.86e-14\tabularnewline
\hline 
24 & 0 & 3.47e-17 & 0 & 9.64e-16 & 2.22e-16 & 2.22e-16 & 8.37e-15 & 1.23e-13\tabularnewline
\hline 
48 & 1.11e-16 & 3.47e-17 & 3.41e-14 & 5.23e-15 & 3.33e-16 & 1.22e-15 & 1.71e-13 & 2.76e-12\tabularnewline
\hline 
96 & 5.55e-16 & 2.95e-17 & 4.73e-16 & 1.76e-14 & 4.44e-16 & 4.44e-16 & 2.76e-13 & 4.05e-12\tabularnewline
\hline 
\end{tabular}
\par\end{centering}
\end{table}

\subsection{The mass matrix}
In the following, we will make extensive use of Legendre polynomials.
For the readers convenience, the necessary properties are collected
in Appendix \ref{subsec:Legendre-polynomials}.
\medskip

Let us turn to $\Phi_{\pi,M}^{R}$ (\ref{eq:PhiR}) again. A critical
ingredient for determining its properties is the mass matrix $L^{R}$
in (\ref{eq:massmatrix}). Denote as before by $l_{i}(\tau)$, $i=1,\ldots,M$,
the Lagrange basis functions for the node sequence (\ref{eq:nodes}),
that is, cf.~(\ref{eq:elfa}),
\[
l_{i}(\tau)=\frac{\prod_{\kappa\neq i}(\tau-\tau_{\kappa})}{\prod_{\kappa\neq i}(\tau_{i}-\tau_{\kappa})}.
\]
For evaluating $L^{R}$, we will use the normalized shifted Legendre
polynomials $\hat{P}_{\nu}=(2\nu+1)^{1/2}\tilde{P}_{\nu}$ (cf Appendix~\ref{subsec:Legendre-polynomials}).
Assume the representation
\begin{equation}
l_{i}(\tau)=\sum_{\nu=1}^{M}\alpha_{i\nu}\hat{P}_{\nu-1}(\tau). \label{eq:LL}
\end{equation}
A short calculation shows
\[
L_{ij}^{R}=\sum_{\lambda=1}^{M}\alpha_{i\lambda}\alpha_{j\lambda}.
\]
Letting $a^{i}=(\alpha_{i1},\ldots,\alpha_{iM})^{T}$ we obtain $L_{ij}^{R}=(a^{i})^{T}a^{j}$.
Collecting the vectors $a^{i}$ in a matrix $A=(a^{1},\ldots,a^{M})$
it holds $L^{R}=A^{T}A$. The definition of the coefficients $\alpha_{i\nu}$
provides us with $\tilde{V}a^{i}=e^{i}$ where $e^{i}$ denotes the $i$-th unit vector and
\begin{equation}
\tilde{V}=\left[\begin{array}{ccc}
\hat{P}_{0}(\tau_{1}) & \ldots & \hat{P}_{M-1}(\tau_{1})\\
\vdots &  & \vdots\\
\hat{P}_{0}(\tau_{M}) & \ldots & \hat{P}_{M-1}(\tau_{M})
\end{array}\right].\label{eq:VDMmatrix}
\end{equation}
This gives $A=\tilde{V}^{-1}$.

$V=\tilde{V}^{T}$ is a so-called Vandermonde-like matrix \cite{Gautschi83}.
It is nonsingular under the condition (\ref{eq:nodes}) \cite[Theorem 3.6.11]{Stoer89}. 
In \cite{Gautschi83}, representations and estimations of the condition
number with respect to the Frobenius norm of such matrices are derived.
In particular, \cite[Table 1]{Gautschi83} shows impressingly small
condition numbers if the collocation nodes are chosen
to be the zeros of $\tilde{P}_{M}$, that is the Gauss-Legendre nodes.
Moreover, this condition number is optimal among all scalings of the
Legendre polynomials \cite{Gautschi83}. A consequence of the Christoffel-Darboux
formula is that the rows of $\tilde{V}$ are orthogonal for Gauss-Legendre
nodes.\footnote{The Christoffel-Darboux formula for Legendre polynomials reads: If
$i\neq\kappa$, then
\[
\sum_{\nu=0}^{M-1}\hat{P}_{\nu}(\tau_{i})\hat{P}_{\nu}(\tau_{\kappa})=\frac{\mu_{M-1}}{\mu_{M}}\frac{\hat{P}_{M}(\tau_{i})\hat{P}_{M-1}(\tau_{\kappa})-\hat{P}_{M}(\tau_{\kappa})\hat{P}_{M-1}(\tau_{i})}{\tau_{i}-\tau_{\kappa}}
\]
where $\mu_{M}$ and $\mu_{M-1}$ are the leading coefficients of
$\hat{P}_{M}$ and $\hat{P}_{M-1}$, respectively. For the Gauss-Legendre
nodes, it holds $\hat{P}_{M}(\tau_{i})=0$. Hence, the right hand side
vanishes.} Thus, we have the representation $\tilde{V}=\mathcal{D}U$ with an
orthogonal matrix $U$ and a diagonal matrix $\mathcal{D}$ with positive
diagonal entries.\footnote{The diagonal elements $d_{i}$, $i=1,\ldots,M$ of $\mathcal{D}$
can be evaluated analytically using the Christoffel-Darboux formula
again:
\[
d_{i}=\sum_{\nu=0}^{M-1}\hat{P}_{\nu}^{2}(\tau_{i})=\frac{\mu_{M-1}}{\mu_{M}}\left(\hat{P}_{M}'(\tau_{i})\hat{P}_{M-1}(\tau_{i})-\hat{P}_{M-1}(\tau_{i})\hat{P}_{M}(\tau_{i})\right)=\frac{\mu_{M-1}}{\mu_{M}}\hat{P}_{M}'(\tau_{i})\hat{P}_{M-1}(\tau_{i}).
\]
}

It is known that the Gauss-Legendre nodes are not the very best set of
nodes. However, a comparison of Tables 1 and 2 in \cite{Gautschi83}
as well as \cite[Table 4]{Gautschi11} indicates that the gain of
choosing optimal nodes for Legendre polynomials compared to the choice
of Gauss-Legendre nodes is rather minor.

In Table \ref{tab:CondV} we provide condition numbers of $\tilde{V}$
with respect to the Euclidean norm for different choices of nodes.
Note that the condition number of $L^{R}$ is the square of that of
$\tilde{V}$.

The condition numbers for all Gauss-type and Chebyshev nodes are remarkably small.

\begin{table}

\caption{Spectral condition numbers of the Vandermonde-like matrices for different
node choices. The columns represent Gauss-Legendre nodes (GLe), Gauss-Radau
nodes (GR), Gauss-Lobatto nodes (GLo), Chebyshev nodes (Ch), Newton-Cotes
type nodes including the boundary (cNC) and without the boundary (oNC).
An asterisk {*} indicates an overflow condition}

\centering{}%
\begin{tabular}{c|cccccc}
$M$ & GLe & GR & GLo & Ch & cNC & oNC\tabularnewline
\hline 
5 & 1.55e+0 & 2.79e+0 & 3.23e+0 & 2.16e+0 & 3.76e+0 & 3.04e+0\tabularnewline
10 & 2.11e+0 & 3.96e+0 & 4.28e+0 & 3.00e+0 & 2.39e+1 & 5.23e+1\tabularnewline
15 & 2.57e+0 & 4.85e+0 & 5.11e+0 & 3.66e+0 & 3.98e+2 & 1.14e+3\tabularnewline
20 & 2.94e+0 & 5.60e+0 & 5.83e+0 & 4.21e+0 & 8.62e+3 & 3.10e+4\tabularnewline
50 & 4.62e+0 & 8.86e+0 & 9.00e+0 & 6.60e+0 & 1.13e+12 & 1.97e+13\tabularnewline
100 & 6.52e+0 & 1.25e+1 & 1.26e+1 & 9.32e+0 & {*} & {*}\tabularnewline
\end{tabular}\label{tab:CondV}
\end{table}

\subsection{Computation of quadrature weights for general $\Phi^{I}_{\pi,M}$}

In oder to apply $\Phi_{\pi,M}^{I}$(\ref{eq:PhiI}), a numerical
quadrature formulae is necessary. For standard
nodes sequences (Gauss-Legendre, Gauss-Lobatto, Gauss-Radau) their
computation has been described above. However, for general node sequences,
the weights must be evaluated. This can be done following the derivations
in \cite[p. 175]{Stoer89}: Let $\hat{P}_{\nu}(\tau)$ denote the
normalized shifted Legendre polynomials as before (cf. Appendix \ref{subsec:Legendre-polynomials}). In
particular, it holds then
\[
\int_{0}^{1}\hat{P}_{0}(\tau){\rm d}\tau=1,\quad\int_{0}^{1}\hat{P}_{\nu}(\tau){\rm d}\tau=0,\quad\nu=1,2,\ldots
\]
For a given function $q\in C[0,1]$,
the integral is approximated by the integral of its polynomial interpolation.
Using the representation (\ref{eq:LL}) of the Lagrange basis functions we obtain
\begin{align*}
\int_{0}^{1}q(\tau)d\tau & \approx\int_{0}^{1}\sum_{i=1}^{M}q(\tau_{i})\sum_{\nu=1}^{M}\alpha_{i\nu}\hat{P}_{\nu-1}(\tau){\rm d}\tau\\
 & =\sum_{i=1}^{M}q(\tau_{i})\sum_{\nu=1}^{M}\alpha_{i\nu}\int_{0}^{1}\hat{P}_{\nu-1}(\tau){\rm d}\tau\\
 & =\sum_{i=1}^{M}q(\tau_{i})\alpha_{i1}.
\end{align*}
Consequently, for the weights it holds $\gamma_{i}=\alpha_{i1}$,
$i=1,\ldots,M$. The definition (\ref{eq:LL}) shows that the vector
$\gamma=(\gamma_{1},\ldots,\gamma_{M})^{T}$ of weights fulfills the
linear system
\[
V\gamma=e^{1}
\]
where $V=\tilde{V}^{T}$with $\tilde{V}$ from (\ref{eq:VDMmatrix})
and $e^{1}=(1,0,\ldots,0)^T$ is the first unit vector.

The discussion of the condition number of $V$ shows that we can expect
reliable and accurate results at least for reasonable node sequences.

For general node sequences, the weights may become negative. This
happens, for example, for uniformly distributed nodes and $M>7$ (Newton-Cotes
formulae) \cite[p. 148]{Stoer89}. So for $\Phi_{\pi,M}^{I}$, only
node sequences leading to positive quadrature weights $\gamma_{i}$
are admitted in order to prevent $L^{I}$ from not being positive definite.

\section{Choice of basis functions for the ansatz space $X_{\pi}$}\label{s.Basis}

The ansatz space $X_{\pi}$ (\ref{eq:Xn}) consists of piecewise polynomials
having the degree $N-1$ for the algebraic components and the degree
$N$ for the differentiated ones on each subinterval of the partition $\pi$ (\ref{eq:mesh}).
For collocation methods for boundary value problems for ordinary differential
equations this questions has lead to the choice of a Runge-Kutta basis
for stability reasons, see \cite{BaderAscher}. This has been later
on also used successfully for boundary value problems for index-1
DAEs \cite{AscherSpiteri,KKPSW,LMW,KMPW10}. However, this ansatz
makes heavily use of the collocation nodes which are at the same time used as
the nodes for the Runge-Kutta basis. In our case, the number $M$
of collocation nodes and the degree $N$ of the polynomials for the
differentiated components do not coincide since $M>N$ such that the
reasoning applied in the case of ordinary differential equations does not 
transfer to the least-squares case.

Taking into account the computational expense for solving the discretized
system, bases with local support are preferable. Ideally, the support
of each basis function consists of only one subinterval of (\ref{eq:mesh}).\footnote{This excludes for example B-spline bases!} Note that the Runge-Kutta basis has this property. We consider the Runge-Kutta basis and further local basis with orthogonal polynomials.
A drawback of
this strategy is the fact that the continuity of the piecewise polynomials
approximating the differentiated components must be ensured explicitly.
This in turn will lead to a discrete least-squares problem with equality
constraints. Details can be found in Appendix~\ref{subsec:The-structure-of}.

Looking for a local basis we turn to the reference interval $[0,1]$.
 Once a basis on this
reference interval is available it can be defined on any subinterval
$(t_{j-1},t_{j})$ by a simple linear transformation.

Assume that $\{p_{0},\ldots,p_{N-1}\}$ is a basis of the set of polynomials
of degree less than $N$ defined on the reference interval $[0,1]$.
Then, a basis $\{\bar{p}_{0},\ldots,\bar{p}_{N}\}$ for the ansatz
functions for the differentiated components is given by
\begin{equation}
\bar{p}_{i}(\rho)=\begin{cases}
1, & i=0,\\
\int_{0}^{\rho}p_{i-1}(\sigma){\rm d}\sigma, & i=1,\ldots,N,\quad\rho\in[0,1],
\end{cases}\label{eq:diffBasis}
\end{equation}
and the transformation to the interval $(t_{j-1},t_{j})$ of the partition
$\pi$ (\ref{eq:mesh}) yields
\begin{align}
p_{ji}(t) & =p_{i}((t-t_{j-1})/h_{j}),\nonumber \\
\bar{p}_{ji}(t) & =h_{j}\bar{p}_{i}((t-t_{j-1})/h_{j}).\label{eq:scaledp}
\end{align}
Additionally to this transformation, the continuity of the piecewise
polynomials must be ensured. This gives rise to the additional conditions
\begin{equation}
\bar{p}_{ji}(t_{j})=\bar{p}_{j+1,i}(t_{j}),\quad  i=1,\ldots,N,\quad  j=1,\ldots,n-1,
\label{eq:CC}
\end{equation}
which must be imposed explicitely.\footnote{This is in contrast to choices of basis functions that fulfil the
basis conditions. An example of such basis functions are B-splines.}

\subsection{The Runge-Kutta basis}

In order to define the Runge-Kutta basis, let the $N$ interpolation points
$\rho_{i}$ with (\ref{eq:IntNodes})
be given. Then, the Lagrange basis functions are chosen,
\[
p_{i}(\rho)=\frac{\prod_{\kappa\neq i+1}(\rho-\rho_{\kappa})}{\prod_{\kappa\neq i+1}(\rho_{i}-\rho_{\kappa})},\quad i=0,\ldots,N-1.
\]
\begin{remark}\label{r.interpolationnodes}
Note that the interpolation nodes are only used to define the local
basis functions. Thus, their selection is completely independent of
the choice of collocation nodes. In view of the estimations \eqref{eq:alpha} and \eqref{eq:alpha1} and the argumentation there we prefer Gauss-Legendre
interpolation nodes. This choice is also supported by Experiments~\ref{exp:2} and \ref{exp:5} below.
\end{remark}
The numerical computation of $\bar{p}$ is more involved. If not precalculated,
the integrals must be available in a closed formula. This can surely
be done by expressing the Lagrange basis functions in the monomial representation
such that the integration can be carried out analytically. Once these
coefficients are known, the evaluation of the values of the basis
functions at a given $\rho\in[0,1]$ is easily done using the Horner
method. However, this approach amounts to the inversion of the Vandermonde
matrix using the nodes (\ref{eq:IntNodes}). This matrix is known
to be extremly ill-conditioned. In particular, its condition number
grows exponentially with $N$ \cite{GaIn88,Beckermann00}.
Therefore, an orthogonal basis might be better suited. This leads
to a representation
\begin{equation}
p_{i}(\rho)=\sum_{\kappa=1}^{N}\alpha_{i\kappa}Q_{\kappa}(\rho)\label{eq:LagI-1}
\end{equation}
for some polynomials $Q_{1},\ldots,Q_{N}$. If these polynomials fulfil
a three-term recursion,\footnote{which they do if the polynomials are orthogonal with respect to some
scalar product \cite[Theorem 6.2]{DeuflHohmann}.} the evaluation of function values can be performed using the Clenshaw
algorithm \cite{FoxParker68} which is only slightly more expensive
than the Horner method. In order to use this approach, the integrals
of $p_{0},\ldots,p_{N-1}$ must be easily representable in terms of
the chosen basis. Here, the Legendre and Chebyshev polynomials are
well-suited (cf below Appendix~\ref{subsec:Legendre-polynomials}
and (\ref{eq:Int_Legendre}) as well as Appendix~\ref{subsec:Chebyshev-polynomials}
and (\ref{eq:Int_Chebyshev})).

\subsection{Orthogonal polynomials}

A reasonable choice for the basis are orthogonal polynomials. We will
consider Legendre polynomials first. A motivation is provided in the
following example.
\begin{example}
Consider the index-1 DAE
\[
x=q(t),\quad t\in[0,1].
\]
Let $\{\hat{P}_{0},\ldots,\hat{P}_{N-1}\}$ be the normalized shifted
Legendre polynomials. Then letting $x=\sum_{i=1}^{N}\alpha_{i}\hat{P}_{i-1}$ for
some vector $\alpha=(\alpha_{1},\ldots,\alpha_{N})^{T}$, the least-squares
functional
\[
\Phi(x)=\int_{0}^{1}(x(t)-q(t))^{2}{\rm d}t
\]
corresponding to this DAE is minimized if $\alpha=b$ and $b=(b_{1},\ldots,b_{N})^{T}$
where $b_{i}=\int_{0}^{1}q(t)\hat{P}_{i-1}(t)dt$ which is just the
best approximation of the solution in $H_{D}^{1}((0,1),\R)=L^{2}((0,1),\R)$.

Similar relations hold for the differential equation $x'=f$ if the
basis functions for the differentiated components are constructed
according to (\ref{eq:diffBasis}). Hence, these basis functions seem
to qualify well for index-1 DAEs.\hfill\qed
\end{example}
The necessary ingredients for the efficient implementation of the
Legendre polynomials are collected in Appendix~\ref{subsec:Legendre-polynomials}.

Another common choice are Chebyshev polynomials of the first kind.
They have been used extensivly in the context of spectral methods
because of their excellent approximation properties, cf \cite{Fornberg96,Trefethen00},
see also \cite{DrHatr14}. The relations used for their implementation
can be found in Appendix~\ref{subsec:Chebyshev-polynomials}.\footnote{Let us note in passing that the first routine for solving two-point
boundary problems in the NAG library (NAG® is a registered trademark
of The Numerical Algorithms Group) besides shooting methods was just
a least-squares collocation method corresponding to $n=1$ and using
a version of the functional $\Phi_{\pi,M}^{C}$. The NAG routine D02AFF
and its predecessor D02TGF (and its driver D02JBF) use Chebyshev polynomials
as basis functions and Gauss-Legendre nodes as collocation points
\cite{Gladwell79,Albasiny78}. This routine appeared as early as 1970
in Mark 8 of the library and survived to date (as of Mark 27 of 2019)
\cite{NAG}.}

\subsection{Comparison of different basis representations}

The choice of the basis function representations is dominated by the
question of obtaining a most robust implementation. The computational
complexity of the representations presented above is not that much
different such that this aspect plays a minor role.

The check for robustness can be subdivided into two questions:
\begin{enumerate}
\item Which representation is most robust locally?
\item Which representation is most robust globally?
\end{enumerate}

In the following experiments, $N$ will be varied. The functional
used is $\Phi_{\pi,M}^{R}$. The number of collocation nodes is $M=N+1$.
Table~\ref{tab:CondV} motivates the choice of the Gauss-Legendre
nodes as collocation nodes. In order to compute the norms of $L^{2}((0,1),\R^{m})$
and $H_{D}^{1}((0,1),\R^{m})$, Gaussian quadrature with $N+2$ integration
nodes on each subinterval of $\pi$ is used.

\subsubsection{Local behavior of the basis representations}

In order to answer the first question, is is reasonable to experiment
first with a higher index example that does not have any dynamic components
(that is, $l=0$) on a grid $\pi$ consisting only of one subinterval
(that is, $n=1$). In that case, we check the ability to interpolate
functions and to numerically differentiate them.

For $n=1$, there are no continuity conditions (\ref{eq:CC}) involved. Therefore, the discrete problem becomes a linear least-squares problem. We will solve it by a Householder
QR-factorization with column pivoting as implemented in the Eigen
library.

The following example is used in \cite{HMTWW,HMT}.

\begin{example}
\label{exa:DAE1}	
\begin{align*} 		x'_{2}(t)+x_{1}(t) & =q_{1}(t),\\ 		t\eta x'_{2}(t)+x'_{3}(t)+(\eta+1)x_{2}(t) & =q_{2}(t),\\ 		t\eta x_{2}(t)+x_{3}(t) & =q_{3}(t),\quad t\in [0,1]. 	\end{align*} 
This
is an index-3 example with dynamical degree of freedom $l=0$ such that no additional boundary or initial
conditions are necessary for unique solvability. We choose the exact
solution 	\begin{align*} 	x_{\ast,1}(t) & =e^{-t}\sin t,\\ 	x_{\ast,2}(t) & =e^{-2t}\sin t,\\ 	x_{\ast,3}(t) & =e^{-t}\cos t 	\end{align*} and
adapt the right-hand side $q$ accordingly. For the exact solution,
it holds $\|x_{\ast}\|_{L^{2}((0,1),\R^{3})}\approx0.673$, $\|x_{\ast}\|_{L^{\infty}((0,1),\R^{3})}=1$,
and $\|x_{\ast}\|_{H_{D}^{1}((0,1),\R^{3})}\approx1.11$.\qed
\end{example}

\begin{experiment}
\label{exp:1} Robustness of the representation of the Runge-Kutta
basis

In a first experiment we intend to clarify the differences between different
representations of the Runge-Kutta basis. The interpolation nodes
(\ref{eq:IntNodes})
have been fixed to be the Gauss-Legendre nodes (cf (\ref{eq:alpha})). The Runge-Kutta basis has
been represented with respect to the monomial, Legendre, and Chebyshev
bases. The results are shown in Figure~\ref{fig:Ex1}
(see appendix). This test indicates that the monomial basis is much
less robust than the others for $N>10$ while the other representations
behave very similar. \hfill\qed
\end{experiment}
\myendexp
\begin{experiment}
\label{exp:2} Robustness of the Runge-Kutta basis with respect to
the node sequence

In this experiment we are interested in understanding the influence
of the interpolation nodes. For that, we compared the uniform nodes
sequence to the Gauss-Legendre and Chebyshev nodes. The uniform nodes
are given by $\rho_{i}=(i-\frac{1}{2})/N$. In accordance with the
results of the previous experiment, the representation of the Runge-Kutta
basis in Legendre polynomials has been chosen. The results are shown
in Figure~\ref{fig:Ex2}. Not unexpectedly,
uniform nodes are inferior to the other choices at least for $N>13$.
On the other hand, there is no significant difference between Gauss-Legendre
and Chebyshev nodes. \hfill\qed
\end{experiment}
\myendexp
\begin{experiment}
\label{exp:3}Robustness of different polynomial representations

In this experiment we intend to compare the robustness of different
bases. Therefore, we have chosen the Runge-Kutta basis with Gauss-Legendre
interpolation nodes, the Legendre polynomials, and the Chebyshev polynomials.
The results are shown in Figure~\ref{fig:Ex3}.
All representations show similar behavior. \hfill\qed
\end{experiment}
A general note is in order. The exact solution has approximately the
norm 1 in all used norms. The machine accuracy is $\varepsilon_{\textrm{mach}}\approx2.22\times10^{-16}$
in all computations. The best accuracy obtained is $10^{-12}$ --
$10^{-14}$. Considering that there is a twofold differentiation involved in the problem of the example 
we would expect a much lower accuracy. This surprising behavior has also been observed in other experiments.

The next example is an index-3 one which has $l=4$ dynamical degrees
of freedom. It is the linearized version of an example presented \cite{CampbellMoore95}
that has also been considered in \cite{HMT}.
\begin{example}
\label{exa:LinCaMo}Consider the DAE
\[
A(Dx(t))'+B(t)x(t)=q(t),\quad t\in[0,5],
\]
where\begin{align*}  A=\begin{bmatrix}     1&0&0&0&0&0\\     0&1&0&0&0&0\\     0&0&1&0&0&0\\     0&0&0&1&0&0\\     0&0&0&0&1&0\\     0&0&0&0&0&1\\     0&0&0&0&0&0    \end{bmatrix},    D=\begin{bmatrix}     1&0&0&0&0&0&0\\     0&1&0&0&0&0&0\\     0&0&1&0&0&0&0\\     0&0&0&1&0&0&0\\     0&0&0&0&1&0&0\\     0&0&0&0&0&1&0    \end{bmatrix}, \end{align*}and
the smooth coefficient matrix\begin{align*}  B(t)=  \begin{bmatrix}            0&0&0&-1&0&0&0\\            0&0&0&0&-1&0&0\\            0&0&0&0&0&-1&0\\            0&0&\sin t&0&1&-\cos t&-2\rho \cos^{2}t\\            0&0&-\cos t&-1&0&-\sin t&-2\rho \sin t\cos t\\            0&0&1&0&0&0&2\rho \sin t\\            2\rho \cos^{2}t&2\rho \sin t \cos t&-2\rho\sin t&0&0&0&0           \end{bmatrix},\quad \rho=5, \end{align*} subject
to the initial conditions
\[
x_{2}(0)=1,\quad x_{3}(0)=2,\quad x_{5}(0)=0,\quad x_{6}(0)=0.
\]
This problem has the tractability index 3 and dynamical dgree of freedom $l=4$. The right-hand side $q$
has been chosen in such a way that the exact solution becomes
\begin{alignat*}{2}  x_{\ast,1} &= \sin t, & x_{\ast,4} &= \cos t, \\
x_{\ast,2} &= \cos t, & x_{\ast,5} &= -\sin t, \\
x_{\ast,3} &= 2\cos^2 t, & x_{\ast,6} &= -2\sin 2t, \\
x_{\ast,7} &= -\rho^{-1}\sin t. & & \end{alignat*} For
the exact solution, it holds $\|x_{\ast}\|_{L^{2}((0,5),\R^{7})}\approx5.2$,
$\|x_{\ast}\|_{L^{\infty}((0,5),\R^{7})}=2$, and $\|x_{\ast}\|_{H_{D}^{1}((0,5),\R^{7})}\approx9.4$.\hfill\qed
\end{example}
The following experiments with Example~\ref{exa:LinCaMo} are carried
out under the same conditions as before when using Example~\ref{exa:DAE1}.
\begin{experiment}
\label{exp:4}Robustness of the representation of the Runge-Kutta
basis

In this experiment we intend to clarify the differences between different
representations of the Runge-Kutta basis. The interpolation points
have been fixed to be the Gauss-Legendre nodes. The Runge-Kutta basis has
been represented with respect to the monomial, Legendre, and Chebyshev
bases. The results are shown in Figure~\ref{fig:Ex4}.
This test indicates that the monomial basis is much less robust than
the others for $N>15$ while the other representations behave very
similar. \hfill\qed
\end{experiment}
\myendexp
\begin{experiment}
\label{exp:5}Robustness of the Runga-Kutta basis with respect to
the node sequence

In this experiment we are interested in understanding the influence
of the interpolation nodes. For that, we compared the uniform nodes
sequence to the Gauss-Legendre and Chebyshev nodes. The uniform nodes
are given by $\rho_{i}=(i-\frac{1}{2})/N$. In accordance with the
results of the previous experiment, the representation of the Runge-Kutta
basis in Legendre polynomials has been chosen. The results are shown
in Figure~\ref{fig:Ex5}. Not unexpectedly,
uniform nodes are inferior to the other choices at least for $N>20$.
However, there is no real difference between Gauss-Legendre and Chebyshev
nodes. \hfill\qed
\end{experiment}
\myendexp
\begin{experiment}
\label{exp:6}Robustness of different polynomial representations

In this experiment we intend to compare the robustness of different
bases. Therefore, we have chosen the Runge-Kutta basis with Gauss-Legendre
interpolation nodes, the Legendre polynomials, and the Chebyshev polynomials.
The results are shown in Figure~\ref{fig:Ex6}.
All representations show similar behavior. \myendexp
\end{experiment}
As a conclusion, we can see that the results of the Experiments \ref{exp:1}-\ref{exp:3}
and \ref{exp:4}-\ref{exp:6} are largely consistent.

\subsubsection{Global behavior of the basis representations}

We are interested in understandig the global error, which corresponds
to error propagation in the case of initial value problems. In order
to understand the error propagation properties we will investigate
the accuracy of the computed solution with respect to an increasing
number of subintervals $n$. This motivates to use a rather low order
$N$ of polynomials. In the previous section we observed that there
is no difference in the local properties between different basis representations
for low degrees $N$ of the ansatz polynomials.

In the following experiments, the functionals used are $\Phi_{\pi,M}^{R}$
and $\Phi_{\pi,M}^{C}$. The number of collocation nodes is again $M=N+1$.
The basis functions are the shifted Legendre polynomials. 

The discrete problem for $n>1$ is an equality constraint
linear least-squares problem. The equality constraints consists just
of the continuity requirements for the differentiated components of
the elements in $X_{\pi}$. The problem is solved by a direct solution method as described in Section~\ref{sec:The-discrete-least-squares}. In short, the equality constraints are eliminated by a sparse QR-decomposition
with column pivoting as implemented in the code SPQR \cite{SPQR}.
The resulting least-squares problem has then been solved by the same code.

\begin{experiment}
\label{exp:7}Influence of selection of collocation nodes, approximation
degree $N$, and number $n$ of subintervals

In this experiment, we use Example \ref{exa:LinCaMo} and vary the
choice of collocation nodes as well as the degree $N$ of the polynomial
basis and the number $n$ of subintervals. We compare Gauss-Legendre, Radau IIA and Lobatto collocation nodes.
Since this example is a
pure initial value problem, the use of the Radau IIA collocation nodes
is especially justified.\footnote{Such methods are proven in time-stepping procedures for ordinary 
initial value problems because of their stability properties. Radau IIA methods are also used for many DAEs with index $\mu\leq2$ since the generated approximations on the grid points satisfy the obvious constraint.} 
The results using $\Phi_{\pi,M}^{R}$ are collected in Table~\ref{tab:LinCaMo},
those using $\Phi_{\pi,M}^{C}$ in Table~\ref{tab:LinCaMo_C}. We
observe no real difference between the different sets of collocation points.
The results seem to confirm the conjecture that, in case of smooth problems, a higher degree $N$ is preferable
over a larger $n$ or, equivalently, a smaller stepsize $h$. In addition, for the highest degree polynomials ($N=20$), the use of
$\Phi_{\pi,M}^{C}$ seems to produce more accurate results than that
of $\Phi_{\pi,M}^{R}$. \hfill\qed

\end{experiment}

\section{The discrete least-squares problem}\label{sec:The-discrete-least-squares}
Once the basis has been chosen and the collocation conditions are
selected, the discrete problems (\ref{eq:PhiR}), (\ref{eq:PhiI}),
and (\ref{eq:PhiC}) for a linear boundary value problem (\ref{eq:DAE}) - (\ref{eq:BC})
lead to a constraint linear least-squares problem
\begin{equation}
\varphi(c)=\|\mathcal{A}c-r\|_{\R^{nmM+l}}^{2}\rightarrow\min!\label{eq:discmin}
\end{equation}
under the constraint
\begin{equation}
\mathcal{C}c=0.\label{eq:discconstraint}
\end{equation}
The equality constraints consists of the $k(n-1)$ continuity conditions for
the approximation of the differential constraints while the functional
$\varphi(c)$ represent a reformulation of the functionals (\ref{eq:PhiR}),
(\ref{eq:PhiI}), and (\ref{eq:PhiC}), respectively. Here, $c\in\R^{n(mN+k)}$ is
the vector of coefficients of the basis functions for $X_{\pi}$ disregarding the continuity conditions. 
Furthermore, it holds $r\in\R^{nM}$,
$\mathcal{A}\in\R^{(nM+l)\times n(mN+k)}$, and $\mathcal{C}\in\R^{(n-1)k\times n(mN+k)}$.
The matrices $\mathcal{A}$ and $\mathcal{C}$ are very sparse. For
details about their structure we refer to Appendix \ref{subsec:The-structure-of}.

\subsection{Approaches to solve the constraint optimization problem (\ref{eq:discmin})-(\ref{eq:discconstraint})}
A number of approaches to solve the constraint optimization problem (\ref{eq:discmin})-(\ref{eq:discconstraint})
have been tested.
\begin{enumerate}
\item Direct method. The solution manifold of (\ref{eq:discconstraint})
forms a subspace which can be characterized by\footnote{$\mathcal{C}$ has full row rank.}
\[
\mathcal{C}c=0\text{ if and only if }c=\tilde{\mathcal{C}}z\text{ for some }z\in\R^{nmN+k-l}.
\]
Here, $\tilde{\mathcal{C}}\in\R^{n(mN+k)\times(nmN+k-l)}$ has orthonormal
columns. With this representation, the constrained minimization problem
can be reduced to the uncontrained one 
\[
\tilde{\varphi}(z)=\|\mathcal{A}\tilde{\mathcal{C}}z-r\|_{\R^{W}}^{2}\rightarrow\min!
\]
The implemented algorithm is that of \cite{BjorckGolub67}, see also
\cite[Section 5.1.2]{Bjorck96} which is sometimes called the direct
elimination method.
\item Weighting of the constraints.
In this approach, a sufficiently large
parameter $\omega>0$ is chosen and the problem (\ref{eq:discmin})
- (\ref{eq:discconstraint}) is replaced by the free minimization
problem
\[
\varphi_{\omega}(c)=\|\mathcal{A}c-r\|_{\R^{W}}^{2}+\omega\|\mathcal{C}c\|\rightarrow\min!
\]
It is known that the\footnote{Assuming a fullrank condition on $\mathcal{A}$!}
minimizer $c_{\omega}$ of $\varphi_{\omega}$ converges towards the
solution of (\ref{eq:discmin}) - (\ref{eq:discconstraint}) for $\omega\rightarrow\infty$
(cf. \cite[Section 12.1.5]{GovLo89}). Two different orderings of the equations
have been implemented. One is 
\[
\mathcal{G}=\left[\begin{array}{c}
\omega\mathcal{C}\\
\mathcal{A}
\end{array}\right],\quad\bar{r}=\left[\begin{array}{c}
0\\
r
\end{array}\right]
\]
while the other uses a block-bidiagonal structure as it is common
for collocation methods for ODEs, cf \cite{BaderAscher}.
It is known that the order of the equations in the weighting method may have a large impact on the accuracy of the solutions \cite{vLoan85}. In our test examples, however, we did not observe a difference in the behavior of both orderings.
\item The direct solution method by eliminating the constraints has often
the deficiency of generating a lot of fill-in in the intermediate
matrices. An approach to overcome this situation has been proposed
in \cite{vLoan85}. The solutions of the weighting approach are iteratively
enhanced by a defect correction process. This method is implemented
in the form presented in \cite{Barlow92,BarlowVemu92}. This form is called the deferred correction procedure for constrained leaset-squares problems by the authors. As a stopping criterion, the estimate (i) in \cite[p. 254]{BarlowVemu92} has been implemented. Additionally, a bound for the maximal numer of iterations can be provided.
Under reasonable
conditions, at most 2 iterations should be sufficient for obtaining
maximal (with respect to the sensitivity of the problem) accuracy
for the discrete solution.
\end{enumerate}
The results of the weighting method depend substantially on the choice
of the parameter $\omega$. In order to have an accurate approximation
of the exact solution $c_{\ast}$ of the problem (\ref{eq:discmin})-(\ref{eq:discconstraint}),
a large value of $\omega$ should be used (in the absence of rounding
errors). However, if $\omega$ becomes too large, the algorithm may
lack numerical stability. A discussion of this topic has been given
in \cite{vLoan85}. In particular, it turns out that the algorithm
used for the QR decomposition and the pivoting strategies have a strong
influence on the success of this method. In our implementation, we
use the sparse QR implementation of \cite{SPQR}. On the other hand,
an accuracy of the solution being much lower than the approximation
error of $x_{\pi}$ is not necessary.\footnote{The Eigen library
has its own implementation of a sparse QR factorization. The latter
turned out to be very slow compared to SPQR.} Therefore, a number of experiments have been done in order to obtain
some insight into what reasonable choices might be.

\begin{experiment}
\label{exp:10}Influence of the choice of the weighting parameter
$\omega$

We use Example~\ref{exa:LinCaMo}. Two sets of parameters are selected:
(i) $N=5$, $n=160$ and (ii) $N=20$, $n=20$. The choice (i) corresponds
to low degree polynomials with a corresponding large number of subintervals
while (ii) uses higher degree polynomials with a corresponding small
number of subintervals. Both cases have been selected according to
Table~\ref{tab:LinCaMo} in such a way that a high accuracy can be
obtained while at the same time having only a small influence of the
problem conditioning. The other parameters chosen in this experiment
are: $M=N+1$, Gauss-Legendre collocation nodes and Legendre polynomials
as basis functions. The error
in dependence of $\omega$ is measured both with respect to the exact
solution and with respect to a reference soltion obtained by the direct
solution method. The results are
provided in Tables \ref{tab:LinCaMoAlpha-1}--\ref{tab:LinCaMoAlpha-1-1}.
The results for Example \ref{exa:SingPer} below are quite similar. The results
indicate that an optimal $\omega$ may vary considerably depending
on the problem parameters. However, the accuracy against the exact
solution is rather insensitive of $\omega$. \hfill\qed
\end{experiment}

\begin{table}
\caption{Influence of parameter $\omega$ for the constraints in Example~\ref{exa:LinCaMo}
using $N=5$ and $n=160$. The error of the solution with respect
to the exact solution (A) and with respect to a discrete reference
solution obtained by a direct method (B) is given in the norms of
$L^{2}((0,5),\R^{7})$, $L^{\infty}((0,5),\R^{7})$ and $H_{D}^{1}((0,5),\R^{7})$\label{tab:LinCaMoAlpha-1}}

\centering{}%
\begin{tabular}{c|ccc|ccc}
 & \multicolumn{3}{c|}{(A)} & \multicolumn{3}{c}{(B)}\tabularnewline
$\omega$ & $L^{\infty}(0,5)$ & $L^{2}(0,5)$ & $H_{D}^{1}(0,5)$ & $L^{\infty}(0,5)$ & $L^{2}(0,5)$ & $H_{D}^{1}(0,5)$\tabularnewline
\hline 
1e-09 & 2.25e+00 & 4.54e+00 & 9.37e+00 & 2.25e+00 & 4.54e+00 & 8.04e+00\tabularnewline
1e-08 & 2.00e+00 & 4.59e+00 & 9.04e+00 & 2.00e+00 & 4.59e+00 & 9.04e+00\tabularnewline
1e-07 & 3.55e-01 & 5.83e-01 & 1.05e+00 & 3.55e-01 & 5.83e-01 & 1.05e+00\tabularnewline
1e-06 & 1.06e-05 & 1.66e-05 & 2.84e-05 & 1.06e-05 & 1.66e-05 & 2.84e-05\tabularnewline
1e-05 & 1.02e-07 & 1.60e-07 & 3.07e-07 & 1.02e-07 & 1.60e-07 & 2.75e-07\tabularnewline
1e-04 & 2.26e-08 & 1.49e-08 & 1.41e-07 & 5.51e-09 & 6.27e-09 & 3.51e-08\tabularnewline
1e-03 & 2.26e-08 & 1.53e-08 & 1.41e-07 & 5.51e-09 & 7.09e-09 & 3.54e-08\tabularnewline
1e-02 & 2.15e-08 & 1.39e-08 & 1.40e-07 & 4.44e-09 & 3.36e-09 & 3.31e-08\tabularnewline
1e-01 & 2.13e-08 & 1.39e-08 & 1.40e-07 & 4.28e-09 & 3.28e-09 & 3.29e-08\tabularnewline
1e+00 & 2.00e-08 & 1.31e-08 & 1.31e-07 & 2.99e-09 & 2.99e-09 & 2.29e-08\tabularnewline
1e+01 & 1.73e-08 & 1.12e-08 & 1.13e-07 & 1.27e-10 & 7.51e-11 & 7.53e-10\tabularnewline
1e+02 & 1.73e-08 & 1.12e-08 & 1.12e-07 & 3.64e-10 & 3.84e-11 & 3.85e-10\tabularnewline
1e+03 & 1.73e-08 & 1.12e-08 & 1.12e-07 & 2.36e-09 & 3.05e-10 & 3.06e-09\tabularnewline
1e+04 & 2.15e-08 & 1.15e-08 & 1.16e-07 & 1.82e-08 & 2.91e-09 & 2.92e-08\tabularnewline
1e+05 & 1.18e-07 & 3.27e-08 & 3.28e-07 & 1.26e-07 & 3.18e-08 & 3.20e-07\tabularnewline
1e+06 & 6.69e-06 & 5.08e-07 & 5.08e-06 & 6.68e-06 & 5.08e-07 & 5.09e-06\tabularnewline
1e+07 & 6.28e-05 & 5.09e-06 & 5.09e-05 & 6.28e-05 & 5.09e-06 & 5.09e-05\tabularnewline
1e+08 & 9.94e-05 & 2.82e-05 & 2.83e-04 & 9.94e-05 & 2.82e-05 & 2.83e-04\tabularnewline
1e+09 & 3.33e+01 & 7.87e+00 & 7.91e+01 & 3.33e+01 & 7.87e+00 & 7.91e+01\tabularnewline
1e+10 & 8.61e+01 & 5.91e+01 & 5.93e+02 & 8.61e+01 & 5.91e+01 & 5.93e+02\tabularnewline
\end{tabular}
\end{table}

\begin{table}
\caption{Influence of parameter $\omega$ for the constraints in Example~\ref{exa:LinCaMo}
using $N=20$ and $n=20$. The error of the solution with respect
to the exact solution (A) and with respect to a discrete reference
solution obtained by a direct method (B) is given in the norms of
$L^{2}((0,5),\R^{7})$, $L^{\infty}((0,5),\R^{7})$ and $H_{D}^{1}((0,5),\R^{7})$\label{tab:LinCaMoAlpha-1-1}}

\centering{}%
\begin{tabular}{c|ccc|ccc}
 & \multicolumn{3}{c|}{(A)} & \multicolumn{3}{c}{(B)}\tabularnewline
$\omega$ & $L^{\infty}(0,5)$ & $L^{2}(0,5)$ & $H_{D}^{1}(0,5)$ & $L^{\infty}(0,5)$ & $L^{2}(0,5)$ & $H_{D}^{1}(0,5)$\tabularnewline
\hline 
1e-09 & 2.44e+00 & 4.91e+00 & 7.59e+00 & 2.44e+00 & 4.91e+00 & 7.59e+00\tabularnewline
1e-08 & 4.40e-02 & 7.51e-02 & 1.31e-01 & 4.40e-02 & 7.51e-02 & 1.31e-01\tabularnewline
1e-07 & 6.38e-08 & 9.91e-08 & 1.85e-07 & 6.38e-08 & 9.91e-08 & 1.85e-07\tabularnewline
1e-06 & 1.35e-08 & 2.38e-08 & 3.80e-08 & 1.35e-08 & 2.38e-08 & 3.80e-08\tabularnewline
1e-05 & 2.76e-09 & 3.68e-09 & 6.77e-09 & 2.76e-09 & 3.68e-09 & 6.77e-09\tabularnewline
1e-04 & 1.86e-10 & 2.77e-10 & 5.11e-10 & 1.86e-10 & 2.77e-10 & 5.13e-10\tabularnewline
1e-03 & 5.12e-11 & 1.59e-11 & 5.68e-11 & 4.59e-11 & 1.60e-11 & 6.23e-11\tabularnewline
1e-02 & 2.49e-11 & 4.53e-12 & 4.29e-11 & 4.25e-11 & 5.62e-12 & 5.43e-11\tabularnewline
1e-01 & 3.63e-11 & 4.57e-12 & 4.59e-11 & 5.97e-11 & 6.32e-12 & 6.35e-11\tabularnewline
1e+00 & 6.01e-11 & 5.37e-12 & 5.40e-11 & 8.58e-11 & 7.61e-12 & 7.64e-11\tabularnewline
1e+01 & 1.51e-10 & 1.64e-11 & 1.64e-10 & 1.53e-10 & 1.60e-11 & 1.61e-10\tabularnewline
1e+02 & 4.67e-10 & 4.35e-11 & 4.37e-10 & 4.39e-10 & 4.14e-11 & 4.16e-10\tabularnewline
1e+03 & 1.29e-08 & 8.11e-10 & 8.15e-09 & 1.29e-08 & 8.13e-10 & 8.17e-09\tabularnewline
1e+04 & 1.50e-07 & 8.22e-09 & 8.26e-08 & 1.50e-07 & 8.22e-09 & 8.26e-08\tabularnewline
1e+05 & 6.26e-07 & 4.26e-08 & 4.28e-07 & 6.26e-07 & 4.26e-08 & 4.28e-07\tabularnewline
1e+06 & 1.10e-05 & 7.53e-07 & 7.57e-06 & 1.10e-05 & 7.53e-07 & 7.57e-06\tabularnewline
1e+07 & 3.43e-05 & 3.17e-06 & 3.19e-05 & 3.43e-05 & 3.17e-06 & 3.19e-05\tabularnewline
1e+08 & 1.85e-04 & 1.22e-05 & 1.23e-04 & 1.85e-04 & 1.22e-05 & 1.23e-04\tabularnewline
1e+09 & 1.77e-05 & 3.69e-06 & 3.22e-05 & 1.77e-05 & 3.69e-06 & 3.22e-05\tabularnewline
1e+10 & 6.74e+00 & 2.38e+00 & 1.47e+01 & 6.74e+00 & 2.38e+00 & 1.47e+01\tabularnewline
\end{tabular}
\end{table}

The following example is a boundary value problem incontrast to Example~\ref{exa:LinCaMo} which is an intial value problem.

\begin{example}
\label{exa:SingPer}On the interval $[0,1]$, consider the DAE
\[
\left[\begin{array}{cccccc}
1 & 0 & 0 & 0 & 0 & 0\\
0 & 1 & 0 & 0 & 0 & 0\\
0 & 0 & 0 & 0 & 0 & 0\\
0 & 0 & 1 & 0 & 0 & 0\\
0 & 0 & 0 & 1 & 0 & 0\\
0 & 0 & 0 & 0 & 1 & 0
\end{array}\right]\frac{d}{dt}\left[\begin{array}{c}
x_{1}\\
x_{2}\\
y_{1}\\
y_{2}\\
y_{3}\\
y_{4}
\end{array}\right]+\left[\begin{array}{cccccc}
0 & -\lambda & 0 & 0 & 0 & 0\\
-\lambda & 0 & 0 & 0 & 0 & 0\\
-1 & 0 & 1 & 0 & 0 & 0\\
0 & 0 & 0 & 1 & 0 & 0\\
0 & 0 & 0 & 0 & 1 & 0\\
0 & 0 & 0 & 0 & 0 & 1
\end{array}\right]\left[\begin{array}{c}
x_{1}\\
x_{2}\\
y_{1}\\
y_{2}\\
y_{3}\\
y_{4}
\end{array}\right]=\left[\begin{array}{c}
0\\
0\\
0\\
0\\
0\\
0
\end{array}\right],\quad\lambda>0,
\]
subject to the boundary conditions
\[
x_{1}(0)=x_{1}(1)=1.
\]
This DAE can be brought into the proper form (\ref{eq:DAE}) by setting
\[
A=\left[\begin{array}{ccccc}
1 & 0 & 0 & 0 & 0\\
0 & 1 & 0 & 0 & 0\\
0 & 0 & 0 & 0 & 0\\
0 & 0 & 1 & 0 & 0\\
0 & 0 & 0 & 1 & 0\\
0 & 0 & 0 & 0 & 1
\end{array}\right],\quad D=\left[\begin{array}{cccccc}
1 & 0 & 0 & 0 & 0 & 0\\
0 & 1 & 0 & 0 & 0 & 0\\
0 & 0 & 1 & 0 & 0 & 0\\
0 & 0 & 0 & 1 & 0 & 0\\
0 & 0 & 0 & 0 & 1 & 0
\end{array}\right],\quad B=\left[\begin{array}{cccccc}
0 & -\lambda & 0 & 0 & 0 & 0\\
-\lambda & 0 & 0 & 0 & 0 & 0\\
-1 & 0 & 1 & 0 & 0 & 0\\
0 & 0 & 0 & 1 & 0 & 0\\
0 & 0 & 0 & 0 & 1 & 0\\
0 & 0 & 0 & 0 & 0 & 1
\end{array}\right].
\]
This DAE has the tractability index $\mu=4$ and dynamical degree of freedom $l=2$. The solution reads
\begin{align*}
x_{\ast,1}(t) & =\frac{e^{-\lambda t}(e^{\lambda}+e^{2\lambda t})}{1+e^{\lambda}}\\
x_{\ast,2}(t) & =\frac{e^{-\lambda t}(-e^{\lambda}+e^{2\lambda t})}{1+e^{\lambda}}\\
y_{\ast,1}(t) & =\frac{e^{-\lambda t}(e^{\lambda}+e^{2\lambda t})}{1+e^{\lambda}}\\
y_{\ast,2}(t) & =\lambda\frac{e^{-\lambda t}(-e^{\lambda}+e^{2\lambda t})}{1+e^{\lambda}}\\
y_{\ast,3}(t) & =\lambda^{2}\frac{e^{-\lambda t}(e^{\lambda}+e^{2\lambda t})}{1+e^{\lambda}}\\
y_{\ast,4}(t) & =\lambda^{3}\frac{e^{-\lambda t}(-e^{\lambda}+e^{2\lambda t})}{1+e^{\lambda}}
\end{align*}

$\qed$
\end{example}

The iterative solver using defect corrections may overcome the difficulties
connected with a suitable choice of the parameter $\omega$ in the
weighting method. According to Experiment 10, we would expect the
optimal $\omega$ to be in the order of magnitude $10^{-3}\ldots10^{+2}$
with an optimum around $10^{-2}$. This is in contrast to the recommendations
given in \cite{BarlowVemu92} where a choice of $\omega\approx\varepsilon_{\text{mach}}^{-1/3}$
is recommended for the deferred correction algorithm. We test the performance of the deferred correction solver in the next experiment. Here, the tolerance in the convergence check is set to $10^{-15}$. The iterations are considered not to converge if the convergence check has failed after two iterations.

\begin{experiment}
\label{exp:11}We check the performance of the deferred correction
solver in dependence of the weight parameter $\omega$. Both Examples
\ref{exa:LinCaMo} and \ref{exa:SingPer} are used. The results are
presented in Tables \ref{tab:def} -- \ref{tab:def3}. The results
indicate that a larger value for $\omega$ seems to be preferable.
\hfill\qed
\end{experiment}
\begin{table}

\caption{Influence of the parameter $\omega$ on the accuracy of th discrete
solution for Example \ref{exa:LinCaMo} using $N=5$ and $n=160$.
The error of the solution with respect to the exact solution (A) and
with respect to a discrete reference solution obtained by a direct
method (B) is given in the norms of $L^{2}((0,5),\R^{7})$, $L^{\infty}((0,5),\R^{7})$
and $H_{D}^{1}((0,5),\R^{7})$. 2 iterations are applied\label{tab:def}}

\begin{minipage}{\linewidth}
\renewcommand{\thefootnote}{\theempfootnote}
\begin{centering}
\begin{tabular}{c|ccc|ccc}
 & \multicolumn{3}{c|}{(A)} & \multicolumn{3}{c}{(B)}\tabularnewline
$\omega$ & $L^{\infty}(0,5)$ & $L^{2}(0,5)$ & $H_{D}^{1}(0,5)$ & $L^{\infty}(0,5)$ & $L^{2}(0,5)$ & $H_{D}^{1}(0,5)$\tabularnewline
\hline 
0.01\footnote{Iteration did not converge} & 2.13e-08 & 1.39e-08 & 1.40e-07 & 4.30e-09 & 3.30e-09 & 3.31e-08\tabularnewline
10 & 1.73e-08 & 1.12e-08 & 1.12e-07 & 5.43e-11 & 1.62e-11 & 1.63e-10\tabularnewline
$\varepsilon_{\text{mach}}^{-1/3}$ & 1.73e-08 & 1.12e-08 & 1.12e-07 & 5.42e-11 & 1.62e-11 & 1.63e-10\tabularnewline
\end{tabular}

\end{centering}
\end{minipage}
\end{table}

\begin{table}
\caption{Influence of the parameter $\omega$ on the accuracy of th discrete
solution for Example \ref{exa:LinCaMo} using $N=20$ and $n=20$.
The error of the solution with respect to the exact solution (A) and
with respect to a discrete reference solution obtained by a direct
method (B) is given in the norms of $L^{2}((0,5),\R^{7})$, $L^{\infty}((0,5),\R^{7})$
and $H_{D}^{1}((0,5),\R^{7})$. 2 iterations are applied\label{tab:def1}}

\begin{minipage}{\linewidth}
\renewcommand{\thefootnote}{\theempfootnote}
\centering{}%
\begin{tabular}{c|ccc|ccc}
 & \multicolumn{3}{c|}{(A)} & \multicolumn{3}{c}{(B)}\tabularnewline
$\omega$ & $L^{\infty}(0,5)$ & $L^{2}(0,5)$ & $H_{D}^{1}(0,5)$ & $L^{\infty}(0,5)$ & $L^{2}(0,5)$ & $H_{D}^{1}(0,5)$\tabularnewline
\hline 
0.01\footnote{Iteration did not converge} & 2.20e-11 & 3.35e-12 & 3.37e-11 & 6.25e-11 & 6.67e-12 & 6.70e-11\tabularnewline
10 & 1.79e-11 & 1.98e-12 & 1.99e-11 & 6.26e-11 & 6.91e-12 & 6.95e-11\tabularnewline
$\varepsilon_{\text{mach}}^{-1/3}$ & 1.11e-11 & 1.71e-12 & 1.72e-11 & 6.26e-11 & 6.94e-12 & 6.97e-11\tabularnewline
\end{tabular}
\end{minipage}
\end{table}

\begin{table}
\caption{Influence of the parameter $\omega$ on the accuracy of the discrete
solution for Example \ref{exa:SingPer} using $N=20$ and $n=5$.
The error of the solution with respect to the exact solution (A) and
with respect to a discrete reference solution obtained by a direct
method (B) is given in the norms of $L^{2}((0,5),\R^{6})$, $L^{\infty}((0,5),\R^{6})$
and $H_{D}^{1}((0,5),\R^{6})$. 2 iterations are applied\label{tab:def2}}

\centering{}%
\begin{tabular}{c|ccc|ccc}
 & \multicolumn{3}{c|}{(A)} & \multicolumn{3}{c}{(B)}\tabularnewline
$\omega$ & $L^{\infty}(0,5)$ & $L^{2}(0,5)$ & $H_{D}^{1}(0,5)$ & $L^{\infty}(0,5)$ & $L^{2}(0,5)$ & $H_{D}^{1}(0,5)$\tabularnewline
\hline 
0.01 & 8.25e-08 & 6.17e-09 & 8.72e-09 & 2.52e-06 & 1.55e-07 & 2.20e-07\tabularnewline
10 & 2.73e-07 & 1.41e-08 & 2.00e-08 & 2.63e-06 & 1.61e-07 & 2.27e-07\tabularnewline
$\varepsilon_{\text{mach}}^{-1/3}$ & 3.84e-09 & 3.61e-10 & 5.11e-10 & 2.45e-06 & 1.56e-07 & 2.20e-07\tabularnewline
\end{tabular}
\end{table}
\begin{table}
\caption{Influence of the parameter $\omega$ on the accuracy of th discrete
solution for Example \ref{exa:SingPer} using $N=5$ and $n=20$.
The error of the solution with respect to the exact solution (A) and
with respect to a discrete reference solution obtained by a direct
method (B) is given in the norms of $L^{2}((0,5),\R^{6})$, $L^{\infty}((0,5),\R^{6})$
and $H_{D}^{1}((0,5),\R^{6})$. 2 iterations are applied\label{tab:def3}}

\begin{minipage}{\linewidth}
\renewcommand{\thefootnote}{\theempfootnote}
\centering{}%
\begin{tabular}{c|ccc|ccc}
 & \multicolumn{3}{c|}{(A)} & \multicolumn{3}{c}{(B)}\tabularnewline
$\omega$ & $L^{\infty}(0,5)$ & $L^{2}(0,5)$ & $H_{D}^{1}(0,5)$ & $L^{\infty}(0,5)$ & $L^{2}(0,5)$ & $H_{D}^{1}(0,5)$\tabularnewline
\hline 
0.01\footnote{Iteration did not converge} & 1.41e-06 & 4.59e-07 & 6.49e-07 & 3.75e-08 & 4.23e-09 & 5.98e-09\tabularnewline
10 & 1.39e-06 & 4.59e-07 & 6.49e-07 & 1.42e-08 & 2.63e-09 & 3.71e-09\tabularnewline
$\varepsilon_{\text{mach}}^{-1/3}$ & 1.39e-06 & 4.59e-07 & 6.49e-07 & 1.71e-08 & 2.83e-09 & 4.00e-09\tabularnewline
\end{tabular}
\end{minipage}
\end{table}

\subsection{Performance of the linear solvers}

In this section, we intend to provide some insight into the behavior of the linear solvers. This concerns both the accuracy as well as the computational resources (computation time, memory consumption). All these data are highly implementation dependent. Also the hardware architecture plays an important role.

The linear solvers have been implemented using the standard strategy of subdividing them into a factorization step and a solve step. The price to pay is a larger memory consumption. However, their use in the context of, e.g., a modified Newton method may decrease the computation time considerably.

The tests have benn run on a Linux laptop Dell Latitude E5550. While the program is a pure sequential one, the MKL library may use shared memory parallel versions of their BLAS and LAPACK routines. The CPU of the machine is an Intel(R) Core(TM) i7-5600U CPU @ 2.60GHz providing two cores, each of them capable of hyperthreading. For the test runs, cpu throttling has been disabled such that all cores ran at roughly 3.2 GHz.

The parameter for the weighting solver is $\omega=1$ while the corresponding parameter for the deferred correction solver is $\omega=\epsilon_{\textrm{mach}}^{-1/3}\approx 1.65\times 10^5$. These parameters have been chosen since they seem to be best suited for the examples. The test cases (combination of $N$ and $n$) have been selected by choosing the best combinations in Tables~\ref{tab:LinCaMo} and \ref{tab:LinCaMo_C}, respectively.

\begin{experiment} \label{exp:12}
First, we consider Example~\ref{exa:LinCaMo}. For all values of $N$, $M=N+1$ Gauss-Legendre nodes have been used. The characteristics of the test cases using Legendre basis functions are provided in Table~\ref{tab:ex12char}. For the special properties of the Legendre polynomials, the matrix $\mathcal{C}$ representing the constraints is extremly sparse featuring only three nonzero elements per row. The computational results are shown in Table~\ref{tab:ex12rt}.
In the next computations, the Chebyshev basis has been used which leads to a slightly more occupied matrix $\mathcal{C}$. The results are provided in Tables~\ref{tab:ex12Cchar} and \ref{tab:ex12Crt}.
\end{experiment}

\begin{table}
\caption{Case characteristics for Experiment~\ref{exp:12} using the Legendre basis. The number of nonzero elements in the matrices $\mathcal{A}$ and $\mathcal{C}$ are provided as reported by the functions of the Eigen library. The columns denote: the number of rows of $\mathcal{A}$
(\texttt{dimA}), the number of rows of $\mathcal{C}$ (\texttt{dimC}), the number of unknowns (\texttt{nun}), the number of nonzero elements of $\mathcal{C}$ (\texttt{nnzC}), the number of nonzero elements of $\mathcal{A}$ (\texttt{nnzA}) for the functional $\Phi^R_{\pi,M}$ and $\Phi^C_{\pi,M}$, respectively\label{tab:ex12char}}

\centering{}%
\begin{tabular}{c|rrrrrr|r|r}
 & & & & & & & $\Phi_{\pi,M}^R$ & $\Phi_{\pi,M}^C$ \\
case & $N$ & $n$ & \texttt{dimA} & \texttt{dimC} & \texttt{nun}
& \texttt{nnzC} & \texttt{nnzA} & \texttt{nnzA} \\
\hline
1 &  3 & 320 & 8964 & 1914 & 8640 & 5742 & 101124 & 101124 \\
2 &  5 &  80 & 3364 &  474 & 3280 & 1422 &  58964 &  59044 \\
3 & 10 &   5 &  389 &   24 &  380 &   72 &  12749 &  12334 \\
4 & 20 &   5 &  739 &   24 &  730 &   72 &  47509 &  46534 \\
\end{tabular}

\end{table}

\begin{table}
\caption{Computing times, permanent workspace needed, and error for the cases described in Table~\ref{tab:ex12char}. The computing times are provided in milliseconds. They are the average of 100 runs of each case. The error is measured in the norm of $H^1_D((0,5),\R^7)$.
The column headings denote: The upper bound on the number of nonzero elements of the $QR$-factors as reported by SPQR (\texttt{nWork}), the time for the matrix assembly (\texttt{tass}), the time for the factorization (\texttt{afact}), and the time for the solution (\texttt{tslv}) for both functionals $\Phi^R_{\pi,M}$ and $\Phi^C_{\pi,M}$
\label{tab:ex12rt}}

\centering{}%
\setlength{\tabcolsep}{3pt}
\begin{tabular}{cl|rrrrc|rrrrc}
 & & \multicolumn{5}{c|}{$\Phi_{\pi,M}^R$} & \multicolumn{5}{c}{$\Phi_{\pi,M}^C$} \\
case & solver & \texttt{nWork} & \texttt{tass} & \texttt{tfact} & \texttt{tslv} & error
& \texttt{nWork} & \texttt{tass} & \texttt{tfact} & \texttt{tslv} & error \\
\hline
1 & direct   & 221829 & 12 & 156 &  4 & 6.74e-04 & 221829 & 12 & 158 &  4 & 6.44e-04 \\
  & weighted & 309438 & 13 &  17 &  6 & 1.15e-03 & 309438 & 13 &  19 &  6 & 6.93e-04 \\
  & deferred & 309438 & 14 &  18 & 16 & 6.74e-04 & 309438 & 13 &  18 & 17 & 6.44e-04 \\
\hline
2 & direct   & 115932 & 12 &  50 &  4 & 9.02e-07 & 116168 &  5 &  25 &  2 & 8.50e-07 \\
  & weighted & 155334 & 14 &  17 &  6 & 1.05e-06 & 155370 &  6 &   8 &  3 & 8.95e-07 \\
  & deferred & 155334 & 14 &  16 & 14 & 9.02e-07 & 155370 &  6 &   8 &  7 & 8.50e-07 \\
\hline
3 & direct   &  24233 &  2 &   4 & 1 & 8.80e-08 &   24967 &  1 &    2 &  0 & 6.59e-08 \\
  & weighted &  26810 &  2 &   3 & 1 & 9.62e-08 &   27028 &  1 &    1 &  0 & 8.00e-08 \\
  & deferred &  26810 &  2 &   3 & 2 & 8.80e-08 &   27028 &  1 &    1 &  1 & 6.59e-08 \\
\hline
4 & direct   &  90277 &  9 &   2 & 2 & 4.47e-12 &   90052 &  1 &    1 &  2 & 5.17e-12 \\
  & weighted &  96544 & 11 &  10 & 3 & 7.44e-12 &   97857 &  9 &   11 &  3 & 5.28e-12 \\
  & deferred &  96544 & 11 &  10 & 6 & 2.17e-12 &   97857 &  9 &   10 &  5 & 2.08e-12 \\
\end{tabular}
\end{table}

\begin{table}
\caption{Case characteristics for Experiment~\ref{exp:12} using the Chebyshev basis. The number of nonzero elements in the matrices $\mathcal{A}$ and $\mathcal{C}$ are provided as reported by the functions of the Eigen library. The columns denote: the number of rows of $\mathcal{A}$
(\texttt{dimA}), the number of rows of $\mathcal{C}$ (\texttt{dimC}), the number of unknowns (\texttt{nun}), the number of nonzero elements of $\mathcal{C}$ (\texttt{nnzC}), the number of nonzero elements of $\mathcal{A}$ (\texttt{nnzA}) for the functional $\Phi^R_{\pi,M}$ and $\Phi^C_{\pi,M}$, respectively\label{tab:ex12Cchar}}

\centering{}%
\begin{tabular}{c|rrrrrr|r|r}
 & & & & & & & $\Phi_{\pi,M}^R$ & $\Phi_{\pi,M}^C$ \\
case & $N$ & $n$ & \texttt{dimA} & \texttt{dimC} & \texttt{nun}
& \texttt{nnzC} & \texttt{nnzA} & \texttt{nnzA} \\
\hline
1 &  3 & 320 & 8964 & 1914 & 8640 & 7656 & 101128 & 101124 \\
2 &  5 &  80 & 3364 &  474 & 3280 & 3318 &  58851 &  59056 \\
3 & 10 &   5 &  389 &   24 &  380 &  360 &  12846 &  12626 \\
4 & 20 &   5 &  739 &   24 &  730 &  720 &  47581 &  47191 \\
\end{tabular}

\end{table}

\begin{table}
\caption{Computing times, permanent workspace needed, and error for the cases described in Table~\ref{tab:ex12Cchar}. The computing times are provided in milliseconds. They are the average of 100 runs of each case. The error is measured in the norm of $H^1_D((0,5),\R^7)$.
The column headings denote: The upper bound on the number of nonzero elements of the $QR$-factors as reported by SPQR (\texttt{nWork}), the time for the matrix assembly (\texttt{tass}), the time for the factorization (\texttt{afact}), and the time for the solution (\texttt{tslv}) for both functionals $\Phi^R_{\pi,M}$ and $\Phi^C_{\pi,M}$
\label{tab:ex12Crt}}

\setlength{\tabcolsep}{3pt}
\centering{}%
\begin{tabular}{cl|rrrrc|rrrrc}
 & & \multicolumn{5}{c|}{$\Phi_{\pi,M}^R$} & \multicolumn{5}{c}{$\Phi_{\pi,M}^C$} \\
case & solver & \texttt{nWork} & \texttt{tass} & \texttt{tfact} & \texttt{tslv} & error
& \texttt{nWork} & \texttt{tass} & \texttt{tfact} & \texttt{tslv} & error \\
\hline
1 & direct   & 334564 & 12 & 161 &  6 & 6.74e-04 & 329266 & 15 & 163 &  6 & 6.44e-04 \\
  & weighted & 367514 & 14 &  21 &  8 & 1.15e-03 & 358591 & 15 &  23 &  8 & 6.93e-04 \\
  & deferred & 367514 & 13 &  21 & 21 & 6.74e-04 & 358591 & 15 &  22 & 22 & 6.44e-04 \\
\hline
2 & direct   & 231988 & 12 &  61 &  7 & 9.02e-07 & 231962 &  5 &  30 &  4 & 8.50e-07 \\
  & weighted & 204243 & 14 &  23 &  8 & 1.05e-06 & 201128 &  6 &  11 &  4 & 8.95e-07 \\
  & deferred & 204243 & 14 &  23 & 21 & 9.02e-07 & 201128 &  6 &  11 & 10 & 8.50e-07 \\
\hline
3 & direct   &  51343 &  2 &   7 & 1 & 8.80e-08 &   51565 &  2 &   7 &  1 & 6.59e-08 \\
  & weighted &  60861 &  2 &   5 & 1 & 9.62e-08 &   61376 &  2 &   5 &  1 & 8.00e-08 \\
  & deferred &  60861 &  3 &   5 & 3 & 8.80e-08 &   61376 &  2 &   5 &  3 & 6.59e-08 \\
\hline
4 & direct   & 208910 &  9 &   3 & 4 & 5.78e-12 &  230195 &  7 &  28 &  5 & 5.17e-12 \\
  & weighted & 164558 & 11 &  15 & 4 & 5.37e-12 &  167836 &  9 &  15 &  3 & 4.75e-12 \\
  & deferred & 164558 & 11 &  15 & 8 & 2.71e-12 &  167836 & 10 &  15 &  8 & 2.23e-12 \\
\end{tabular}

\end{table}

The previous example is an initial value problem. This structure may have consequences on the performance of the linear solvers. Therefore, in the next experiment, we consider a boundary value problem.

\begin{experiment} \label{exp:13}
We repeat Experiment~\ref{exp:12} with Example~\ref{exa:SingPer}. The problem characteristics and computational results are provided in Tables~\ref{tab:ex13char} -- \ref{tab:ex13Crt}. It should be noted that the deferred correction solver returned normally (tolerance as before $10^{-15}$) after at most two iterations in all cases. However, in some cases, the results are completely off. This happens, for example,
in Tables~\ref{tab:ex13rt} and \ref{tab:ex13Crt}, cases 1 and 2, for $\Phi_{\pi,M}^C$.
\end{experiment}

\begin{table}
\caption{Case characteristics for Experiment~\ref{exp:13} using the Legendre basis. The number of nonzero elements in the matrices $\mathcal{A}$ and $\mathcal{C}$ are provided as reported by the functions of the Eigen library. The columns denote: the number of rows of $\mathcal{A}$
(\texttt{dimA}), the number of rows of $\mathcal{C}$ (\texttt{dimC}), the number of unknowns (\texttt{nun}), the number of nonzero elements of $\mathcal{C}$ (\texttt{nnzC}), the number of nonzero elements of $\mathcal{A}$ (\texttt{nnzA}) for the functional $\Phi^R_{\pi,M}$ and $\Phi^C_{\pi,M}$, respectively\label{tab:ex13char}}

\centering{}%
\begin{tabular}{c|rrrrrr|r|r}
 & & & & & & & $\Phi_{\pi,M}^R$ & $\Phi_{\pi,M}^C$ \\
case & $N$ & $n$ & \texttt{dimA} & \texttt{dimC} & \texttt{nun}
& \texttt{nnzC} & \texttt{nnzA} & \texttt{nnzA} \\
\hline
1 &  4 & 320 & 9602 & 1595 & 9280 & 4785 &  86403 &  80643 \\
2 &  5 & 160 & 5762 &  795 & 5600 & 1422 &  63363 &  63363 \\
3 & 10 &   5 &  332 &   20 &  325 &   60 &   6933 &   6663 \\
4 & 20 &   5 &  632 &   20 &  625 &   60 &  25793 &  25263 \\
\end{tabular}

\end{table}

\begin{table}
\caption{Computing times, permanent workspace needed, and error for the cases described in Table~\ref{tab:ex13char}. The computing times are provided in milliseconds. They are the average of 100 runs of each case. The error is measured in the norm of $H^1_D((0,1),\R^6)$.
The column headings denote: The upper bound on the number of nonzero elements of the $QR$-factors as reported by SPQR (\texttt{nWork}), the time for the matrix assembly (\texttt{tass}), the time for the factorization (\texttt{afact}), and the time for the solution (\texttt{tslv}) for both functionals $\Phi^R_{\pi,M}$ and $\Phi^C_{\pi,M}$
\label{tab:ex13rt}}

\setlength{\tabcolsep}{3pt}
\centering{}%
\begin{tabular}{cl|rrrrc|rrrrc}
 & & \multicolumn{5}{c|}{$\Phi_{\pi,M}^R$} & \multicolumn{5}{c}{$\Phi_{\pi,M}^C$} \\
case & solver & \texttt{nWork} & \texttt{tass} & \texttt{tfact} & \texttt{tslv} & error
& \texttt{nWork} & \texttt{tass} & \texttt{tfact} & \texttt{tslv} & error \\
\hline
1 & direct   & 437085 & 14 & 164 &  8 & 1.58e-04 & 397127 & 13 & 158 &  7 & 1.24e-04 \\
  & weighted & 235746 & 14 &  16 &  5 & 8.22e-05 & 341713 &  7 &  22 &  7 & 2.07e-05 \\
  & deferred & 235746 & 14 &  16 & 13 & 5.53e-02 & 341713 & 14 &  21 & 25 & \textbf{9.09e+02} \\
\hline
2 & direct   & 348742 & 17 & 124 & 12 & 2.58e-05 & 348742 & 15 & 123 & 12 & 1.57e-05 \\
  & weighted & 153062 &  9 &   9 &  3 & 9.29e-07 & 153062 &  9 &   9 &  3 & 7.75e-06 \\
  & deferred & 153062 & 10 &   9 &  8 & 1.38e-01 & 153062 &  9 &  10 & 10 & \textbf{1.47e-01} \\
\hline
3 & direct   &  11617 &  1 &   3 & 0 & 8.04e-10 &   12155 &  1 &    2 &  0 & 1.06e-09 \\
  & weighted &  12400 &  2 &   2 & 1 & 1.26e-09 &   12141 &  1 &    1 &  0 & 5.52e-09 \\
  & deferred &  12400 &  2 &   2 & 1 & 4.18e-11 &   12141 &  1 &    1 &  1 & 5.08e-09 \\
\hline
4 & direct   &  46847 &  6 &   9 & 1 & 7.24e-08 &   46883 &  2 &    4 &  7 & 3.54e-07 \\
  & weighted &  42947 &  7 &   6 & 2 & 1.42e-07 &   42859 &  3 &    3 &  1 & 1.71e-07 \\
  & deferred &  42947 &  6 &   6 & 4 & 5.27e-09 &   42859 &  3 &    3 &  2 & 1.51e-07 \\
\end{tabular}

\end{table}

\begin{table}
\caption{Case characteristics for Experiment~\ref{exp:13} using the Chebyshev basis. The number of nonzero elements in the matrices $\mathcal{A}$ and $\mathcal{C}$ are provided as reported by the functions of the Eigen library. The columns denote: the number of rows of $\mathcal{A}$
(\texttt{dimA}), the number of rows of $\mathcal{C}$ (\texttt{dimC}), the number of unknowns (\texttt{nun}), the number of nonzero elements of $\mathcal{C}$ (\texttt{nnzC}), the number of nonzero elements of $\mathcal{A}$ (\texttt{nnzA}) for the functional $\Phi^R_{\pi,M}$ and $\Phi^C_{\pi,M}$, respectively\label{tab:ex13Cchar}}

\centering{}%
\begin{tabular}{c|rrrrrr|r|r}
 & & & & & & & $\Phi_{\pi,M}^R$ & $\Phi_{\pi,M}^C$ \\
case & $N$ & $n$ & \texttt{dimA} & \texttt{dimC} & \texttt{nun}
& \texttt{nnzC} & \texttt{nnzA} & \texttt{nnzA} \\
\hline
1 &  4 & 320 & 9602 & 1595 & 9280 & 7656 &  86406 &  82566 \\
2 &  5 & 160 & 5762 &  795 & 5600 & 5565 &  63367 &  63367 \\
3 & 10 &   5 &  332 &   20 &  325 &  300 &   6945 &   6795 \\
4 & 20 &   5 &  632 &   20 &  325 &  600 &  25830 &  25560 \\
\end{tabular}

\end{table}

\begin{table}
\caption{Computing times, permanent workspace needed, and error for the cases described in Table~\ref{tab:ex13Cchar}. The computing times are provided in milliseconds. They are the average of 100 runs of each case. The error is measured in the norm of $H^1_D((0,1),\R^6)$.
The column headings denote: The upper bound on the number of nonzero elements of the $QR$-factors as reported by SPQR (\texttt{nWork}), the time for the matrix assembly (\texttt{tass}), the time for the factorization (\texttt{afact}), and the time for the solution (\texttt{tslv}) for both functionals $\Phi^R_{\pi,M}$ and $\Phi^C_{\pi,M}$
\label{tab:ex13Crt}}

\setlength{\tabcolsep}{3pt}
\centering{}%
\begin{tabular}{cl|rrrrc|rrrrc}
 & & \multicolumn{5}{c|}{$\Phi_{\pi,M}^R$} & \multicolumn{5}{c}{$\Phi_{\pi,M}^C$} \\
case & solver & \texttt{nWork} & \texttt{tass} & \texttt{tfact} & \texttt{tslv} & error
& \texttt{nWork} & \texttt{tass} & \texttt{tfact} & \texttt{tslv} & error \\
\hline
1 & direct   & 796757 & 27 & 360 & 28 & 6.77e-05 & 807507 & 26 & 363 & 29 & 1.77e-04 \\
  & weighted & 502962 & 16 &  28 & 11 & 1.90e-06 & 471966 & 16 &  29 & 11 & 1.11e-05 \\
  & deferred & 502962 & 15 &  28 & 27 & 2.33e-07 & 471966 & 15 &  29 & 35 & \textbf{1.19e+02} \\
\hline
2 & direct   & 513054 & 17 & 143 & 16 & 3.59e-05 & 513054 & 15 & 143 & 17 & 2.37e-05 \\
  & weighted & 347439 & 10 &  19 &  7 & 8.73e-07 & 347439 &  9 &  19 &  7 & 4.71e-06 \\
  & deferred & \multicolumn{5}{c}{Solver failed} & 347439 &  9 &  20 & 25 & \textbf{5.07e+02} \\
\hline
3 & direct   &  29347 &  2 &   4 & 1 & 2.69e-09 &   30843 &  1 &   3 &  1 & 1.40e-09 \\
  & weighted &  25392 &  2 &   3 & 1 & 5.10e-10 &   26984 &  1 &   1 &  0 & 8.52e-10 \\
  & deferred &  25392 &  2 &   2 & 2 & 4.41e-11 &   26984 &  1 &   1 &  1 & 1.22e-09 \\
\hline
4 & direct   & 122665 &  6 &  16 & 3 & 6.70e-08 &  148882 &  5 &  18 &  4 & 6.68e-07 \\
  & weighted & 109429 &  7 &  10 & 3 & 5.22e-08 &  109345 &  6 &  10 &  3 & 5.43e-08 \\
  & deferred & 109429 &  7 &  11 & 7 & 6.09e-11 &  109345 &  6 &  11 &  7 & 2.62e-09 \\
\end{tabular}

\end{table}

It should be noted that a considerable amount of memory for the QR-factorizations is consumed by the internal representation of the Q-factor in SPQR. This can be avoided if the factorization and solution steps are intervowen.

\subsection{Sensitivity of boundary condition weighting}\label{subsec:Linear-DAEs}
As already known for boundary value problems for ODEs and index-1 DAEs, a special problem is the scaling of the boundary condition, and hence, here the inclusion of the boundary conditions (\ref{eq:BC}).
Their scaling is independent of the scaling of the DAE (\ref{eq:DAE}).
Therefore, it seems to be reasonable to provide an additional possibility
for the scaling of the boundary conditions. We decided to enable this
by introducing an additional parameter $\alpha$ to be chosen by the
user. So, $\Phi$ from (\ref{eq:Phi}) is replaced by the functional
\[
\tilde{\Phi}(x)=\int_{a}^{b}|A(t)(Dx)'(t)+B(t)x(t)-q(t)|^{2}{\rm dt}
+\alpha |G_ax(a)+G_bx(b)
-d|^{2}.
\]
Analogously, the discretized versions $\Phi_{\pi,M}^{R}$, $\Phi_{\pi,M}^{I}$
and $\Phi_{\pi,M}^{C}$ are replaced by their counterparts $\tilde{\Phi}_{\pi,M}^{R}$,
$\tilde{\Phi}_{\pi,M}^{I}$ and $\tilde{\Phi}_{\pi,M}^{C}$ with weighted
boundary conditions. The convergence theorems will hold true for these
modifications of the functional, too.
\begin{experiment}
\label{exp:8}Influence of $\alpha$ on the accuracy

We use the example and settings of Experiment~\ref{exp:10}.
The results are provided in Table~\ref{tab:LinCaMoAlpha}.
\hfill\qed
\end{experiment}
\begin{table}

\caption{Influence of weight parameter $\alpha$ for the boundary conditions
in Example~\ref{exa:LinCaMo}. The error of the solution is given
in the norms of $L^{2}((0,5),\R^{7})$, $L^{\infty}((0,5),\R^{7})$
and $H_{D}^{1}((0,5),\R^{7})$\label{tab:LinCaMoAlpha}}

\centering{}%
\begin{tabular}{c|ccc|ccc}
 & \multicolumn{3}{c|}{$N=5$, $n=160$} & \multicolumn{3}{c}{$N=20$, $n=20$}\tabularnewline
$\alpha$ & $L^{\infty}(0,5)$ & $L^{2}(0,5)$ & $H_{D}^{1}(0,5)$ & $L^{\infty}(0,5)$ & $L^{2}(0,5)$ & $H_{D}^{1}(0,5)$\tabularnewline
\hline 
1e-10 & 3.18e+00 & 7.03e+00 & 1.21e+01 & 1.60e+00 & 3.10e+00 & 5.09e+00\tabularnewline
1e-09 & 9.33e-07 & 2.33e-06 & 3.84e-06 & 1.60e+00 & 3.10e+00 & 5.09e+00\tabularnewline 
1e-08 & 1.58e-07 & 3.52e-07 & 6.16e-07 & 1.05e-07 & 1.94e-07 & 3.54e-07\tabularnewline
1e-07 & 1.27e-07 & 1.39e-08 & 3.26e-08 & 5.06e-09 & 1.10e-08 & 2.00e-08\tabularnewline
1e-06 & 7.17e-08 & 2.20e-09 & 1.68e-08 & 9.60e-10 & 2.29e-09 & 4.10e-09\tabularnewline
1e-05 & 9.60e-08 & 1.59e-09 & 1.58e-08 & 7.64e-11 & 2.07e-10 & 3.80e-10\tabularnewline
1e-04 & 6.99e-08 & 1.59e-09 & 1.60e-08 & 5.00e-11 & 4.07e-11 & 9.26e-11\tabularnewline
1e-03 & 9.83e-08 & 1.82e-09 & 1.83e-08 & 3.91e-11 & 6.41e-12 & 5.46e-11\tabularnewline
1e-02 & 1.15e-07 & 2.28e-09 & 2.29e-08 & 6.37e-11 & 6.26e-12 & 6.25e-11\tabularnewline
1e-01 & 6.43e-08 & 1.27e-09 & 1.27e-08 & 5.11e-11 & 6.61e-12 & 6.64e-1\tabularnewline
1e+00 & 6.04e-08 & 1.13e-09 & 1.13e-08 & 6.66e-11 & 7.50e-12 & 7.54e-11\tabularnewline
1e+01 & 2.15e-07 & 3.40e-09 & 3.42e-08 & 7.97e-11 & 9.85e-12 & 9.89e-11\tabularnewline
1e+02 & 4.12e-07 & 5.66e-09 & 5.68e-08 & 6.78e-11 & 8.10e-12 & 8.14e-11\tabularnewline
1e+03 & 4.51e-06 & 5.74e-08 & 5.76e-07 & 9.60e-11 & 9.81e-12 & 9.85e-11\tabularnewline
1e+04 & 2.31e-05 & 2.93e-07 & 2.95e-06 & 2.24e-09 & 1.52e-10 & 1.52e-09\tabularnewline
1e+05 & 4.68e-04 & 5.94e-06 & 5.97e-05 & 2.91e-08 & 1.35e-09 & 1.36e-08\tabularnewline
1e+06 & 2.12e+03 & 5.16e+01 & 5.19e+02 & 2.34e-07 & 1.68e-08 & 1.68e-07\tabularnewline
1e+07 & 6.53e+03 & 1.03e+02 & 1.04e+03 & 2.97e-06 & 1.77e-07 & 1.77e-06\tabularnewline
1e+08 & 4.60e+02 & 1.78e+01 & 1.79e+02 & 4.76e-06 & 3.72e-07 & 3.73e-06\tabularnewline
1e+09 & 2.05e+01 & 3.27e+00 & 3.24e+01 & 4.56e+01 & 4.90e+00 & 4.91e+01\tabularnewline
\end{tabular}
\end{table}

\begin{experiment}
\label{exp:9}Influence of $\alpha$ on the accuracy

We repeat the previous experiment with Example~\ref{exa:SingPer}.
The discretization parameters are (i) $N=5$, $n=20$ and (ii) $N=20$,
$n=5$. All other settings correspond to those of Experiment \ref{exp:8}.
The results are presented in Table~\ref{tab:SingPerAlpha}. \hfill\qed
\end{experiment}
\begin{table}
\caption{Influence of weight parameter $\alpha$ for the boundary conditions
in Example~\ref{exa:SingPer}. The error of the solution is given
in the norms of $L^{2}((0,1),\R^{6})$, $L^{\infty}((0,1),\R^{6})$
and $H_{D}^{1}((0,1),\R^{6})$\label{tab:SingPerAlpha}}

\centering{}%
\begin{tabular}{c|ccc|ccc}
 & \multicolumn{3}{c|}{$N=5$, $n=20$} & \multicolumn{3}{c}{$N=20$, $n=5$}\tabularnewline
$\alpha$ & $L^{\infty}(0,1)$ & $L^{2}(0,1)$ & $H_{D}^{1}(0,1)$ & $L^{\infty}(0,1)$ & $L^{2}(0,1)$ & $H_{D}^{1}(0,1)$\tabularnewline
\hline 
1e-10 & 4.21e-02 & 7.02e-02 & 9.13e-02 & 1.03e-06 & 8.55e-08 & 1.21e-07\tabularnewline
1e-09 & 4.46e-04 & 7.38e-04 & 9.60e-04 & 1.00e-06 & 6.11e-08 & 8.64e-08\tabularnewline
1e-08 & 4.40e-06 & 6.71e-06 & 8.80e-06 & 1.14e-06 & 6.48e-08 & 9.16e-08\tabularnewline
1e-07 & 1.47e-06 & 4.87e-07 & 6.88e-07 & 9.84e-07 & 6.02e-08 & 8.51e-08\tabularnewline
1e-06 & 1.39e-06 & 4.59e-07 & 6.49e-07 & 1.67e-06 & 1.10e-07 & 1.56e-07\tabularnewline
1e-05 & 1.40e-06 & 4.59e-07 & 6.49e-07 & 1.19e-06 & 8.21e-08 & 1.16e-07\tabularnewline
1e-04 & 1.40e-06 & 4.59e-07 & 6.49e-07 & 8.55e-07 & 6.48e-08 & 9.17e-08\tabularnewline
1e-03 & 1.40e-06 & 4.59e-07 & 6.49e-07 & 1.44e-06 & 1.04e-07 & 1.47e-07\tabularnewline
1e-02 & 1.40e-06 & 4.59e-07 & 6.49e-07 & 5.14e-07 & 4.77e-08 & 6.75e-08\tabularnewline
1e-01 & 1.40e-06 & 4.59e-07 & 6.49e-07 & 1.69e-06 & 8.49e-08 & 1.20e-07\tabularnewline
1e+00 & 1.40e-06 & 4.59e-07 & 6.49e-07 & 2.45e-06 & 1.56e-07 & 2.20e-07\tabularnewline
1e+01 & 1.40e-06 & 4.59e-07 & 6.49e-07 & 1.83e-06 & 1.09e-07 & 1.54e-07\tabularnewline
1e+02 & 1.40e-06 & 4.59e-07 & 6.49e-07 & 1.91e-05 & 8.14e-07 & 1.15e-06\tabularnewline
1e+03 & 1.40e-06 & 4.59e-07 & 6.49e-07 & 1.40e-04 & 1.10e-06 & 1.55e-06\tabularnewline
1e+04 & 1.41e-06 & 4.59e-07 & 6.49e-07 & 1.27e-03 & 5.34e-05 & 7.56e-05\tabularnewline
1e+05 & 1.39e-06 & 4.59e-07 & 6.49e-07 & 3.69e-04 & 1.94e-05 & 2.75e-05\tabularnewline
1e+06 & 1.63e-06 & 4.66e-07 & 6.59e-07 & 3.98e-04 & 3.42e-05 & 4.83e-05\tabularnewline
1e+07 & 1.99e+02 & 5.07e+01 & 7.18e+01 & 2.11e-03 & 3.53e-04 & 4.99e-04\tabularnewline
1e+08 & 1.99e+02 & 5.07e+01 & 7.18e+01 & 1.22e-01 & 2.83e-02 & 4.01e-02\tabularnewline
1e+09 & 1.99e+02 & 5.07e+01 & 7.18e+01 & 4.86e-01 & 2.05e-01 & 2.90e-01\tabularnewline
\end{tabular}
\end{table}

The results of Experiments \ref{exp:8} and \ref{exp:9} indicate
that the final accuracy is rather insensitive to the choice of $\alpha$.
It should be noted that the coefficient matrices in Examples \ref{exa:LinCaMo}
and \ref{exa:SingPer} are well-scaled.

\section{Final remarks and conclusions}\label{s.Final}

In summary, in the present paper, we investigated questions related to an efficient and reliable realization of a least-squares collocation method. These questions are particularly important since a higher index DAE is an essentially ill-posed problem in naturally given spaces, which is why we must be prepared for highly sensitive discrete problems. In order to obtain a overall procedure that is as robust as possible, we provided criteria which led to a robust selection of the collocation points and of the basis functions, whereby the latter is also useful for the shape of the resulting discrete problem. Additionally, a number of new, more detailed, error estimates have been given that support some of the design decisions. 
The following particular items are worth highlighting in this context:
\begin{itemize}
\item The basis for the approximation space should be appropriately shifted
and scaled orthogonal polynomials. We could not observe any larger differences
between the behavior of Legendre and Chebyshev polynomials.
\item The collocation points should be chosen to be the Gauss-Legendre,
Lobatto, or Radau nodes. This leads to discrete problems whose conditioning
using the discretization by interpolation ($\Phi_{\pi,M}^{R}$) is
not much worse than that resembling collocation methods for ordinary
differential equations ($\Phi_{\pi,M}^{C}$). A particular efficient
and stable implementation is obtained if Gauss-Legendre or Radau nodes
are used since, in this case, diagonal weighting ($\Phi_{\pi,M}^{I}$)
coincides with the interpolation approach.
\item A critical ingredient for the implementation of the method
is the algorithm used for the solution of the constrained linear least-squares
problems. Given the expected bad conditioning of the least-squares
problem, a QR-factorization with column pivoting must lie at the heart
of the algorithm. At the same time, the sparsity structure must be
used as best as possible. In our tests, the direct solver seems to be the most robust one. With respect to efficiency and accuracy, the deferred correction solver is preferable. However, it failed in certain tests.
\item It seems as if, for problems with a smooth solution, a higher degree $N$ of the ansatz polynomials with
a low number of subintervals $n$ in the mesh is preferable over a
smaller degree with a larger number of subintervals with respect to
accuracy. Some first theoretical justification has been provided for this claim.
\item So far, in all experiments of this and previously published papers, we did not observe any serious differences
in the accuracy obtained in dependence on the choice of $M>N$ for
fixed $n$. The results for $M=N+1$ are not much different from those obtained for a larger $M$.
\item While superconvergence in classical collocation for ODEs and index-1 DAEs is a very favorable phenomenon, we could not find anything analogous in all our experiments.
\item The simple collocation procedure using $\Phi_{\pi,M}^C$ performs surprisingly well. In fact, the results are, in our experiments, in par with those using $\Phi_{\pi,M}^R=\Phi_{\pi,M}^I$. However, we have no theoretical justification for this as yet.
\item Our method is designed for variable grids. However, so far we have only worked with constant step size.
 In order to be able to adapt the
grid and the polynomial degree, or even select appropriate grids, it is important to understand the structure of the error, that is, 
how the global error depends on local errors. This is a very important open problem, for which we have no solution yet.
\end{itemize}
In conclusion, we note that earlier implementations, among others the one from the very first paper in this matter \cite{HMTWW}, which started from proven ingredients  for ODE codes, are from today's point of view and experience a rather bad version for the least-squares collocation. 
Nevertheless, the test results calculated with it were already very impressive. This strengthens our belief that 
a careful implementation of the method gives rise to a very efficient solver for higher-index DAEs.

\appendix

\section{Some facts about classical orthogonal polynomials}

In the derivations, classical orthogonal polynomials have been heavily
used. For the reader's convenience important properties are collected
below.

\subsection{Legendre Polynomials\label{subsec:Legendre-polynomials}}

The Legendre polynomials $P_{\nu}$, $\nu=0,1,\ldots$, are defined
by the recurrence relation
\begin{align}
P_{0}(\tau) & =1,\nonumber \\
P_{1}(\tau) & =\tau,\label{eq:Legendre-1}\\
(\nu+1)P_{\nu+1}(\tau) & =(2\nu+1)\tau P_{\nu}(\tau)-\nu P_{\nu-1}(\tau),\quad\nu=1,2,\ldots.\nonumber 
\end{align}
Some properties of the Legendre polynomials are
\begin{enumerate}
\item $P_{\nu}(-1)=(-1)^{\nu}$, $P_{\nu}(1)=1,\quad\nu=0,1,\ldots$,
\item $\int_{-1}^{1}P_{0}(\tau)=2$, $\int_{-1}^{1}P_{\nu}(\tau)=0,\quad\nu=1,2,\ldots$
\item $\int_{-1}^{1}P_{\nu}(\tau)P_{\mu}(\tau)d\tau=\frac{2}{2\nu+1}\delta_{\nu\mu},\quad\nu,\nu=0,1,\ldots$,
where $\delta_{\nu\mu}$ denotes the Kronecker $\delta$-symbol,
\item $P_{\nu+1}'(\tau)-P_{\nu-1}'(\tau)=(2\nu+1)P_{\nu}(\tau),\quad\nu=1,2,\ldots$
\end{enumerate}
The latter property is useful for representing integrals,
\begin{align}
\int_{-1}^{\tau}P_{\nu}(\sigma)d\sigma & =\frac{1}{2\nu+1}\left(P_{\nu+1}(\tau)-P_{\nu-1}(\tau)-(-1)^{\nu+1}+(-1)^{\nu-1}\right)\nonumber \\
 & =\frac{1}{2\nu+1}\left(P_{\nu+1}(\tau)-P_{\nu-1}(\tau)\right).\label{eq:Int_Legendre}
\end{align}
Moreover, $\int_{-1}^{\tau}P_{0}(\sigma)d\sigma=\tau+1$.

For a stable evaluation of the Legendre polynomials, we use a representation
proposed in \cite{LebedevBarburin65},
\[
P_{\nu+1}(\tau)=\frac{\nu}{\nu+1}(\tau P_{\nu}(\tau)-P_{\nu-1}(\tau))+\tau P_{\nu}(\tau).
\]
In the implementation, all polynomials must be evaluated simulataneously
for each given $\tau$. The evaluation of the recursions is cheap.
Linear combinations of the basis function can be conveniently and
stably evaluated using the Clenshaw algorithm \cite[p. 56]{FoxParker68}\cite{Barrio02,Smok02}.

The shifted Legendre polynomials $\tilde{P}_{\nu}$ are given by $\tilde{P}_{\nu}(\rho)=P_{\nu}(2\rho-1)$,
$\nu=0,1,\ldots$\footnote{$\tilde{P}_{\nu}$ ist eine Standardbezeichnung.}
They fulfill the orthogonality relations
\[
\int_{0}^{1}\tilde{P}_{\nu}(\rho)\tilde{P}_{\mu}(\rho)d\rho=\frac{1}{2\nu+1}\delta_{\nu\mu}.
\]

Moreover, we introduce the normalized shifted Legendre polynomials
$\hat{P}_{\nu}$ by
\[
\hat{P}_{\nu}(\rho)=(2\nu+1)^{1/2}\tilde{P}(\rho).
\]

\subsection{Chebyshev polynomials\label{subsec:Chebyshev-polynomials}}

The Chebyshev polynomials of the first kind $T_{\nu}$, $\nu=0,1,\ldots$,
are defined by the recurrence relation
\begin{align}
T_{0}(\tau) & =1,\nonumber \\
T_{1}(\tau) & =\tau,\label{eq:Chebyshev-1}\\
T_{\nu+1}(\tau) & =2\tau T_{\nu}(\tau)-T_{\nu-1}(\tau),\quad\nu=1,2,\ldots.\nonumber 
\end{align}
Some properties of the Chebyshev polynomials are
\begin{enumerate}
\item $T_{\nu}(-1)=(-1)^{\nu},\quad T_{\nu}(1)=1,\quad\nu=0,1,\ldots$
\item $T_{\nu}(\tau)=\frac{1}{2}\left(\frac{1}{\nu+1}T_{\nu+1}'(\tau)-\frac{1}{\nu-1}T_{\nu-1}'(\tau)\right),\quad\nu=2,3,\ldots$
\end{enumerate}
Similarly as before, we obtain the simple presentation
\begin{equation}
\int_{-1}^{\tau}T_{\nu}(\sigma)d\sigma=\frac{1}{2(\nu^{2}-1)}\left((\nu-1)T_{\nu+1}(\tau)-(\nu+1)T_{\nu-1}(\tau)+(-1)^{\nu-1}2\right).\label{eq:Int_Chebyshev}
\end{equation}
The orthogonality property of the Chebyshev polynomials reads
\[
\int_{-1}^{1}T_{\nu}(\tau)T_{\mu}(\tau)\frac{d\tau}{\sqrt{1-x^{2}}}=\begin{cases}
0, & \nu\neq\mu,\\
\pi, & \nu=\mu=0,\\
\frac{\pi}{2}, & \nu=\mu\neq0.
\end{cases}
\]
The normalized Chebyshev polynomials $\bar{T}_{\nu}$ are given by

\[
\bar{T}_{\nu}(\tau)=\begin{cases}
\sqrt{\frac{1}{\pi}}T_{0}(\tau), & \nu=0,\\
\sqrt{\frac{2}{\pi}}T_{\nu}(\tau), & \nu=1,2,\ldots.
\end{cases}
\]
Linear combinations of Chebyshev polynomials can be stably computed
by the Clenshaw algorithm \cite[p. 57ff]{FoxParker68},\cite{Barrio02,Smok02}.

\subsection{The structure of the discrete problems\label{subsec:The-structure-of}}

Consider the linear DAE (\ref{eq:DAE}). In order to simplify the
notation slightly, define $E=AD$ such that, for sufficiently smooth
functions $x\in X_{\pi}$, (\ref{eq:DAE}) is equivalent to
\[
E(t)x'(t)+B(t)x(t)=q(t),\quad t\in(t_{j-1},t_{j}),\quad j=1,\ldots,n.
\]
Let, on $(t_{j-1},t_{j})$, $x(t)=(x_{j1}(t),\ldots,x_{jm}(t))^{T}$.
Then, we have the representations
\begin{align}
x_{j\kappa}(t) & =\sum_{l=0}^{N}c_{j\kappa l}\bar{p}_{jl}(t),\quad\kappa=1,\ldots,k,\nonumber \\
x_{j\kappa}(t) & =\sum_{l=0}^{N-1}c_{j\kappa l}p_{jl}(t),\quad\kappa=k+1,\ldots,m,\label{eq:reprx}
\end{align}
with $p_{jl}$, $\bar{p}_{jl}$ from (\ref{eq:scaledp}). Introduce
\[
\bar{Q}_{j}(t)=(\bar{p}_{j1}(t),\ldots,\bar{p}_{lN}(t)),\quad Q_{j}(t)=(p_{j1}(t),\ldots,p_{j,N-1}(t))
\]
as well as
\[
a_{j}(t)=\left[\begin{array}{cc}
I_{k}\otimes\bar{Q}_{j}(t) & 0\\
0 & I_{m-k}\otimes Q_{j}(t)
\end{array}\right]\in\R^{m\times(mN+k)}.
\]
Collect the coefficents in (\ref{eq:reprx}) in the vector
\[
c_{j}=(c_{j10},\ldots,c_{j1N},c_{j20},\ldots,c_{jm,N-1})^{T}\in\R^{mN+k}.
\]
Then it holds 
\[
x_{j}(t)=a_{j}(t)c_{j}.
\]
Then, for $W_{j}$ of (\ref{eq:W}), we have the representation
\[
w_{j}(t_{ji})=\left[E(t_{ji})a_{j}'(t_{ji})+B(t_{ji})a_{j}(t_{ji})\right]c_{j}
-q(t_{ji})=:A_{ji}c_{j}-r_{ji}
\]
and
\[
W_{j}=h^{1/2}\left[\begin{array}{c}
A_{j1}\\
\vdots\\
A_{jM}
\end{array}\right]c_{j}-
h^{1/2}\left[\begin{array}{c}r_{j1}\\
\vdots\\
r_{jM}
\end{array}\right]
.
\]

The functionals $\Phi_{\pi,M}^{r}$ have, for $r=R,I,C$ an representation
of the kind
\[
\Phi_{\pi,M}^{r}=W^{T}\mathcal{L}^{r}W+|G_{a}x(a)+G_{b}x(b)-d|^{2}.
\]
Assume that there exists a matrix $\hat{L}^{r}$ such that $L^{r}=\left(\hat{L}^{r}\right)^{T}\hat{L}^{r}$.
For $r=I,C$, simple possibilities are $\hat{L}^{I}=\diag(\sqrt{\gamma_{1}},\ldots,\sqrt{\gamma_{M}})$
and $\hat{L}^{C}=M^{-1/2}I_{M}$. For $L^{R}$, the choice $\hat{L}^{R}=\tilde{V}^{-1}$
(cf (\ref{eq:VDMmatrix})) is suitable. Define 
\[
A_{j}=h_{j}^{1/2}\diag(\hat{L}^{r}\otimes I_{m},\ldots,\hat{L}^{r}\otimes I_{m})\left[\begin{array}{c}
A_{j1}\\
\vdots\\
A_{jM}
\end{array}\right],
\quad
r_j = h_{j}^{1/2}\diag(\hat{L}^{r}\otimes I_{m},\ldots,\hat{L}^{r}\otimes I_{m})\left[\begin{array}{c} r_{j1} \\ \vdots \\ r_{jM}
\end{array}\right]
.
\]
Then we set
\[
\mathcal{A}=\left[\begin{array}{ccccc}
A_{1} & 0 & \cdots &  & 0\\
0 & \ddots &  &  & \vdots\\
\vdots &  & \ddots\\
 &  &  & \ddots & 0\\
0 &  &  &  & A_{n}\\
G_{a}a_{1}(a) & 0 & \cdots & 0 & G_{b}a_{n}(b)
\end{array}\right],
\quad
r=\left[\begin{array}{c} r_1 \\ \vdots \\ r_n \end{array}\right]
.
\]
Moreover, the continuity conditions (\ref{eq:CC}) can be represented by the matrix
\[
\mathcal{C}=\left[\begin{array}{cccccc}
I_{k}\otimes Q_{1}(t_{1}) & I_{k}\otimes Q_{2}(t_{1})\\
 & I_{k}\otimes Q_{2}(t_{2}) & I_{k}\otimes Q_{3}(t_{2})\\
 &  & \ddots & \ddots\\
 &  &  & \ddots & \ddots\\
 &  &  &  & I_{k}\otimes Q_{n-1}(t_{n-1}) & I_{k}\otimes Q_{n}(t_{n-1})
\end{array}\right].
\]
The discrete minimization problem becomes, therefore,
\[
 \varphi^r(c)=\|\mathcal{A}c-r\|_{\R^{nmM+l}}^{2}\rightarrow\min!
\]
under the constraint
\[
\mathcal{C}c=0.
\]

\bibliographystyle{plain}
\bibliography{LSCMImpl}

\section*{\clearpage}

\begin{figure}
 \begin{tabular}{c}
  \includegraphics[scale=0.4]{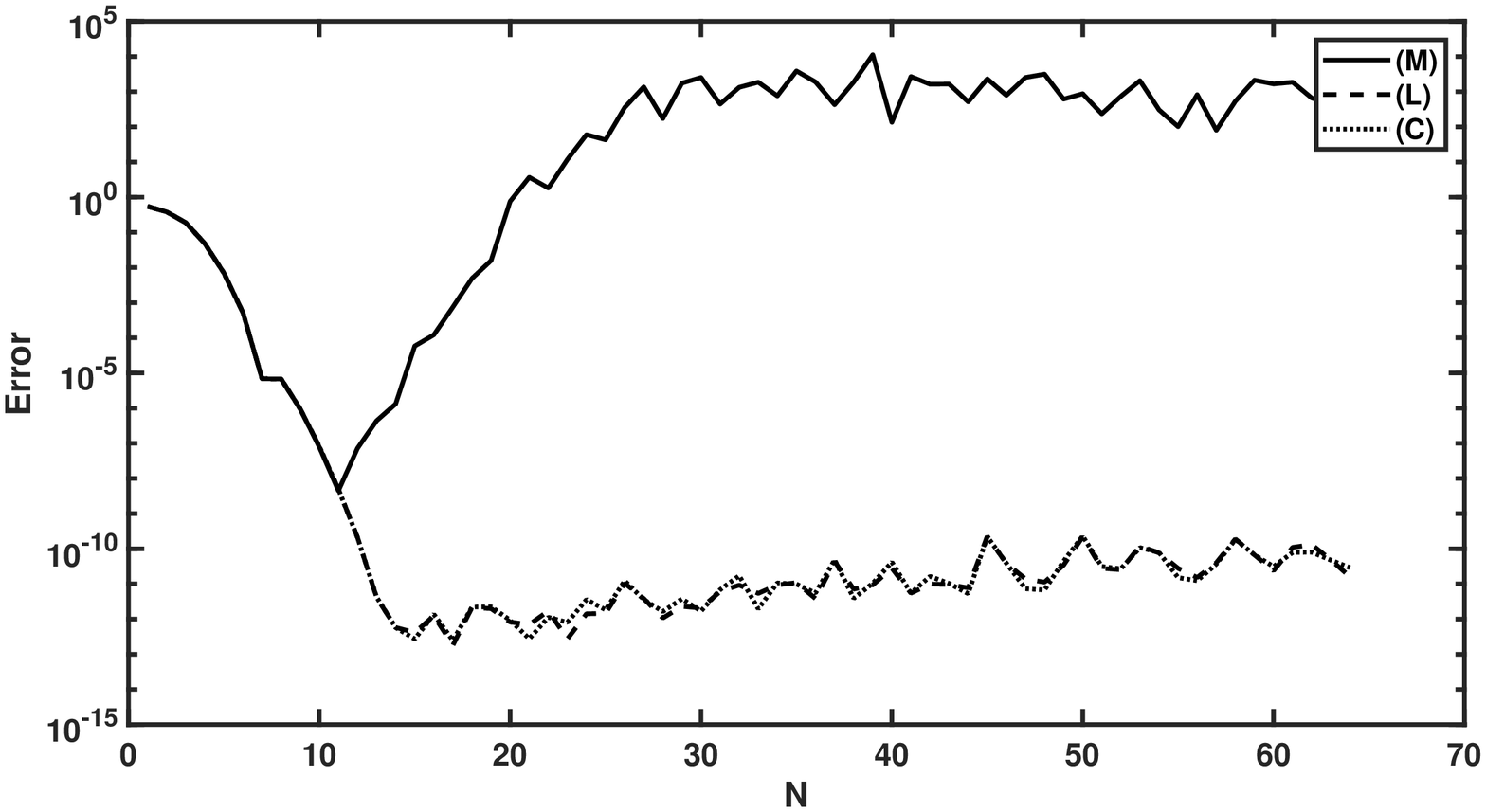} \\
  (1) \\
  \includegraphics[scale=0.4]{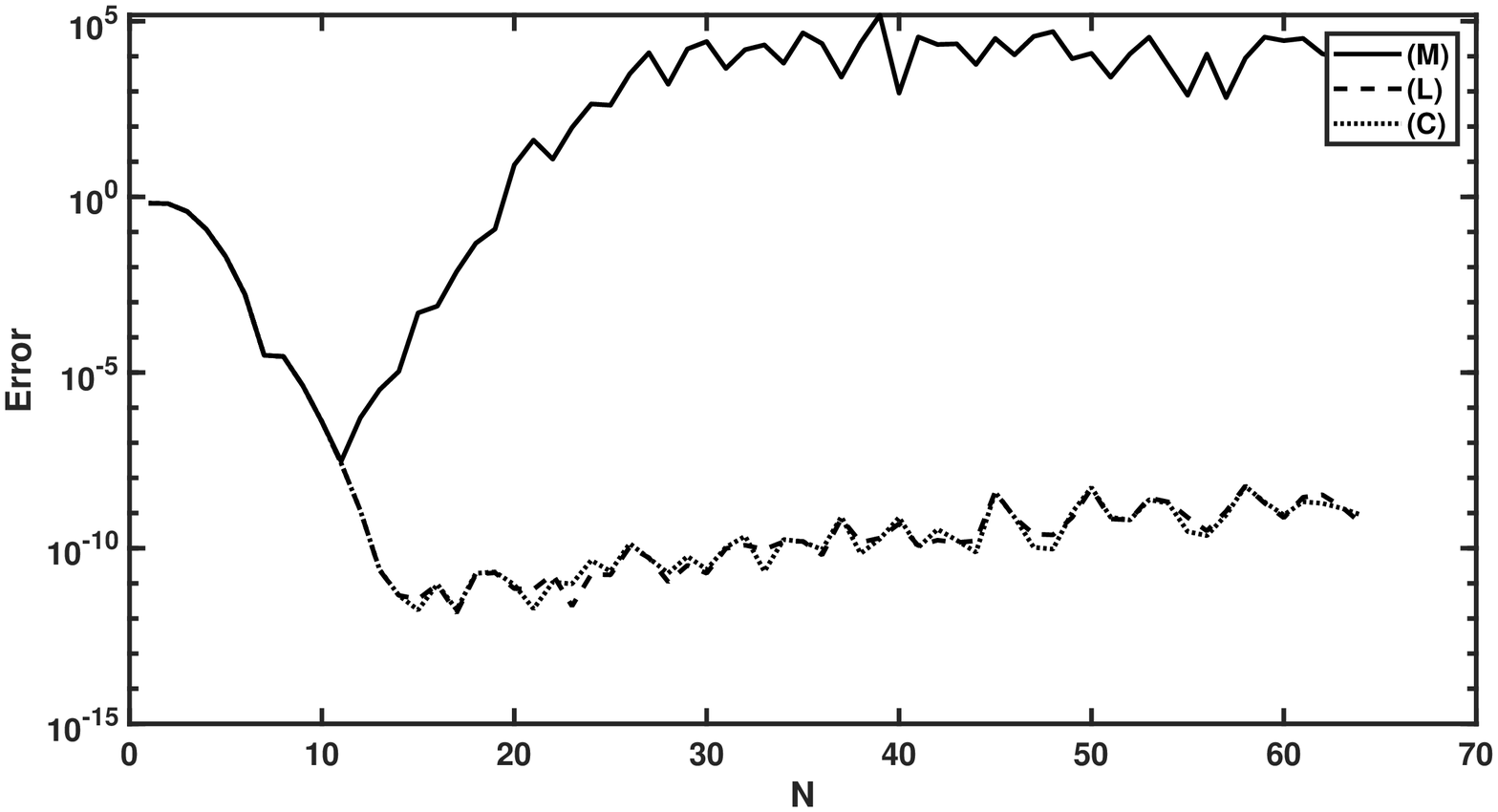} \\
  (2) \\
  \includegraphics[scale=0.4]{ImNew/BasisTestDAE1/Ex1_Linf} \\
  (3)
 \end{tabular}
\caption{Error of the approximate solution in Experiment~\ref{exp:1}.
The abbreviations (M) for the monomial basis, (L) for the Legendre basis, and (C) for the Chebyshev basis are used.
The error is measured in the norms of (1) $L^{2}((0,1),\R^{m})$, (2)
$L^{\infty}((0,1),\R^{m})$, and (3) $H_{D}^{1}((0,1),\R^{m})$.\label{fig:Ex1}}
\end{figure}

\begin{figure}
 \begin{tabular}{c}
  \includegraphics[scale=0.4]{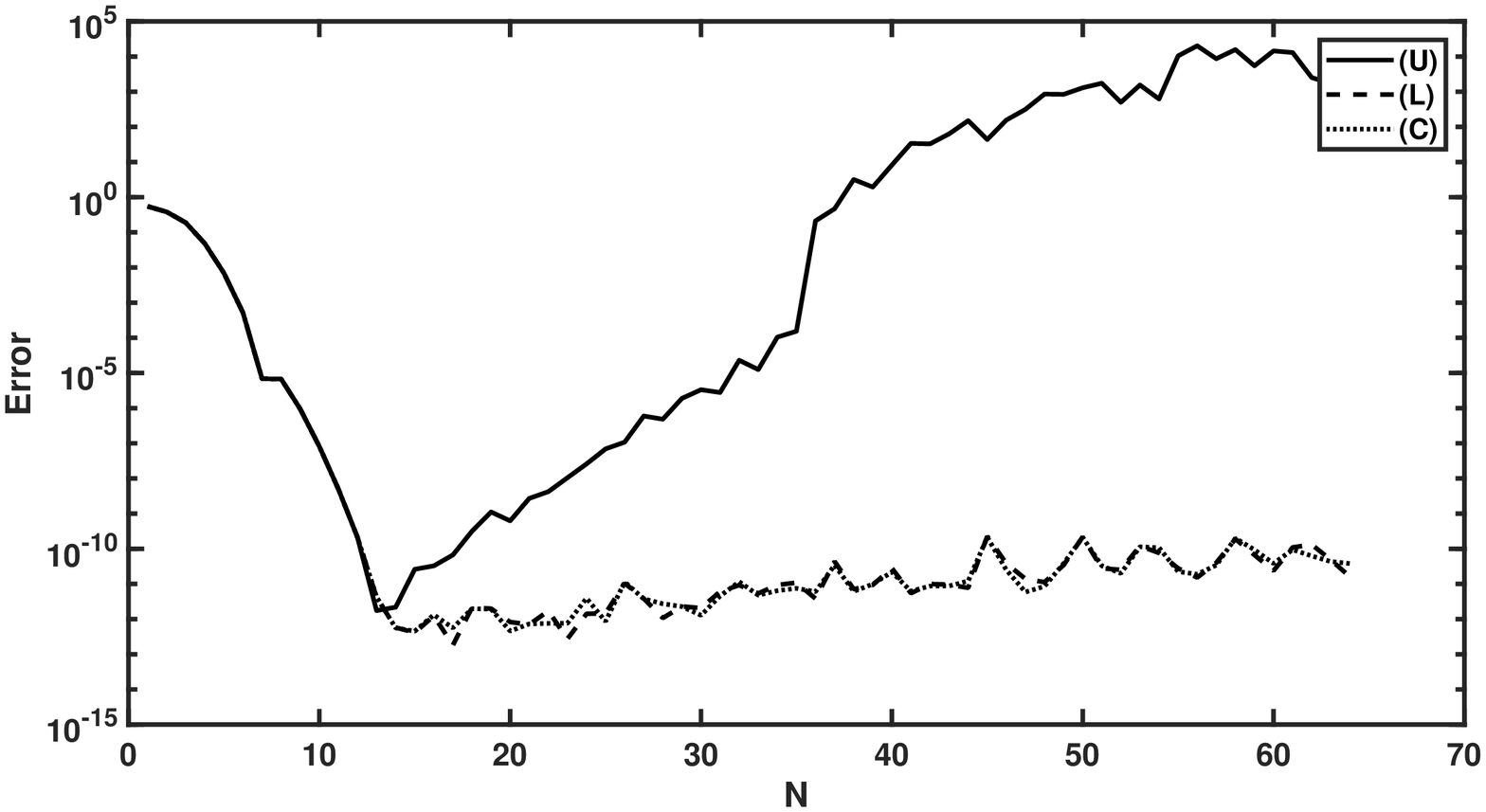} \\
  (1) \\
  \includegraphics[scale=0.4]{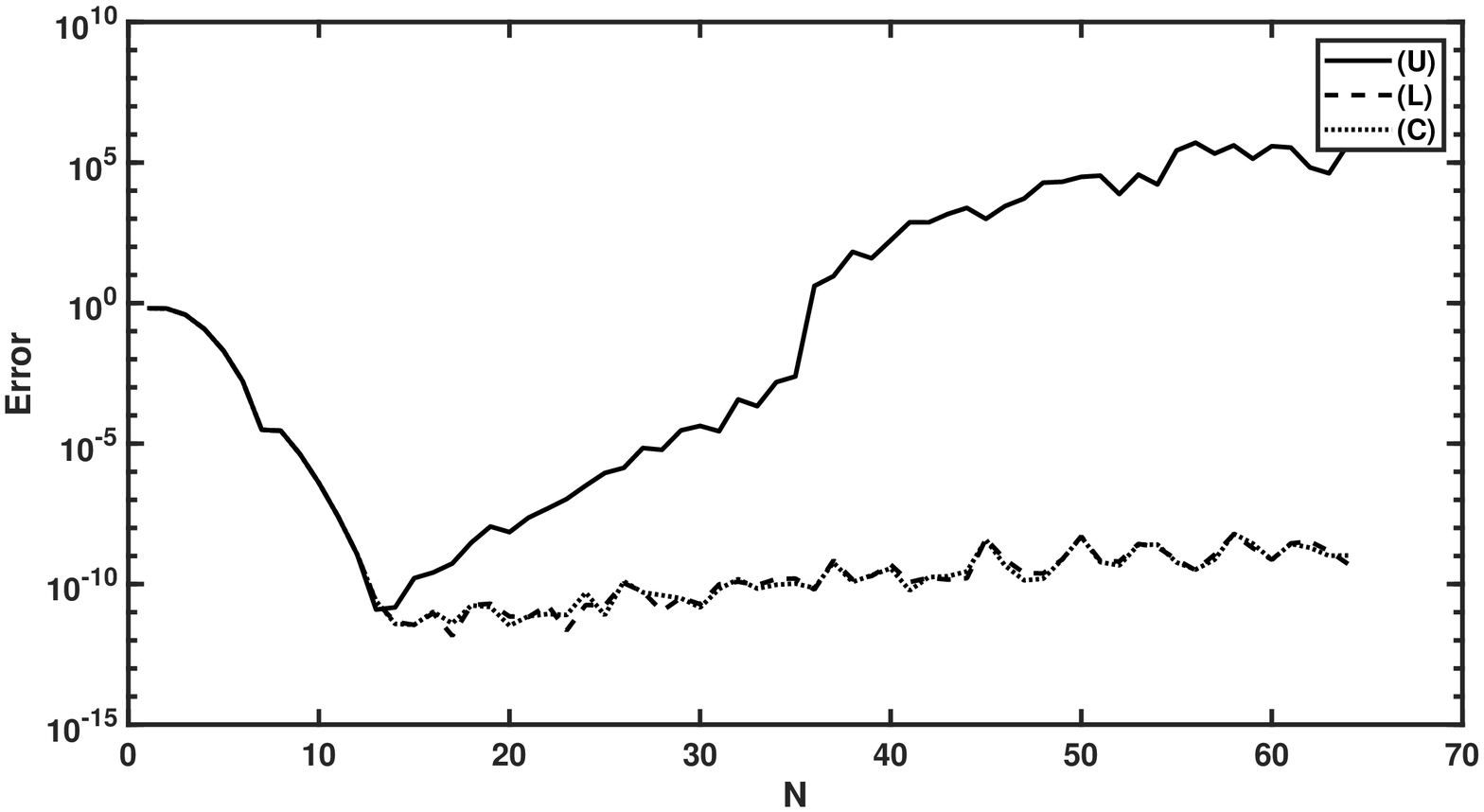} \\
  (2) \\
  \includegraphics[scale=0.4]{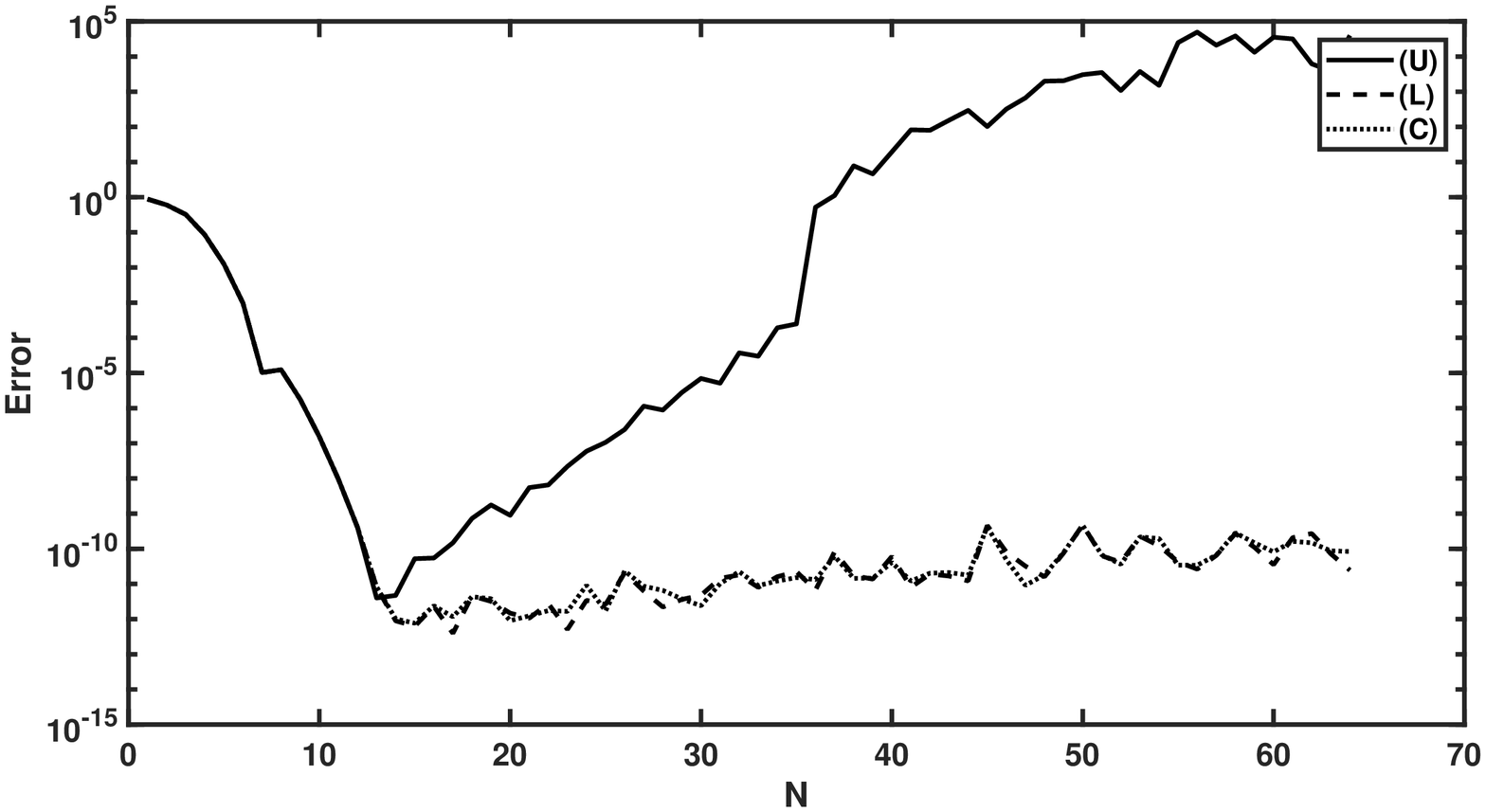} \\
  (3)
 \end{tabular}
\caption{Error of the approximate solution in Experiment~\ref{exp:2}. 
The abbreviations (U) for uniform nodes, (L) for the Gauss-Legendre nodes, and (C) for the Chebyshev nodes are used.
The error is measured in the norms of (1) $L^{2}((0,1),\R^{m})$, (2)
$L^{\infty}((0,1),\R^{m})$, and (3) $H_{D}^{1}((0,1),\R^{m})$.\label{fig:Ex2}}
\end{figure}

\begin{figure}
 \begin{tabular}{c}
  \includegraphics[scale=0.4]{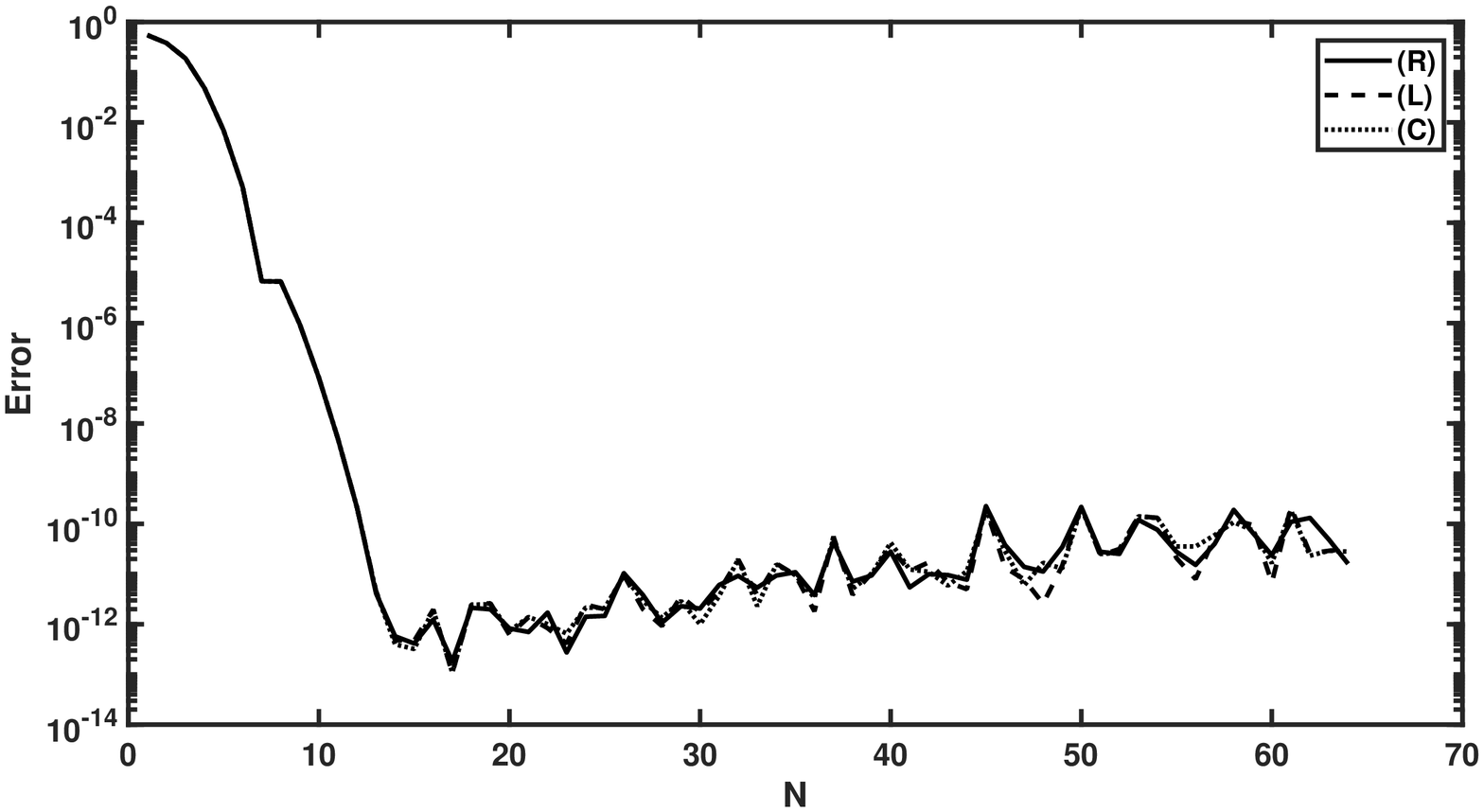} \\
  (1) \\
  \includegraphics[scale=0.4]{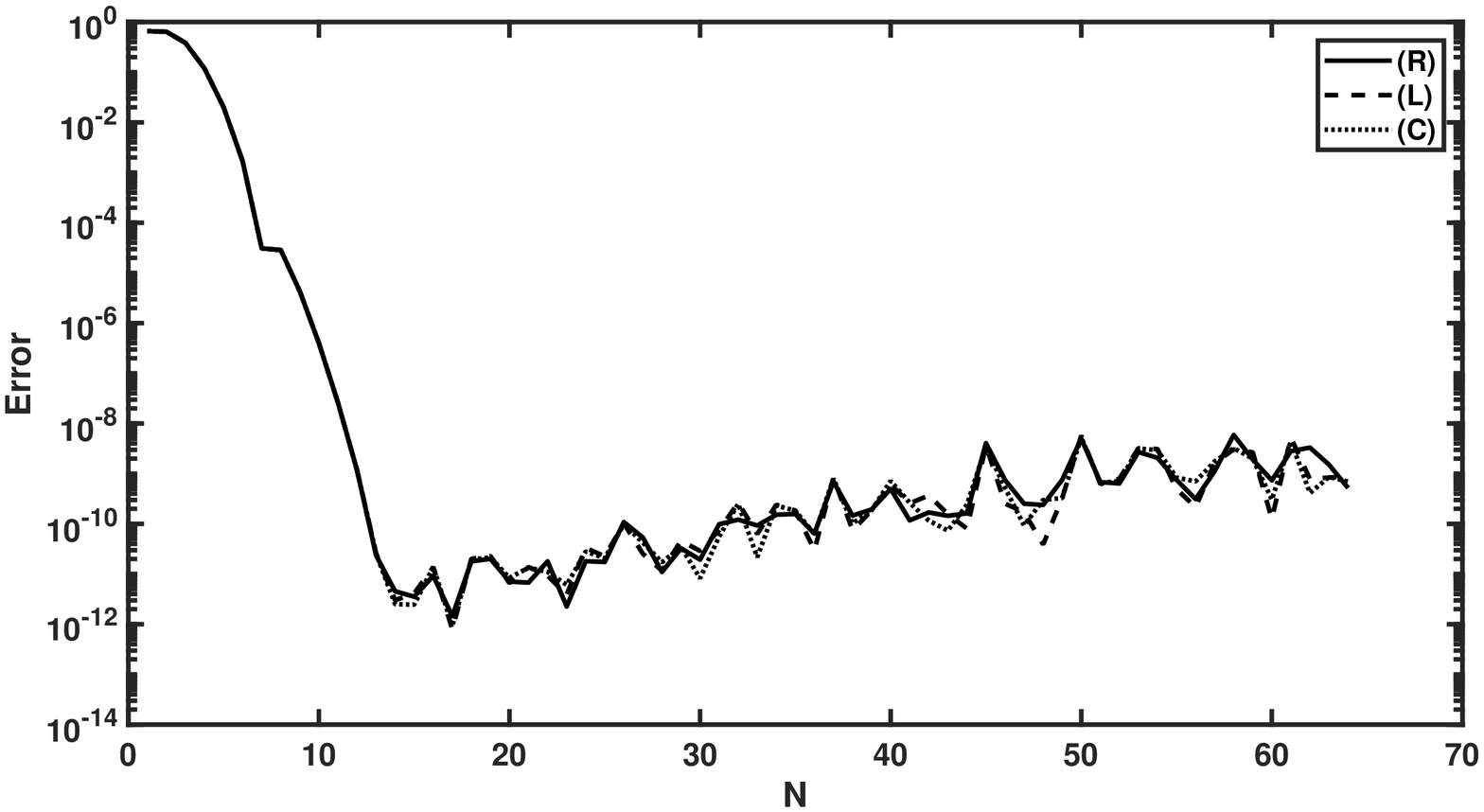} \\
  (2) \\
  \includegraphics[scale=0.4]{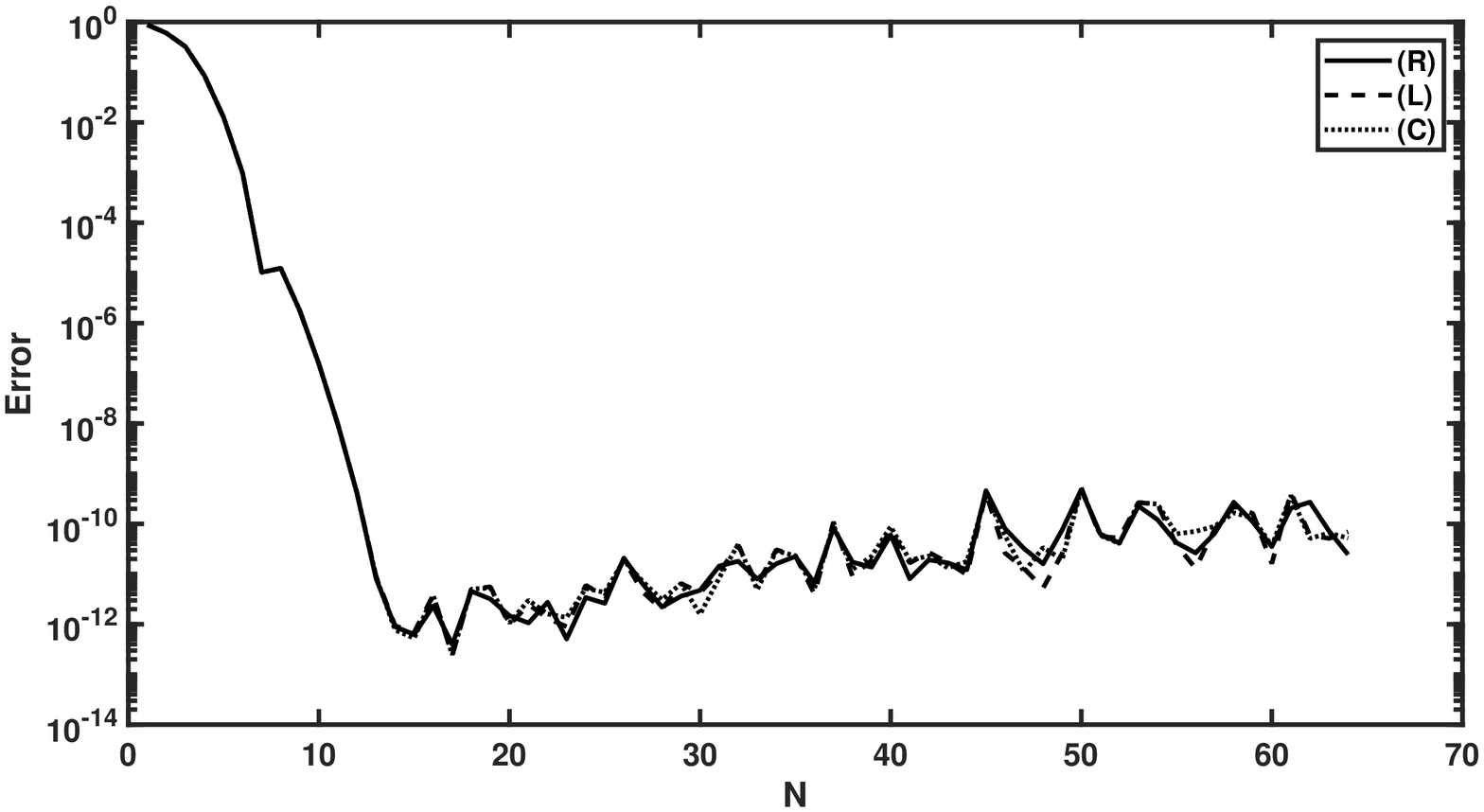} \\
  (3)
 \end{tabular}
\caption{Error of the approximate solution in Experiment~\ref{exp:3}. 
The abbreviations (R) for the Runge-Kutta basis in Legendre representation, (L) for the Legendre basis, and (C) for the Chebyshev basis are used.
The error is measured in the norms of (1) $L^{2}((0,1),\R^{m})$, (2)
$L^{\infty}((0,1),\R^{m})$, and (3) $H_{D}^{1}((0,1),\R^{m})$.\label{fig:Ex3}}
\end{figure}

\begin{figure}
 \begin{tabular}{c}
  \includegraphics[scale=0.4]{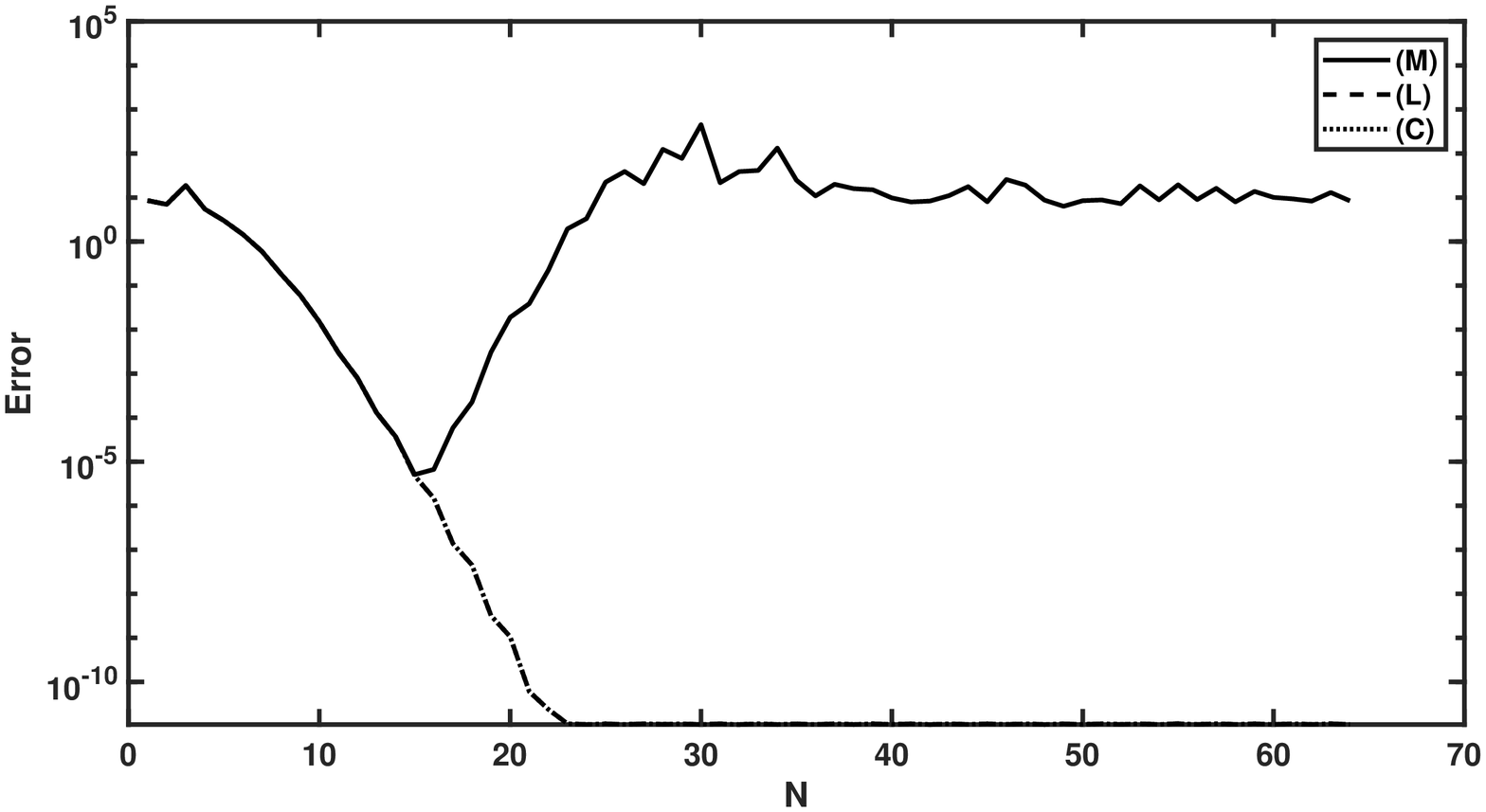} \\
  (1) \\
  \includegraphics[scale=0.4]{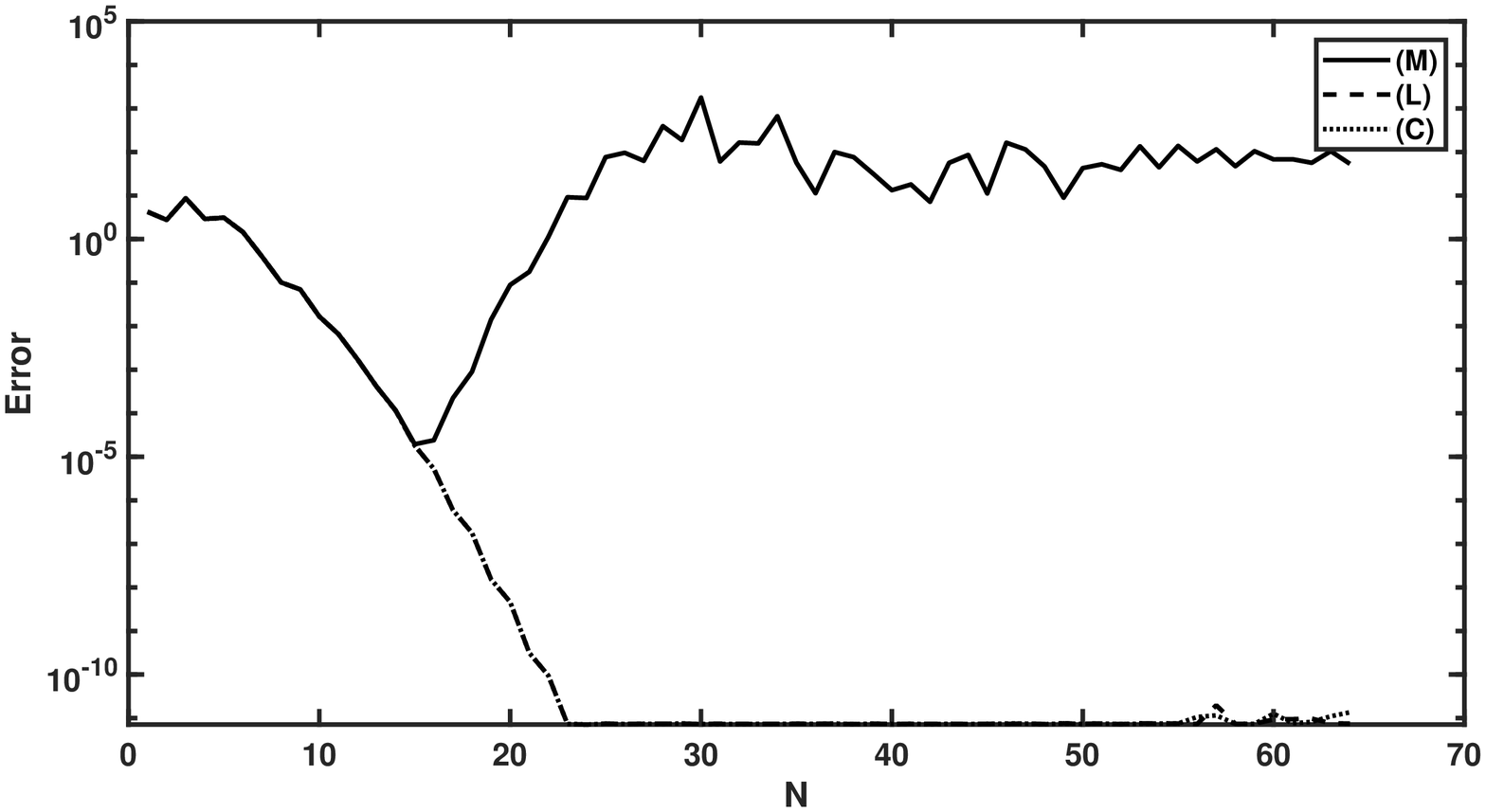} \\
  (2) \\
  \includegraphics[scale=0.4]{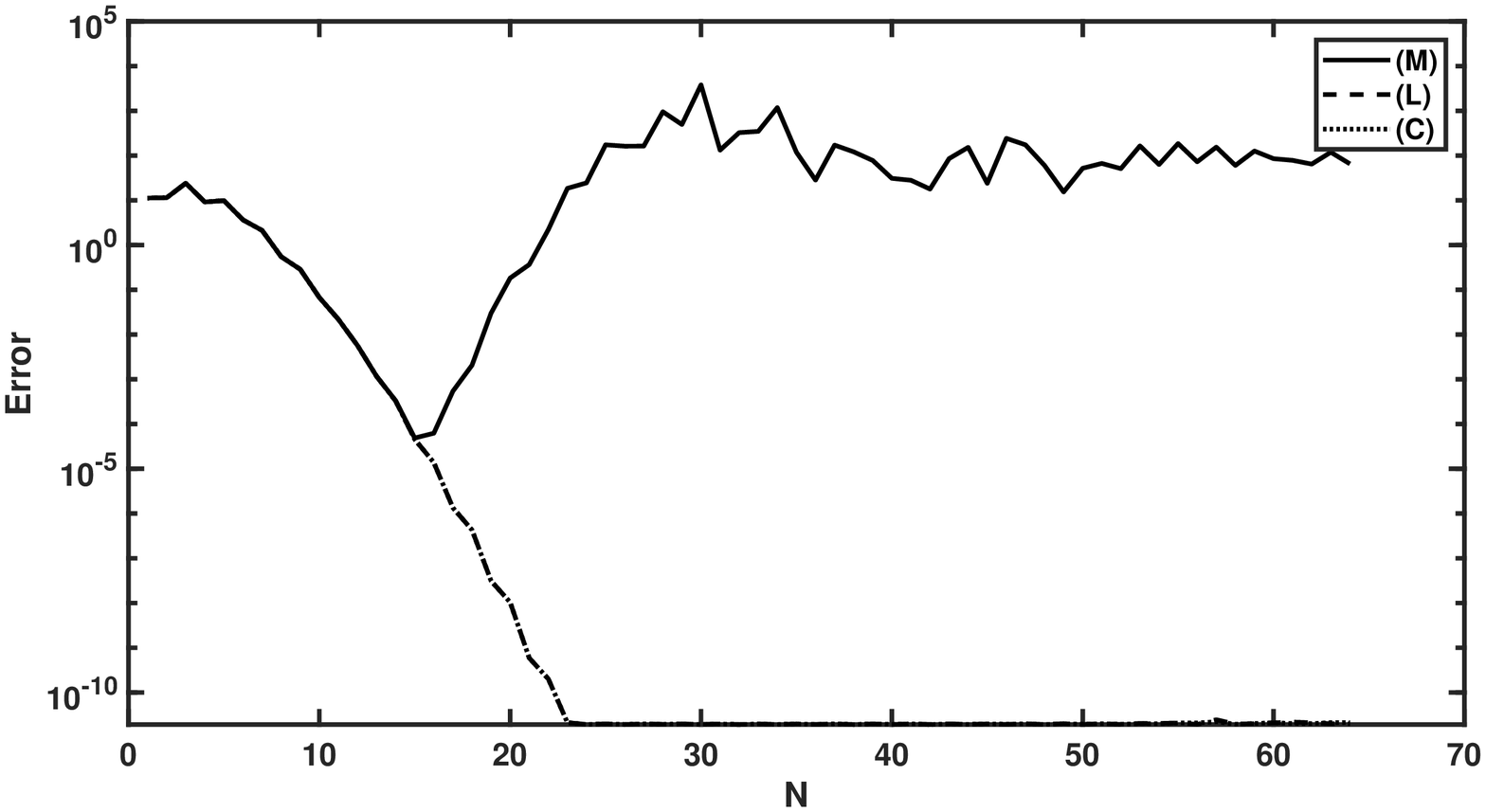} \\
  (3)
 \end{tabular}
\caption{Error of the approximate solution in Experiment~\ref{exp:4}.
The abbreviations (M) for the monomial basis, (L) for the Legendre basis, and (C) for the Chebyshev basis are used.
The error is measured in the norms of (1) $L^{2}((0,1),\R^{m})$, (2)
$L^{\infty}((0,1),\R^{m})$, and (3) $H_{D}^{1}((0,1),\R^{m})$.\label{fig:Ex4}}
\end{figure}

\begin{figure}
 \begin{tabular}{c}
  \includegraphics[scale=0.4]{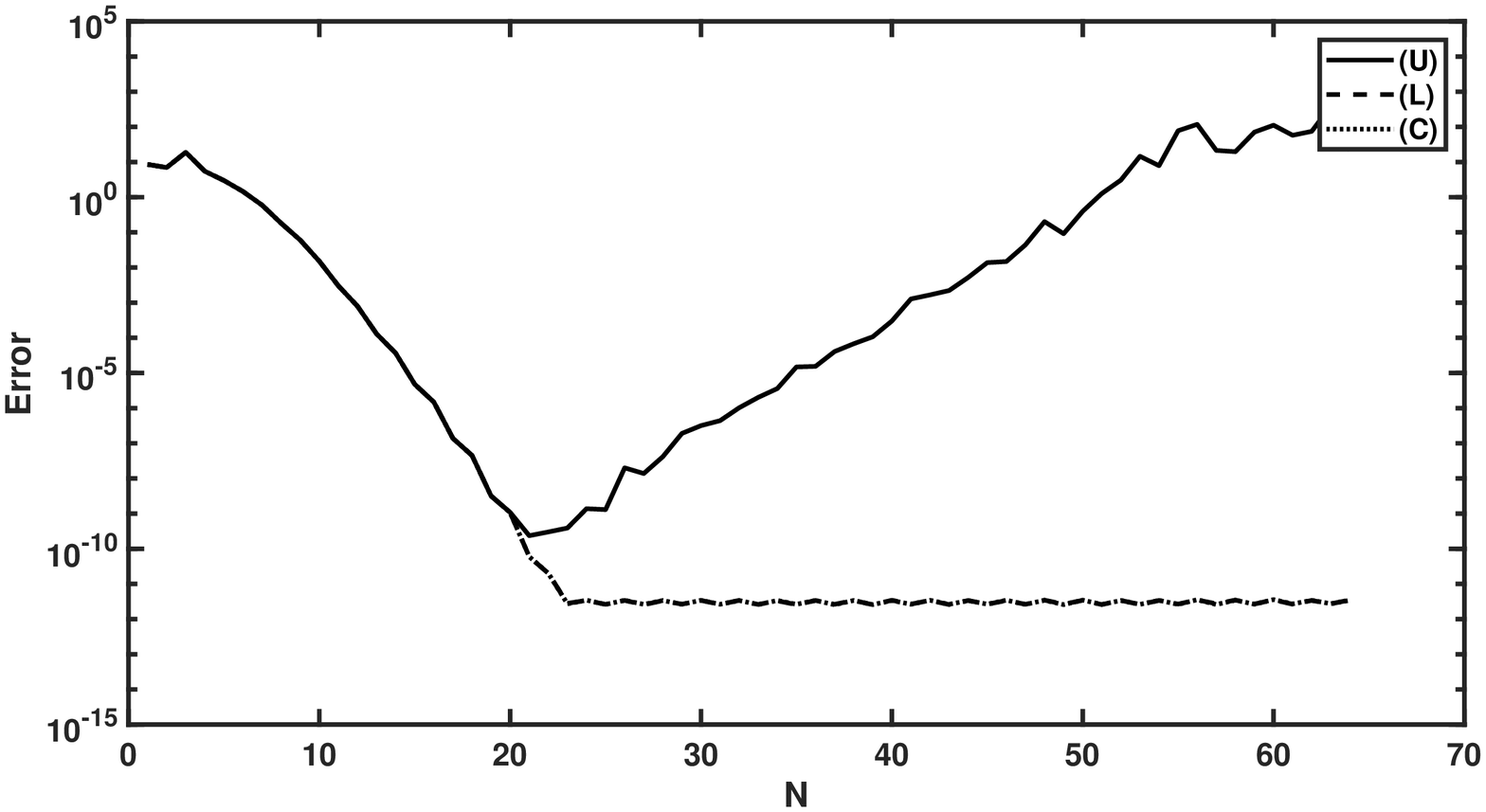} \\
  (1) \\
  \includegraphics[scale=0.4]{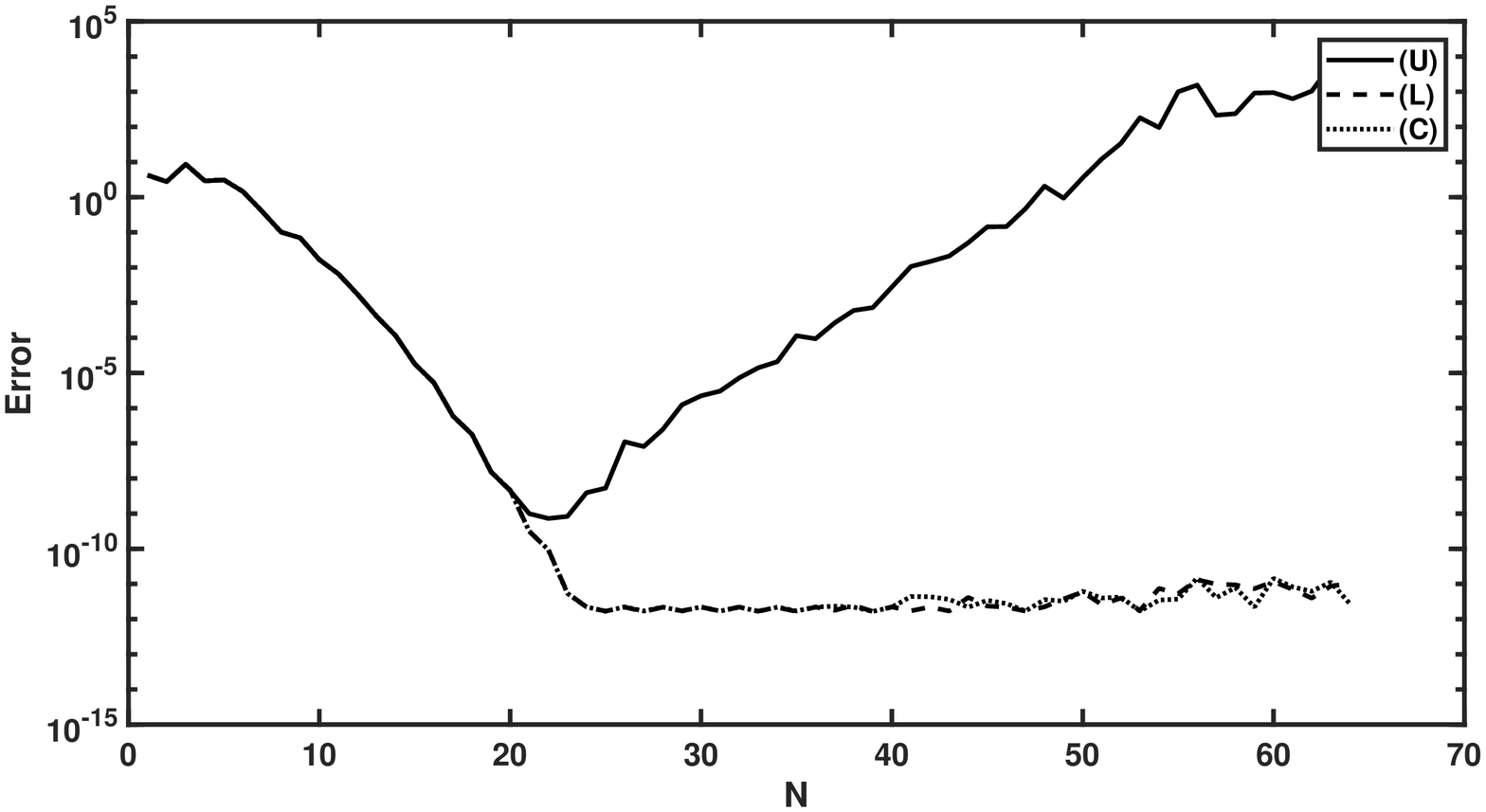} \\
  (2) \\
  \includegraphics[scale=0.4]{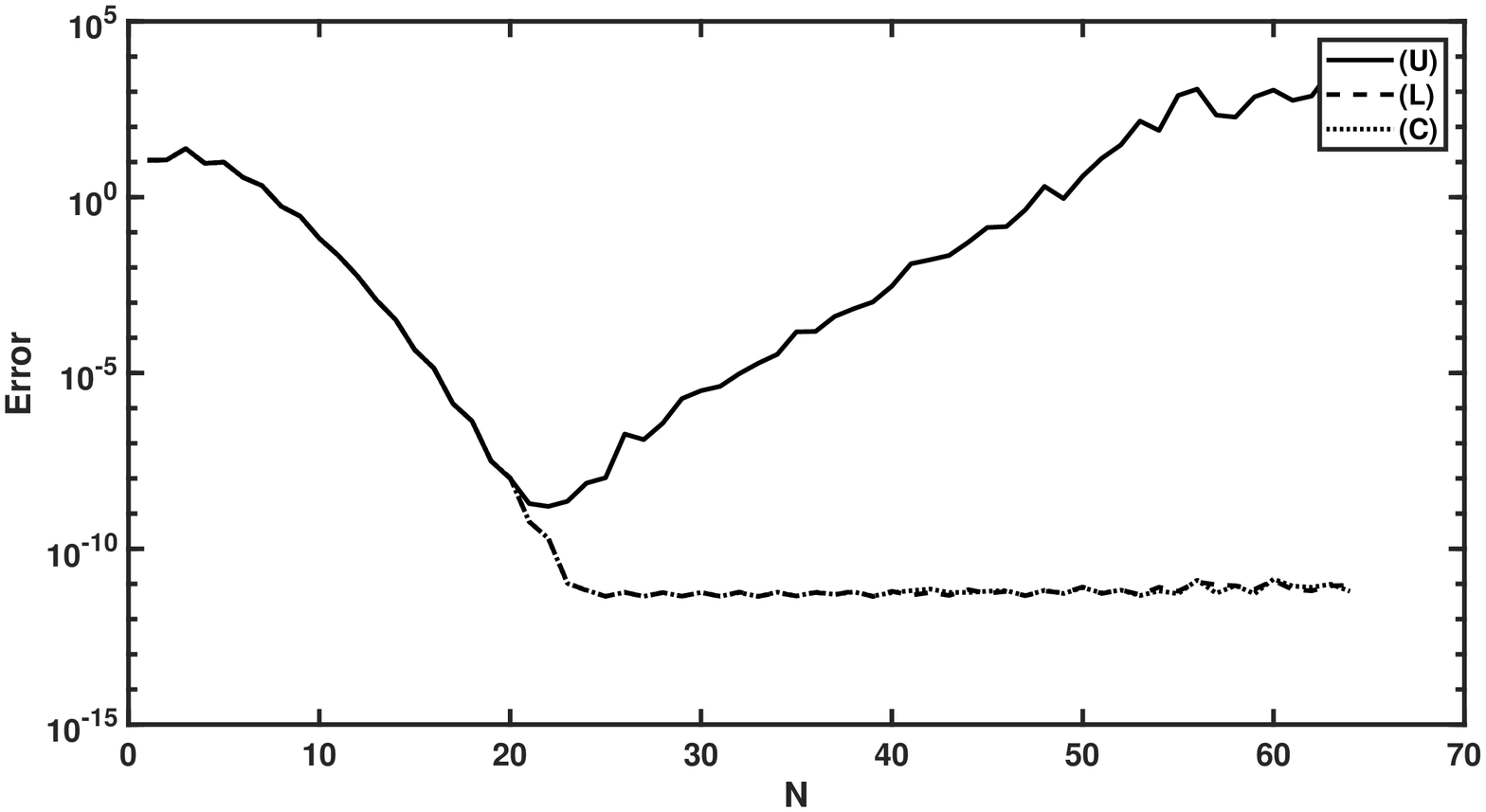} \\
  (3)
 \end{tabular}
\caption{Error of the approximate solution in Experiment~\ref{exp:5}. 
The abbreviations (U) for uniform nodes, (L) for the Gauss-Legendre nodes, and (C) for the Chebyshev nodes are used.
The error is measured in the norms of (1) $L^{2}((0,1),\R^{m})$, (2)
$L^{\infty}((0,1),\R^{m})$, and (3) $H_{D}^{1}((0,1),\R^{m})$.\label{fig:Ex5}}
\end{figure}

\begin{figure}
 \begin{tabular}{c}
  \includegraphics[scale=0.4]{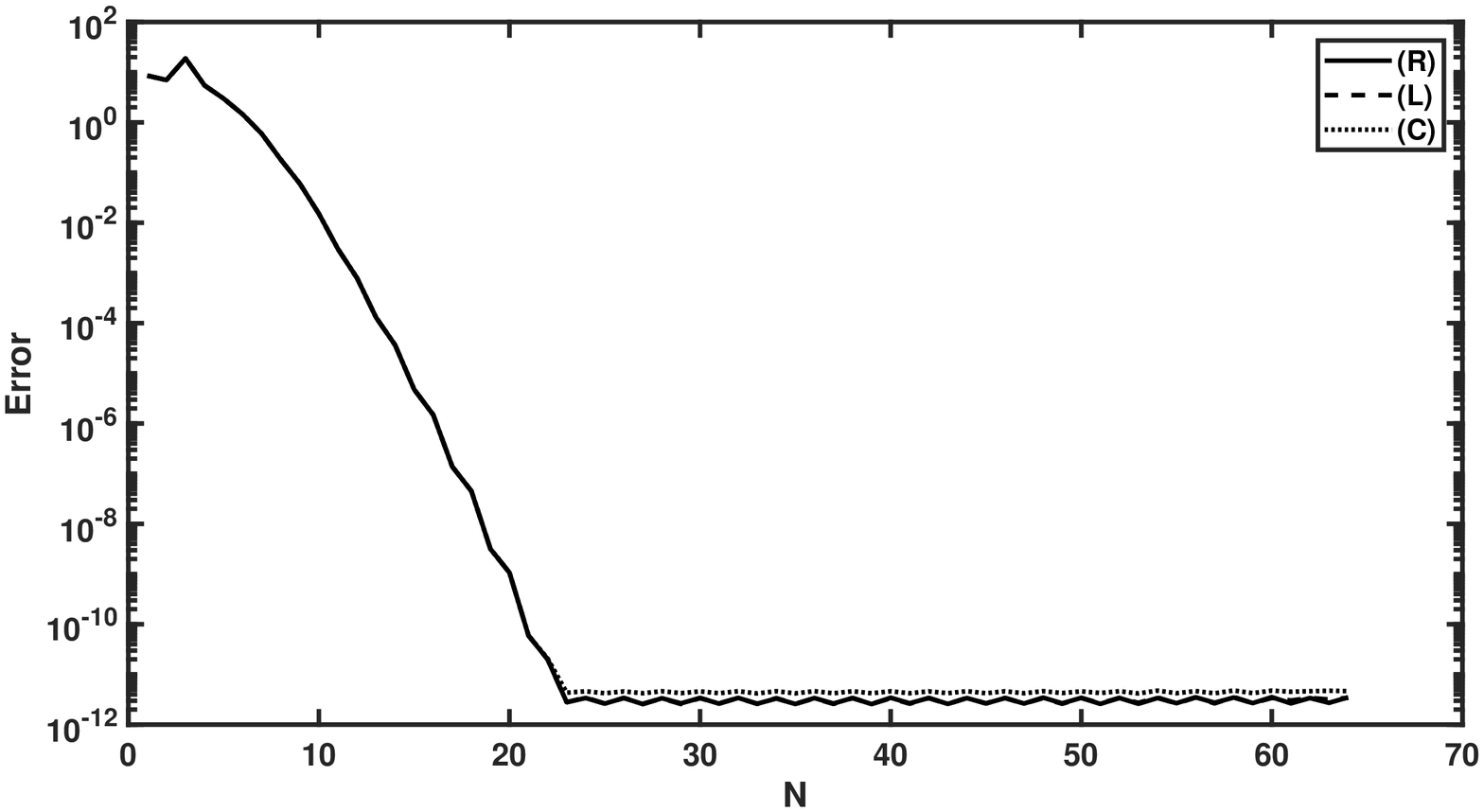} \\
  (1) \\
  \includegraphics[scale=0.4]{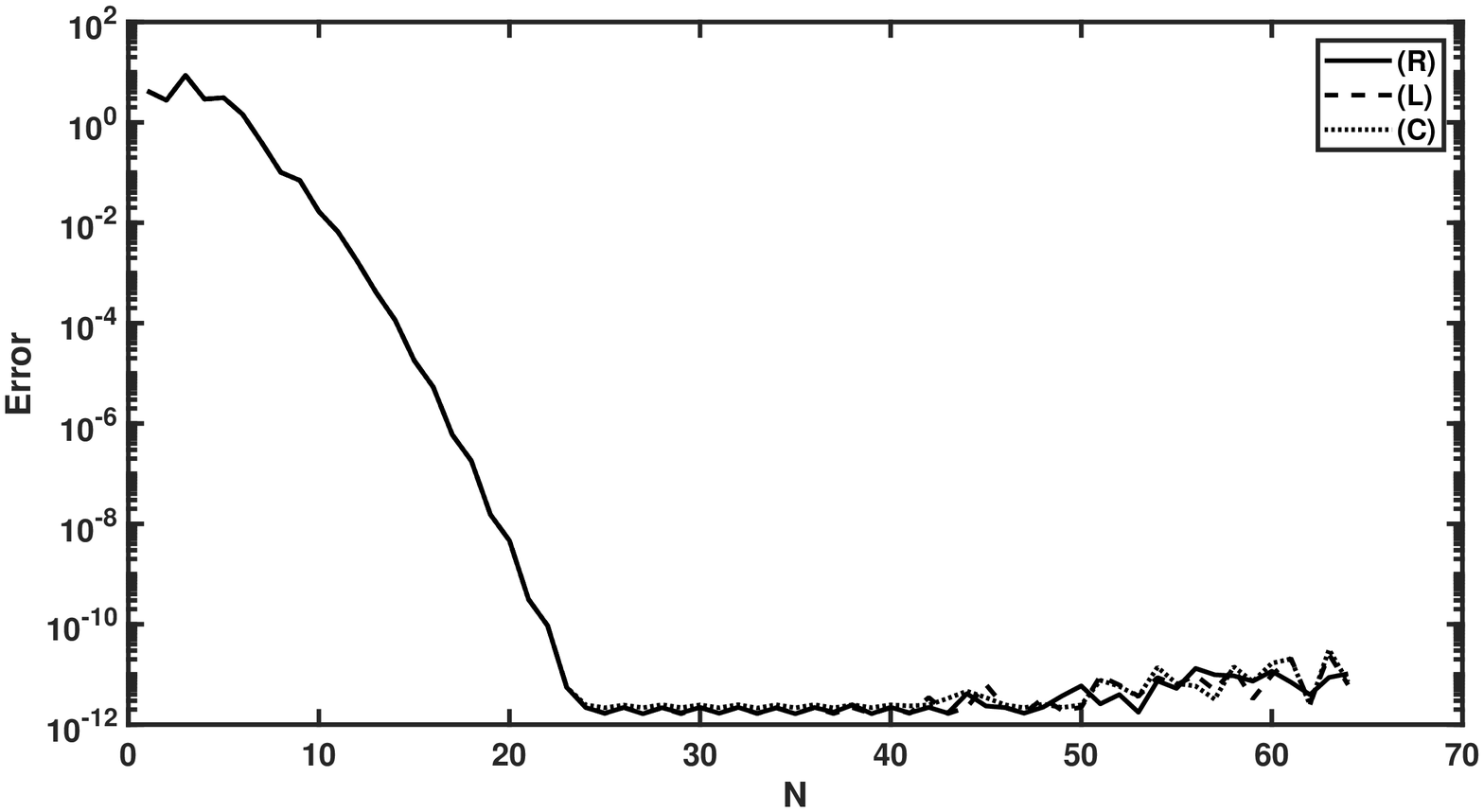} \\
  (2) \\
  \includegraphics[scale=0.4]{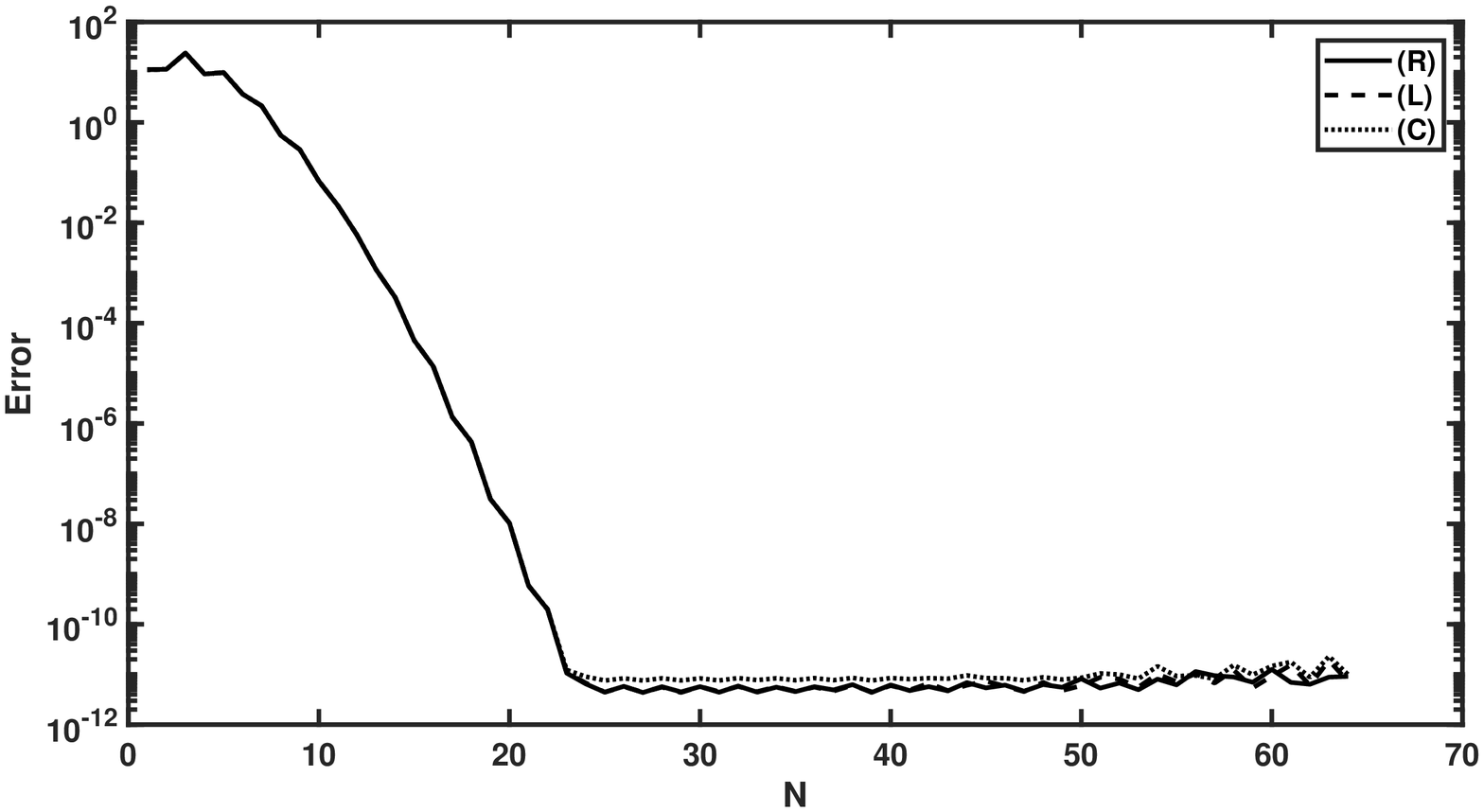} \\
  (3)
 \end{tabular}
\caption{Error of the approximate solution in Experiment~\ref{exp:6}. 
The abbreviations (R) for the Runge-Kutta basis in Legendre representation, (L) for the Legendre basis, and (C) for the Chebyshev basis are used.
The error is measured in the norms of (1) $L^{2}((0,1),\R^{m})$, (2)
$L^{\infty}((0,1),\R^{m})$, and (3) $H_{D}^{1}((0,1),\R^{m})$.\label{fig:Ex6}}
\end{figure}

\clearpage

\begin{sidewaystable}
\vspace{0.5\textwidth}

\caption{Error of the approximate solution using Legendre basis functions and
Gauss-Legendre (G), Radau IIA (R), and Lobatto (L) collocation nodes
when using the functional $\Phi_{\pi,M}^{R}$ in Example \ref{exa:LinCaMo}.
The norm is that of $H_{D}^{1}((0,5),\R^{7})$\label{tab:LinCaMo}}

\begin{center}
\begin{tabular}{c|ccc|ccc|ccc|ccc}
 & \multicolumn{3}{c|}{$N=3$} & \multicolumn{3}{c|}{$N=5$} & \multicolumn{3}{c|}{$N=10$} & \multicolumn{3}{c}{$N=20$}\tabularnewline
$n$ & G & R & L & G & R & L & G & R & L & G & R & L\tabularnewline
\hline 
5 & 5.37e-03 & 5.86e-03 & 5.55e-03 & 1.37e-05 & 1.52e-05 & 1.38e-05 & 3.41e-12 & 4.08e-12 & 3.61e-12 & 8.97e-11 & 5.31e-11 & 1.04e-10\tabularnewline
10 & 2.15e-03 & 2.33e-03 & 2.20e-03 & 1.68e-06 & 1.77e-06 & 1.69e-06 & 3.98e-11 & 2.53e-11 & 2.51e-11 & 4.78e-10 & 7.58e-10 & 1.00e-09\tabularnewline
20 & 9.95e-04 & 1.04e-03 & 1.00e-03 & 2.08e-07 & 2.14e-07 & 2.08e-07 & 2.04e-10 & 2.53e-10 & 1.80e-10 & 3.18e-09 & 2.97e-09 & 3.30e-09\tabularnewline
40 & 4.80e-04 & 4.91e-04 & 4.81e-04 & 2.58e-08 & 2.62e-08 & 2.58e-08 & 1.34e-09 & 1.67e-09 & 1.64e-09 & 2.56e-08 & 2.31e-08 & 3.29e-08\tabularnewline
80 & 2.36e-04 & 2.39e-04 & 2.36e-04 & 3.34e-09 & 3.32e-09 & 3.63e-09 & 1.19e-08 & 1.33e-08 & 1.60e-08 & 1.99e-07 & 2.03e-07 & 2.14e-07\tabularnewline
160 & 1.17e-04 & 1.17e-04 & 1.17e-04 & 9.16e-09 & 9.16e-09 & 1.18e-08 & 8.66e-08 & 1.01e-07 & 1.13e-07 & 1.74e-06 & 1.45e-06 & 1.94e-06\tabularnewline
320 & 5.81e-05 & 5.83e-05 & 5.81e-05 & 8.06e-08 & 6.79e-08 & 8.74e-08 & 7.90e-07 & 8.25e-07 & 9.82e-07 & 1.39e-05 & 1.27e-05 & 1.38e-05\tabularnewline
\end{tabular}
\end{center}
\end{sidewaystable}

\begin{sidewaystable}
\vspace{0.5\textwidth}

\caption{Error of the approximate solution using Legendre basis functions and
Gauss-Legendre (G), Radau IIA (R), and Lobatto (L) collocation nodes
when using the functional $\Phi_{\pi,M}^{C}$ in Example \ref{exa:LinCaMo}.
The norm is that of $H_{D}^{1}((0,5),\R^{7})$\label{tab:LinCaMo_C}}

\begin{center}
\begin{tabular}{c|ccc|ccc|ccc|ccc}
 & \multicolumn{3}{c|}{$N=3$} & \multicolumn{3}{c|}{$N=5$} & \multicolumn{3}{c|}{$N=10$} & \multicolumn{3}{c}{$N=20$}\tabularnewline
$n$ & G & R & L & G & R & L & G & R & L & G & R & L\tabularnewline
\hline 
5 & 5.22e-03 & 7.20e-03 & 7.81e-03 & 1.30e-05 & 1.50e-05 & 1.44e-05 & 2.89e-12 & 4.32e-12 & 1.82e-12 & 5.15e-11 & 3.67e-11 & 4.24e-11\tabularnewline
10 & 2.06e-03 & 2.85e-03 & 3.46e-03 & 1.59e-06 & 1.75e-06 & 1.76e-06 & 3.24e-11 & 1.95e-11 & 1.79e-11 & 3.23e-10 & 1.57e-10 & 1.91e-10\tabularnewline
20 & 9.49e-04 & 1.27e-03 & 1.67e-03 & 1.96e-07 & 2.11e-07 & 2.19e-07 & 2.19e-10 & 1.66e-10 & 1.06e-10 & 2.39e-09 & 1.55e-09 & 7.72e-10\tabularnewline
40 & 4.58e-04 & 6.04e-04 & 8.27e-04 & 2.42e-08 & 2.60e-08 & 2.73e-08 & 1.59e-09 & 1.01e-09 & 7.38e-10 & 1.80e-08 & 1.65e-08 & 4.39e-09\tabularnewline
80 & 2.25e-04 & 2.95e-04 & 4.12e-04 & 3.12e-09 & 3.51e-09 & 3.58e-09 & 1.16e-08 & 8.67e-09 & 5.25e-09 & 1.45e-07 & 1.41e-07 & 2.56e-08\tabularnewline
160 & 1.11e-04 & 1.46e-04 & 2.06e-04 & 9.57e-09 & 1.01e-08 & 8.23e-09 & 9.33e-08 & 7.27e-08 & 3.96e-08 & 1.10e-06 & 1.20e-06 & 1.52e-07\tabularnewline
320 & 5.54e-05 & 7.24e-05 & 1.03e-04 & 7.95e-08 & 8.73e-08 & 7.13e-08 & 7.47e-07 & 5.86e-07 & 3.33e-07 & 8.82e-06 & 9.78e-06 & 1.11e-06\tabularnewline
\end{tabular}
\end{center}
\end{sidewaystable}

\end{document}